\documentclass[draft,reqno,T1]{amsproc}
\usepackage{tikz}
\usetikzlibrary{arrows,automata,calc,shapes,positioning}
\usepackage{color}
\usepackage{amssymb}
\usepackage{amsmath}
\usepackage{amsfonts}
\usepackage[top=2.5cm,bottom=2.5cm,outer=2.2cm,inner=2.2cm,heightrounded,marginparwidth=1cm,marginparsep=.5cm]{geometry}
\usepackage[noadjust]{marginnote}
\usepackage{tkz-berge} % provides drawing of graphs for tikz
\usepackage{thmtools}
\usepackage{proof} % provides \infer
\usepackage{cmll} % provides \parr
\usepackage{centernot} % provides \centernot
\usepackage{wasysym} % provides \pentagon
\usepackage{caption} % allows figure caption with multiple paragraphs

\usepackage[all]{xy}

\definecolor{myurlcolor}{rgb}{0,0,0.4}
\definecolor{mycitecolor}{rgb}{0,0.5,0}
\definecolor{myrefcolor}{rgb}{0.5,0,0}
\usepackage[pagebackref,draft=false]{hyperref}
\hypersetup{colorlinks,
linkcolor=myrefcolor,
citecolor=mycitecolor,
urlcolor=myurlcolor}

\newcommand{\beq}{\[}
\newcommand{\eeq}{\]}
\newcommand{\beqn}{\begin{equation}}
\newcommand{\eeqn}{\end{equation}}
\newcommand{\Z}{\mathbb{Z}}
\newcommand{\N}{\mathbb{N}}
\newcommand{\Npos}{\mathbb{N}_{>0}}
\newcommand{\Q}{\mathbb{Q}}
\newcommand{\Qpos}{\mathbb{Q}_{> 0}}
\newcommand{\Qplus}{\mathbb{Q}_{\geq 0}}
\newcommand{\B}{\mathbb{B}}
\newcommand{\C}{\mathbb{C}}
\newcommand{\R}{\mathbb{R}}
\newcommand{\Rplus}{\mathbb{R}_{\geq 0}}
\newcommand{\Rpos}{\mathbb{R}_{>0}}

\newcommand{\lra}{\longrightarrow}
\newcommand{\eps}{\varepsilon}

\newcommand{\lin}{\mathrm{lin}}

\newcommand{\defin}{:=}
\newcommand{\ann}{\mathrm{ann}} % annihilator
\newcommand{\Rreg}{R^\mathrm{reg}}

\newcommand{\Chem}{\mathtt{Chem}}
\newcommand{\CC}{\mathtt{CommCh}}
\newcommand{\Grph}{\mathtt{Grph}}
\newcommand{\ProbMaj}{\mathtt{ProbMajor}}

\newcommand{\OCM}{\mathtt{OCM}}
\newcommand{\OAG}{\mathtt{OAG}}
\newcommand{\OVS}{\mathtt{OVS}_\Q}
\newcommand{\AOVS}{\mathtt{AOVS}_\Q}

\newcommand{\canc}{\mathrm{canc}}
\newcommand{\oag}{\mathrm{oag}}
\newcommand{\tf}{\mathrm{tf}}
\newcommand{\ovs}{\mathrm{ovs}_\Q}
\newcommand{\aovs}{\mathrm{aovs}_\Q}

\newcommand{\inalph}{\mathcal{A}}
\newcommand{\outalph}{\mathcal{B}}
\newcommand{\inalphh}{\mathcal{C}}
\newcommand{\outalphh}{\mathcal{D}}
\newcommand{\enc}{\mathop{\mathrm{enc}}}
\newcommand{\dec}{\mathop{\mathrm{dec}}}

\newcommand{\lovasz}{\overline{\vartheta}}

\swapnumbers
\theoremstyle{plain}
\newtheorem{thm}{Theorem}[section]
\newtheorem{defn}[thm]{Definition}
\newtheorem{lem}[thm]{Lemma}
\newtheorem{prop}[thm]{Proposition}
\newtheorem{cor}[thm]{Corollary}

\newtheorem{qstn}[thm]{Question}
\newtheorem*{uqstn}{Question}
\theoremstyle{definition}
\newtheorem{ex}[thm]{Example}
\newtheorem{prob}[thm]{Problem}

\theoremstyle{remark}
\newtheorem{rem}[thm]{Remark}

\numberwithin{equation}{section}

\newcommand{\implproof}[2]{\underline{\ref{#1}$\Rightarrow$\ref{#2}:}}

\allowdisplaybreaks

\renewcommand{\emph}[1]{\textbf{#1}}
\newcommand{\emphalt}[1]{\textit{#1}}

\begin{document}
\sloppy

% vertical spacing in multiline equations
\setlength{\jot}{6pt}

%-------------------------------------------------------------------

%%%%%%%%%%%% title page stuff %%%%%%%%%%%%%%%%%%%%%%%%%%

\title{Resource convertibility and ordered commutative monoids}

\author{\medskip}

\address{Tobias Fritz, Perimeter Institute for Theoretical Physics\\
Waterloo ON, Canada}
\email{tfritz@perimeterinstitute.ca}

\date{\today}



\thanks{\textit{Acknowledgements.} I would like to thank Rob Spekkens for numerous discussions over the past two years which shaped my thinking and for very detailed comments on a draft; Corsin Pfister for reminding me of the work of Lieb and Yngvason; Jamie Vicary for asking the most obvious questions that were not obvious to me; Micha{\l} Horodecki for opposition to embezzlement; David Roberson and Laura Man{\v{c}}isnka for discussions on graph theory; John Baez, David Deutsch, Elliot Lieb and Chiara Marletto for valuable feedback; an anonymous referees for important suggestions and for having high standards; and last but not least, all contributors to the \href{http://www.ncatlab.org/nlab/show/HomePage}{nLab} for having created an incredibly useful resource for everything categorical. Research at Perimeter Institute is supported by the Government of Canada through Industry Canada and by the Province of Ontario through the Ministry of Economic Development and Innovation. The author has been supported by the John Templeton Foundation.}

\begin{abstract}
Resources and their use and consumption form a central part of our life. Many branches of science and engineering are concerned with the question of which given resource objects can be converted into which target resource objects. For example, information theory studies the conversion of a noisy communication channel instance into an exchange of information. Inspired by work in quantum information theory, we develop a general mathematical toolbox for this type of question. The convertibility of resources into other ones and the possibility of combining resources is accurately captured by the mathematics of ordered commutative monoids. As an intuitive example, we consider chemistry, where chemical reaction equations such as
\[
\mathrm{2H_2 + O_2} \lra \mathrm{2H_2O}
\]
are concerned both with a convertibility relation ``$\lra$'' and a combination operation ``$+$''. We study ordered commutative monoids from an algebraic and functional-analytic perspective and derive a wealth of results which should have applications to concrete resource theories, such as a formula for rates of conversion. As a running example showing that ordered commutative monoids are also of purely mathematical interest without the resource-theoretic interpretation, we exemplify our results with the ordered commutative monoid of graphs.

While closely related to both Girard's linear logic and to Deutsch's constructor theory, our framework also produces results very reminiscent of the utility theorem of von Neumann and Morgenstern in decision theory and of a theorem of Lieb and Yngvason on the foundations of thermodynamics.

Concerning pure algebra, our observation is that some pieces of algebra can be developed in a context in which equality is not necessarily symmetric, i.e.~in which the equality relation is replaced by an ordering relation. For example, notions like cancellativity or torsion-freeness are still sensible and very natural concepts in our ordered setting.
\end{abstract}

\vspace*{-0mm}

\maketitle

\tableofcontents

\newpage
\section{Introduction}
\label{introduction}

In our industrialized world, we deal with resources and their management on a daily basis. This applies to things that are commonly considered resources, such as energy, water, and certain minerals like coal or iron ore. But also more abstract quantities such as time or information can be considered resources. The main characteristic of a resource is that it can be used or consumed, either in order to produce some desirable commodity, or in order to produce \emphalt{other resources} for producing desirable commodities, such as machine equipment.

The goal of our present work is part of an ongoing project to develop a general mathematical theory of resources and their convertibility, with intended applications to the natural sciences, engineering, and eventually economic theory. This has been started in~\cite{resourcesI} and will be continued here. Before diving into the details, we now describe the main ingredients of our approach and where they originate. Other already existing approaches will be discussed in Section~\ref{sectcompare} and compared to ours.

For our purposes, we make no distinction between desired target commodities and the resource objects that are used or consumed in generating them. More concretely, we also treat any target commodity itself as a resource object. The advantage of this is that it allows for a uniform mathematical formalism in which ``everything is a resource object'', where ``everything'' refers to every entity in the context under consideration. In real-world applications, this may also include things like waste or pollution. These are ``resource'' objects to which one would naturally ascribe a negative value. In the language of economics, most \emphalt{negative externalities} are resource objects of this form. In some cases, it may not even be clear whether a given resource object has a positive or negative value. For example, a chemical such as lead (Pb) may be very useful for the construction of batteries, but at the same time may be a serious annoyance due to its toxicity. For this reason, it is among the basic tenets of our framework that resource objects should \emphalt{not} be assigned a unique number which measures their value. Rather, the utility of a resource is determined solely through its interaction with other resource objects. This \emphalt{structuralist} philosophy will be familiar to anyone who has studied some category theory.

As far as this work is concerned, the main characteristic of resource objects is their \emphalt{convertibility}, i.e.~the potential of resource objects to be turned into other resource objects. Due to the inherent circularity in this statement, this cannot be understood as a definition of what constitutes a resource object. Rather, it serves as a guiding principle for the mathematical formalism that we develop. This formalism is intended to be a toolbox for asking and answering questions of the following form:\medskip
\begin{enumerate}
\item\label{qa} Under which conditions can a resource object $x$ be converted into a resource object $y$?
\item Can the use of a third resource object $z$ as a catalyst help in achieving the conversion of $x$ into $y$?
\item\label{qc} If one tries to convert many copies of $x$ into many copies of $y$, then how many copies of $x$ does one need on average in order to produce one copy of $y$?
\end{enumerate}\medskip
As we argued in~\cite{resourcesI}, these kinds of questions belong to the \emphalt{pragmatic} approach to science, since they are of an engineering-type nature. So our work is not intended to be of relevance to e.g.~fundamental physics, but we do not rule out such applications.

Questions~\ref{qa}--\ref{qc} have already been studied extensively in certain particular contexts, and this is where we draw some of our motivation and ideas from. Especially in quantum information theory, the ``resource'' aspect of certain quantum states has been widely investigated, and answers to the above questions have been sought. For example, quantum entanglement has long been thought of as a valuable resource for quantum information processing~\cite{pureent,partialent,quantuminformation,entanglement}. The conversion of entanglement as a resource into information processing as a target commodity happens through a multitude of protocols such as quantum teleportation, dense coding etc. In recent years, the study of \emphalt{resource theories} other than entanglement has become a subdiscipline of quantum information theory~\cite{thermocost,nonuniformity,asymmetry,secondlaws}. The term ``resource theory'' refers to a particular context in which one investigates a collection of resource objects with a given convertibility relation.

Some of these resource theories have close connections to thermodynamics and its second law. In fact, resource-theoretic considerations have been central to thermodynamics since Carnot's work on heat engines, which was concerned precisely with questions of the above form: how much work can be extracted from two bodies of different temperatures? So it is no surprise that definitions and results closely related to ours can also be found in the approach of Lieb and Yngvason to thermodynamics and thermodynamic entropy~\cite{secondlaw1}\footnote{See~\cite{secondlaw2,freshlook,secondlaw3} for further developments and exposition, and~\cite{entropyprinciple} for a textbook account.}. We hope that our toolbox will be useful for the further development of Lieb--Yngvason thermodynamics. For example, our Theorem~\ref{numocm} is closely related to the main theorem of~\cite{secondlaw1} on the characterization of thermodynamic entropy---and surprisingly, also to the ``utility theorem'' of von Neumann and Morgenstern in the foundations of economics~\cite{games}. Besides the study of resource theories in quantum information theory, these can be considered to be the main precursors to our work. Lieb and Yngvason already anticipated the possibility that mathematical structures similar to theirs may have wide applicability in the sciences~\cite[p.~2]{secondlaw3}. 

\begin{figure}
\centering
\xymatrix{ 
& & & & & \textrm{Section~\ref{sectfun}} \\
& & \textrm{Section~\ref{sectoag}} \ar[r] & \textrm{Section~\ref{sectovs}} \ar[r] & \textrm{Section~\ref{sectaovs}} \ar[ru] \ar[dr] \ar[r] & \textrm{Section~\ref{sectrates}} \ar@{-->}[d] \\
 \textrm{Section~\ref{prelims}} \ar[r] & \textrm{Section~\ref{sectocm}} \ar[ur] \ar[dr] & & & & \textrm{Section~\ref{sectonedim}} \\
 & & \textrm{Section~\ref{sectcompare}} }
\caption{Section dependence.}
\label{sectdepend}
\end{figure}
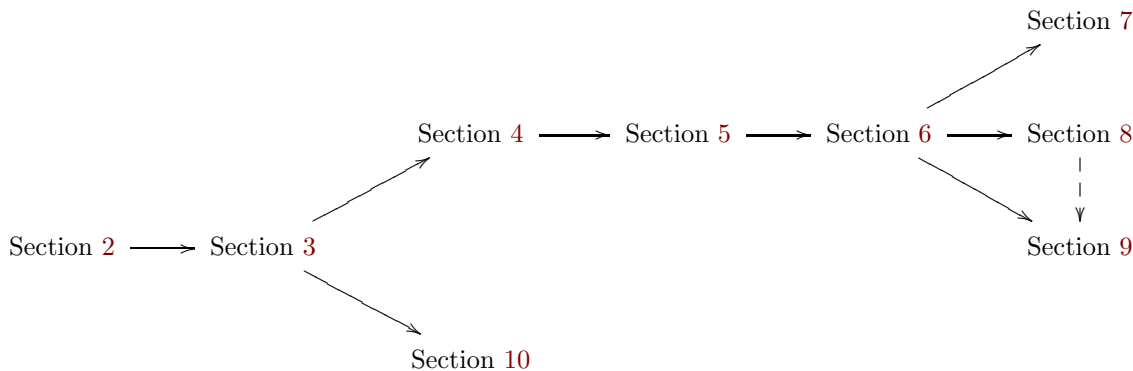

\bigskip

\subsection*{How to read this paper}

Except for Appendix~\ref{enrichment}, this paper is self-contained. Almost all our references are pointers to further reading or have been included for proper attribution. Since this paper is long and contains a substantial amount of material, we now try to assist the reader in deciding which parts to read. Of course, this depends on the reader's background and interests.\medskip

\begin{itemize}
\item Sections~\ref{prelims} to~\ref{sectovs} are mathematically rather basic, and our main points in these sections are the resource-theoretic interpretations. The mathematics will be close to obvious to anyone familiar with universal constructions in algebra or with the different kinds of symmetric monoidal categories (Appendix~\ref{enrichment}), and these readers may proceed through these sections rather quickly.
\item On the other hand, Sections~\ref{sectaovs} to~\ref{sectonedim} are mathematically more technical, and contain less new material on the resource-theoretic interpretation. This is where most of our hard results are to be found, which we hope to be useful in concrete applications. The kind of mathematics used here is functional analysis, and readers with a good background on locally convex spaces will encounter familiar material.
\item Readers with a background in linear logic may start with Section~\ref{linlog}, possibly after a cursory read of Section~\ref{sectocm}.
\item Readers who understand symmetric monoidal categories well may want to refer to Appendix~\ref{enrichment} rather early on and refer back to it repeatedly.
\end{itemize}\medskip

The following detailed outline will help readers get a better idea of which parts of the paper are of interest to them.

\bigskip

\subsection*{Summary and main results}

The purpose of this paper is to construct a mathematical framework in which questions of the form~\ref{qa}--\ref{qc} can be asked, and to develop some general tools for answering them.

In Section~\ref{prelims}, we start gently by considering the mathematics of ordered sets. These formalize the convertibility relation between resource objects: if $x$ can be converted into $y$ and $y$ can be converted into $z$, then also $x$ can be converted into $z$; moreover, any $x$ can be converted into itself. Mathematically, this constitutes the definition of a preorder. Since any two mutually interconvertible resource objects might as well be considered to be ``the same'' resource object, we argue that there is no loss of generality in working with partial orders instead of preorders, which is what we do from then on. We discuss some questions of interpretation and present the resource theory of chemistry as our first example, where the resource objects are collections of molecules or atoms such as
\beq
\mathrm{CH_4}, \qquad \mathrm{C_{60} + HCl + C_6 H_6}, \qquad \mathrm{2H_2O_2},\qquad \mathrm{4 N_2 + O_2}, \qquad \ldots
\eeq
There is a notion of convertibility of two such resource objects generated by chemical reactions, such as
\beq
\mathrm{CH_4 + 2O_2} \lra \mathrm{CO_2 + 2H_2 O}, \qquad \mathrm{Zn + 2HCl} \lra \mathrm{ZnCl_2 + H_2}, \qquad \ldots
\eeq
Then we introduce ordered maps as morphisms of ordered sets. Resource monotones which measure the value of every object in terms of a real number are ordered maps with values in $\R$. We explain that $x\geq y$ in an ordered set $A$ if and only if $f(x)\geq f(y)$ for all such resource monotones $f:A\to\R$. A resource monotone $f$ is a conservation law if $-f$ is a resource monotone as well.

In Section~\ref{sectocm}, we go further by also considering the other main structure of resource theories besides the convertibility relation, and this is the combination operation: for any two resource objects $x$ and $y$, there should exist a resource object denoted $x+y$ which describes the situation of having both $x$ and $y$ together, and moreover there should exist a resource object denoted $0$ which represents the vacuous resource object or ``nothing''. This leads to our Definition~\ref{ocm}, which introduces the notion of ordered commutative monoid. The idea is that in any resource theory, the convertibility relation and combination operation should equip the collection of resource objects with the structure of an ordered commutative monoid. Mathematically, ordered commutative monoids are the central concept of this paper. Our third example is the resource theory of communication, as developed by Shannon in his foundational work which established information theory. It has communication channels---mathematically formalized by stochastic matrices---as resource objects. These can be converted into each other by pre- and post-processing using representing encoding and decoding operations. Two such communication channels can be combined by using them in parallel. In this way, the resource theory of communication is described by an ordered commutative monoid denoted $\CC$. Unfortunately, we do not yet know how to deal with the analytic details in Shannon's theory having to do with allowing small errors of communication in order to optimize throughput. Therefore we consider $\CC$ mostly because of its close relation with $\Grph$ in terms of zero-error communication, and also to give some idea of what the challenges for improving our present approach are. We hope to achieve a comprehensive resource-theoretic treatment of Shannon's theorems in future work.

More examples of resource theories can be found in~\cite{resourcesI}. Specific examples will need to be developed in detail in dedicated work for each application separately, and so we keep the selection of examples rather sparse in this paper.

The section continues with the definition of homomorphisms and isomorphisms of ordered commutative monoids. Definition~\ref{deffunc} introduces functionals on an ordered commutative monoid as homomorphisms with values in $\R$; in resource-theoretic terms, these are the resource monotones which respect the combination operation. We then pose the main question of the present work:

\begin{restatable*}{qstn}{mainqstn}
\label{mainqstn}
If $f(x) \geq f(y)$ for all functionals $f$, what does this tell us about the ordering of $x$ relative to $y$?
\end{restatable*}

In resource-theoretic terminology, the same question reads like this:

\begin{uqstn}
If $f(x) \geq f(y)$ for all additive resource monotones $f$, what does this tell us about the convertibility of $x$ into $y$?
\end{uqstn}

We use this as a guiding question for the rest of the paper---not only because it seems like a natural question, but also because the methods that we develop to answer it should be interesting and useful in their own right. Definition~\ref{generatingpair} then introduces the notion of generating pair as a technical property that a given ordered commutative monoid may or may not possess. We hope that generating pairs will exist in many cases of interest, and many of our subsequent results assume the existence of a generating pair.

Section~\ref{sectoag} explores ordered abelian groups as a particularly nice and tractable class of ordered commutative monoids. We explain how to every commutative monoid $A$, one can associate an ordered abelian group $\oag(A)$ together with a homomorphism $A\to\oag(A)$ having the property that every homomorphism $f:A\to G$ to another ordered abelian group $G$ factors uniquely through $\oag(A)$, resulting in Theorem~\ref{ocmtooag}. Since $\R$ is an ordered abelian group, this applies in particular with $G=\R$, and hence gives a partial answer to Question~\ref{mainqstn}. The construction of $\oag(A)$ has an appealing resource-theoretic interpretation in terms of catalysis: $\oag(A)$ describes the same resource theory as $A$, up to two important differences. First, the convertibility relation is replaced by ``catalytic convertibility'', in which arbitrary catalysts are allowed to facilitate conversions. Second, for every resource object $x$ there is a new resource object $-x$ such that $x + (-x) = 0$, i.e.~such that $-x$ stands for ``owing'' a copy of $x$ to a ``bank''. In fact, this second modification necessitates the first, since then we can borrow any desired catalyst from a bank and return it after use. The construction of $\oag(A)$ generalizes the standard construction of the Grothendieck group associated to a commutative monoid.

Section~\ref{sectovs} continues this strategy of ``regularizing'' ordered commutative monoids to structures that are easier to analyze mathematically. The new protagonists are ordered $\Q$-vector spaces, which form a particularly well-behaved subclass of ordered abelian groups, namely those that are torsion-free and divisible. Similar to the previous case, we associate to every ordered abelian group $G$ an ordered $\Q$-vector space $\ovs(G)$ together with a homomorphism $G\to \ovs(G)$ such that any other homomorphism $G\to V$ to an ordered $\Q$-vector space $V$ factors uniquely through $\ovs(G)$. This regularization also has a nice resource-theoretic interpretation, again in terms of two important differences between $G$ and $\ovs(G)$. First, the convertibility relation on $\ovs(G)$ is the many-copy convertibility, where $x$ can be converted into $y$ if and only if some number of copies of $x$ can be converted into the same number of copies of $y$ in $G$. Second, $\ovs(G)$ contains additional resource objects which behave like formal $n$-th parts or $n$-th fractions of the resource objects in $G$. 

If we start with an ordered commutative monoid $A$, then we can combine the two regularization procedures and obtain an ordered $\Q$-vector space $\ovs(\oag(A))$, which we also denote by $\ovs(A)$ for brevity. By doing this, both the catalytic and the many-copy regularizations have been performed. This may facilitate many conversions between resource objects that were impossible originally. At the mathematical level, we are now dealing with a vector space rather than a monoid, and hence the well-understood tools of linear algebra and functional analysis can be applied. The price to pay is that our two regularizations have lost some information about the original ordered commutative monoid $A$, and in particular about the original plain convertibility relation.

In Section~\ref{sectaovs}, a third step of regularization $V\mapsto\aovs(V)$ takes us from an ordered $\Q$-vector space $V$ to an Archimedean ordered $\Q$-vector space $\aovs(V)$, again satisfying a universal property analogous to the previous two. And again, an Archimedean ordered $\Q$-vector space is an ordered $\Q$-vector space which is well-behaved in a certain technical sense. The resource-theoretic interpretation is that in addition to allowing catalysis and many-copy conversions, we now also allow the use of seeds, i.e.~the consumption of an arbitrarily small amount of a resource object which may facilitate a desired conversion. At this point, we finally find an answer to Question~\ref{mainqstn} in the form of Theorem~\ref{aovshb}, which can be rephrased as saying that $x\geq y$ holds in an Archimedean ordered $\Q$-vector space $W$ if and only if $f(x)\geq f(y)$ for all functionals $f:W\to\R$. In Theorem~\ref{ocmhb}, we lift this result to a complete answer to Question~\ref{mainqstn} for ordered commutative monoids with a generating pair.

Figure~\ref{smcats} provides an overview of the three regularization procedures that we perform between ordered commutative monoids and the three well-behaved subclasses of ordered commutative monoids that have been mentioned.

In Section~\ref{sectfun}, we look abstractly at the collection of all functionals on an ordered commutative monoid, with particular emphasis on those ordered commutative monoids which have a generating pair. For this case, we prove in Corollary~\ref{ocmfunrep}, roughly speaking, that every functional is a nonnegative linear combination of extremal functionals (possibly infinite, i.e.~an integral).

In Section~\ref{sectrates}, we consider the resource-theoretic notion of rate and formalize it in our framework. Rates are concerned with the scaling of how many copies of a resource object $y$ one can extract from many copies of a resource object $x$. We propose a precise definition of rate and prove various properties. Since this notion of rate still turns out to have some technical deficits, we also propose a notion of regularized rate which applies to ordered commutative monoids with a generating pair. We derive a multitude of results on regularized rates, some of which are straightforward to show, while others seem surprisingly deep. Our probably most useful result is Theorem~\ref{rateformula}, which provides a formula for maximal regularized rates,
\beq
\boxed{\Rreg_{\max}(x\to y) = \inf_f \frac{f(x)}{f(y)}.}
\eeq
Here, $f$ ranges over all functionals with $f(y)\neq 0$, and it is assumed that $x,y\geq 0$. This results in a new formula~\eqref{scapform} for the Shannon capacity of a graph, which is a graph invariant that is notoriously hard to compute. As speculated in Example~\ref{probmajconj}, the rate formula may also have immediate applicability to the resource theories which have been investigated in quantum information theory. 

Section~\ref{sectonedim} investigates when an ordered commutative monoid $A$ embeds into $\R$, which means that there exists a functional $f:A\to\R$ such that $x\geq y$ is equivalent to $f(x)\geq f(y)$. In resource-theoretic terms, this property means that there is an essentially unique way of assigning a value or price to every resource object. Using our earlier results, we provide conditions for when this happens in Theorem~\ref{numocm}. This result is closely related to the utility theorem of von Neumann and Morgenstern in decision theory and economics, and also to a theorem of Lieb and Yngvason in the foundations of thermodynamics. While it is an extremely strong requirement for $A$ itself to embed into $\R$, it is a much weaker requirement for the regularization $\aovs(A)$ to embed into $\R$, and Theorem~\ref{aovsonedim} provides a characterization for when this happens. For technical reasons, we formulate and prove this result only for those ordered commutative monoids $A$ which have a generating pair and satisfy $x\geq 0$ for all $x\in A$. 

In Section~\ref{sectcompare}, we briefly compare our approach to other approaches for mathematically formalizing resources and their convertibility. This comprises Girard's linear logic, Deutsch's constructor theory, and some works which are more or less general approaches to resource theories in quantum information theory.

Appendix~\ref{appendixhb} proves the standard Hahn--Banach extension theorem for $\Q$-vector spaces, which we use in Sections~\ref{sectaovs} and~\ref{sectrates}.

Finally, Appendix~\ref{enrichment} explains how parts of the main text can be regarded as special cases of certain pieces of category theory. The reason that this is interesting and relevant is that using the more general setup of category theory could help us answer resource-theoretic questions of a different flavour than~\ref{qc}--\ref{qc}, such as: if $x$ can be converted into $y$, then \emphalt{how} does one achieve such a conversion?

\tikzstyle{isa}=[draw=black,arrows={>->},thick]
\tikzstyle{regularize}=[draw=black,arrows={->},thick]
\tikzstyle{turnsinto}=[draw=black,arrows={->},thick]
\tikzstyle{category}=[rectangle,draw=black,align=center,rounded corners]

\def\adjointpair#1#2#3{
	\draw [isa] (#2) to [bend left=25] (#1) ;
	\draw [regularize] (#1) to [bend left=25] (#2) ;
	\path (#1) -- (#2) node[midway,align=center] {$\perp$\\[-3pt] #3} ; }

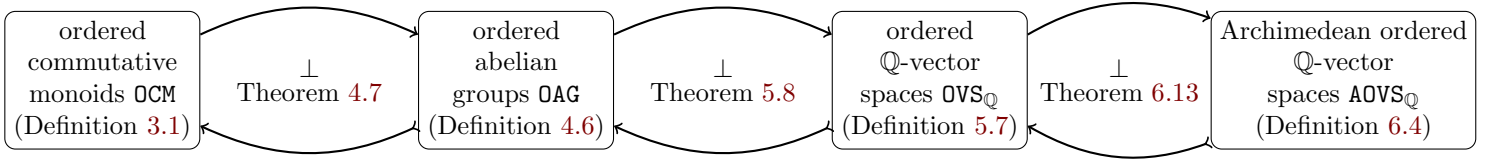
\begin{figure}
\makebox[\textwidth][c]{
\begin{tikzpicture}
\newcommand*{\horizontalsep}{5.5}
\node[category] at (0,0) (ocm) {ordered\\ commutative\\ monoids $\OCM$\\ (Definition~\ref{ocm})} ;
\node[category] at (\horizontalsep,0) (oag) {ordered\\ abelian\\ groups $\OAG$\\ (Definition~\ref{oag})} ;
\node[category] at (2*\horizontalsep,0) (ovs) {ordered\\ $\Q$-vector\\ spaces $\OVS$\\ (Definition~\ref{ovs})} ;
\node[category] at (3*\horizontalsep,0) (aovs) {Archimedean ordered\\ $\Q$-vector\\ spaces $\AOVS$\\ (Definition~\ref{aovs})} ;
\adjointpair{ocm}{oag}{Theorem~\ref{ocmtooag}}
\adjointpair{oag}{ovs}{Theorem~\ref{oagtoovs}}
\adjointpair{ovs}{aovs}{Theorem~\ref{ovstoaovs}}
\end{tikzpicture}}
\caption[]{The main kinds of structures considered in this paper and their relationships. Every box is a category, every arrow is a functor, and the left-pointing arrows indicate inclusions of full subcategories. Each pair of arrows denotes an adjunction, so that all these subcategories are reflective.

In down-to-earth terms, this means that we can read the first left-pointing arrow as ``ordered abelian groups are just ordered commutative monoids, but with extra properties'', while the first right-pointing arrow says ``Every ordered commutative monoid can be regularized to an ordered abelian group in a universal manner'', and similarly for the other arrows.}
\label{smcats}
\end{figure}

\bigskip
\subsection*{Terminology and notation}

We have followed three guiding principle in choosing our terminology:\medskip

\begin{itemize}
\item The mathematical terms should be separate and independent from their resource-theoretic interpretation. One reason for this is that the mathematical structures that we investigate here also come up in other contexts, such as the $K$-theory of operator algebras~\cite{KtheoryCstar,extstates}. One exception where we have not followed this principle is with the notion of \emphalt{rate} investigated in Section~\ref{sectrates}.
\item The various terms used should nicely match up with each other. For example, we use the adjective ``ordered'' all the way from ordered commutative monoids to Archimedean ordered $\Q$-vector spaces.
\item Many concepts from algebra in an unordered setting have a similar flavour in our ordered setting, and hence we find it useful to employ the same terminology. This applies to the definitions of annihilators (Remark~\ref{annihilator}), cancellativity (Definition~\ref{canc}), homomorphisms (Definition~\ref{defhom}) and torsion-freeness (Definition~\ref{deftf}).
\end{itemize}
\medskip

In some cases, this means that our mathematical terminology and notation clashes with some other mathematical literature, but the benefits of a consistent terminology following these principles should be higher than the cost.

\newcommand{\notation}[2]{#1: & \textrm{\: #2}. \\}

Concerning the mathematical notation, we have tried not to overload the same symbol with different meanings. The following table provides an overview of most of our notation, roughly in order of appearance:

\begin{align*}
\notation{j,k,l,m,n}{natural number coefficients in $\N$ (including $0$) or $\Npos$}
\notation{\alpha,\beta,\eps,\kappa,\lambda,\mu}{rational coefficients in $\Q$, $\Qplus$ or $\Qpos$, even when not explicitly designated as rational}
\notation{r,s,t}{real coefficients in $\R$, $\Rplus$ or $\Rpos$}
\notation{A}{an ordered commutative monoid (Definition~\ref{ocm}), in Section~\ref{prelims} an ordered set (Definition~\ref{os})}
\notation{x,y,z,w}{elements of an ordered commutative monoid}
\notation{\mathcal{G},\mathcal{H}}{graphs (Example~\ref{graphs})}
\notation{f}{a functional (Definition~\ref{deffunc}) or more general kind of homomorphism (Definition~\ref{defhom})}
\notation{(g_+,g_-)}{a generating pair in an ordered commutative monoid (Definition~\ref{generatingpair})}
\notation{\mathcal{K}_n}{the complete graph on $n$ vertices (Example~\ref{grphgenpair})} 
\notation{G}{an ordered abelian group (Definition~\ref{oag})}
\notation{g}{a generator in an ordered abelian group (Definition~\ref{oaggenerator})}
\notation{V}{an ordered $\Q$-vector space (Definition~\ref{ovs})}
\notation{W}{an Archimedean ordered $\Q$-vector space (Definition~\ref{aovs})}
\notation{U}{an absolutely convex absorbent set (Definition~\ref{defacc})}
\notation{R_{\min}, R_{\max}}{minimal and maximal rate (defined in~\eqref{defminrate} and~\eqref{defmaxrate})}
\notation{\Rreg_{\min}, \Rreg_{\max}}{minimal and maximal regularized rates (Remark~\ref{minmaxregrates})}
\notation{p}{a sublinear function (Theorem~\ref{hbext})}
\end{align*}

Last but not least, we try to harness the power of category theory as an organizing language whenever this seems possible, while also trying to explain matters in a way which will hopefully be comprehensible to those who have not yet come to appreciate the ``joy of cats''~\cite{cats,basiccats}.

\bigskip
\subsection*{A disclaimer} 

Due to the wide span of phenomena and situations which our general formalism is built to capture, we cannot be certain that our definitions will necessarily apply to all different kinds of resource theories that are of relevance to applications. In fact, as Remark~\ref{epsilonification} shows, it is already clear that there are resource theories with high relevance to probability theory and information theory which cannot be described in our framework. We hope to address this important issue in future work.

Furthermore, there is evidence that the technical details of even our current limited framework have not yet reached their final form. In particular, we see two main issues:

\begin{enumerate}
\item Many of our results assume the existence of a generator or generating pair. We suspect that this can mostly be avoided at the cost of higher complexity in the definitions and theorem statements: Lemma~\ref{seqclose} shows that the existence of a generator in an ordered $\Q$-vector space $V$ simplifies the construction of $\aovs(A)$ significantly as compared to the general case. It would be better for the presentation of our results to either assume the existence of a generating pair throughout, or to do away with it completely. Since we know neither whether we can expect a generating pair to exist in (almost) all applications of interest, nor how much more complicated the general definitions and theorems would be, we have left this tension unresolved for now.
\item The catalytic regularization of Section~\ref{sectoag} and the many-copy regularization of Section~\ref{sectovs} are qualitatively different from the seed regularization of Section~\ref{sectaovs}: the former two do not only modify the ordering or convertibility relation, but also complete the collection of resource objects by throwing in new ``imaginary'' objects like formal negatives or formal fractions; the latter regularization only regularizes the ordering without adding new elements. It turns out that one can perform this third step in a manner which also adds new elements and thereby completes an Archimedean ordered $\Q$-vector space further to an Archimedean ordered $\R$-vector space; see Remark~\ref{QvsR}. However, an Archimedean ordered $\R$-vector space is more than just an ordered commutative monoid with special properties: the scalar multiplication by arbitrary reals really does add additional structure. Since we try to focus on ordered commutative monoids in this work, we do not consider this additional step. We rather work with Archimedean ordered $\Q$-vector spaces, since these still are ordered commutative monoids with extra properties but no extra structure.
\end{enumerate}

Finally, since this paper touches on many different areas and brings them together, it is only natural that many of our methods and results are not new or even completely standard. We have tried to explain this in all cases in which we are aware of this, but it is hard to guarantee that all non-original ideas are properly designated as such. For this reason, we do not want to claim any particular one of our ideas to be original, but nevertheless hope that some of them are novel and interesting.

\newpage
\section{Preliminaries on ordered sets}
\label{prelims}

Before embarking on the journey to resource theories with both a notion of convertibility and a notion of combining resource objects, we set the stage by considering the structure of resource convertibility only. This allows for a simple mathematical treatment in which we can still make some essential observations. We hope that grasping these points will help the reader understand the upcoming main part of this work.

So what is a notion of convertibility? For any two resource objects $x$ and $y$, either $x$ is convertible into $y$, which we denote $x\geq y$, or $x$ is not convertible into $y$, for which we write $x\not\geq y$. The way that we think about the convertibility $x\geq y$ is that if we have $x$, then we can turn it into $y$, and $x$ gets consumed in doing so.

Which mathematical properties should this convertibility relation have? Clearly, any $x$ should be convertible into itself, e.g.~by doing nothing. Furthermore, if $x$ is convertible into $y$ and $y$ into $z$, then $x$ should also be convertible into $z$, e.g.~by composing a conversion of $x$ into $y$ and a conversion of $y$ into $z$. In this way, we find that the collection of resource objects forms an ordered set:

\begin{defn}
\label{os}
An \emph{ordered set} $A$ is a set equipped with a binary relation $\geq$ satisfying
\begin{itemize}
\item reflexivity, 
\beq
x\geq x.
\eeq
\item transitivity,
\beq
x\geq y,\quad  y\geq z \quad\Longrightarrow\quad x\geq z.
\eeq
\item antisymmetry,
\beq
x\geq y,\quad y\geq x \quad\Longrightarrow\quad x = y.
\eeq
\end{itemize}
\end{defn}

The way that we think about the antisymmetry axiom is not as an additional axiom, but rather as the \emphalt{definition} of equality in terms of the ordering:
\beqn
\label{equaldef}
x=y \quad:\Longleftrightarrow\quad x\geq y \:\textrm{ and }\: y\geq x .
\eeqn
In other words, two resource objects can be considered the same for the purpose of resource theories as soon as they are mutually interconvertible. This is the relevant notion of being ``the same'' in our context, and hence we might as well use the equality symbol for it and call it ``equality'', instead of introducing another term like ``equivalence''. So here and in all of the following sections, we work with~\eqref{equaldef}, even when not mentioned explicitly. This means that whenever we define a binary relation ``$\geq$'' which satisfies reflexivity and transitivity, we automatically regard two elements of the underlying set as equal as soon as they are ordered both ways, i.e.~we automatically take the quotient of the underlying set such that the antisymmetry axiom holds as well.

\begin{ex}
Another reason to adopt~\eqref{equaldef} as the definition of equality is that in some resource theories, it is not even clear what \emphalt{other} notions of equality there are. For example, consider a resource theory in which a resource object is a finite probability space, i.e.~a finite set equipped with a probability distribution. Then when should two such resource objects be considered equal? Do the underlying sets of these distributions have to contain exactly the same elements? Or is it sufficient to have a bijection between the outcome sets which preserves the individual elements' probabilities? This second notion of equality would be much coarser than the first. In particular, it would make two ``different'' distributions on \emphalt{the same set} equal as soon as the set's elements can be permuted so that the probabilities match up. But not only is the notion of equality of these probability spaces conventional; as far as the resource theory is concerned, it is also \emphalt{irrelevant}.
\end{ex}

\begin{rem}
\label{set}
For us, an unordered set is an ordered set in which the ordering relation is symmetric, i.e.~in which $x\geq y$ holds if and only if $y\geq x$ holds. This clearly means that $x\geq y$ holds if and only if $x=y$. Along these lines, we may also think of an ordering relation $\geq$ as a notion of equality in which the assumption of symmetry of equality is dropped.
\end{rem}

\begin{rem}
Instead of interpreting $x\geq y$ as ``$x$ can be converted into $y$'', we can also take it to mean ``$x$ can substitute for $y$''. For example, consider a company and two of their employees, say Alice and Bob. Say that Alice can substitute for Bob when Bob is away, but Bob cannot stand in for Alice when she is off work. This results in a ``substitutivity'' relation which is very similar in spirit to a convertibility relation, but the interpretation is slightly different. In particular, Alice can trivially substitute for herself, which leads to reflexivity; and if Alice can substitute for Bob and Bob for Charlie, then Alice can also substitute for Charlie (not necessarily at the same time), which leads to transitivity. Substitutivity is not symmetric in general, and therefore not an equivalence relation. This has previously been found to be important by the philosopher Brandom, who discusses the issue of symmetry of substitutivity in the context of inferentialism in~\cite[Chapter~4]{inferentialism}.

The difference between a convertibility relation and a substitutivity relation is that in convertibility, it is the objects themselves which become different; in substitutivity, the objects stay the same, and it is only the way in which they can be used which changes. This is very similar to the distinction between \emphalt{active transformations} and \emphalt{passive transformations} in physics. Since the mathematics of convertibility relations and substitutivity relations is the same, for us there is no need to distinguish between convertibility and substitutivity. In particular, all of our results also apply to substitutivity relations. 

Yet other possible interpretations in possible. In economics, we may take ``$x\geq y$'' to mean ``I prefer $x$ over $y$'', so that the preferences of an agent are encoded in an ordered set~\cite{games}. Another possible interpretation is the ``financial accessibility'' of~\cite[Section~1.2]{entropyprinciple}, where ``$x\geq y$'' is taken to stand for ``there is a market in which it is possible to exchange $x$ for $y$''.
\end{rem}

\begin{ex}[The resource theory of chemistry]
\label{rtchem}
The resource theory of chemistry forms a subfield of chemistry known as \emphalt{stoichiometry}. Here, the resource objects are collections of molecules denoted as formal sums such as
\beq
\mathrm{CH_4}, \qquad \mathrm{C_{60} + HCl + C_6 H_6}, \qquad \mathrm{2H_2O_2},\qquad \mathrm{4 N_2 + O_2}, \qquad \ldots
\eeq
where the coefficients describe the number of molecules of each kind. Hereby, we think of each molecule in the collection as contained in a separate ``box'', so that the collection really consists of a bunch of boxes containing one molecule each. This will ensure that an imaginary chemical engineer operating with these molecules has complete control over what happens to each individual molecule, and can let the different molecules react as desired. We also make the idealized assumption that while a molecule is contained in its box, it is guaranteed not to decompose spontaneously.

The convertibility relation is going to be defined in terms of chemical reactions, such as:
\beqn
\label{chemreact}
\mathrm{CH_4 + 2O_2} \lra \mathrm{CO_2 + 2H_2 O},\qquad \mathrm{Zn + 2HCl} \lra \mathrm{ZnCl_2 + H_2} .
\eeqn
Concretely, we declare a collection of molecules to be ``greater than or equal to'' another one if and only if there is a sequence of chemical reactions which can be performed so as to convert the former collection into the latter. For example, the reactions~\eqref{chemreact} taken together tell us that
\beq
\mathrm{CH_4 + 2O_2 + Zn + 2HCl} \geq \mathrm{CO_2 + 2H_2O + ZnCl_2 + H_2}.
\eeq
Hereby, we imagine that the boxes containing each molecule can be brought together at will, so that the operator can decide which reactions are to take place in order to convert which subcollection of molecules into something else. With this in mind, one collection of molecules becomes convertible into another if and only if there is a sequence of individual reactions which achieves the conversion, and such that each reaction operates on a subcollection of the molecules. Following~\eqref{equaldef}, two collections of molecules are regarded as equal as soon as they are mutually interconvertible. This defines the ordered set $\Chem$, modulo a certain ambiguity as to which which reactions exactly one wants to allow. In applications to laboratory or industrial chemistry, one can choose all reactions~\eqref{chemreact} that can be realized with the given equipment.

We will soon consider $\Chem$ as an ordered commutative monoid, and this is then the mathematical structure which formalizes the resource theory of chemistry.
\end{ex}

\begin{rem}
\label{epsilonification}
In many resource theories of practical interest, there is more mathematical structure in resource convertibility than just an ordering. The reason is the following: if a given $x$ cannot be converted into a desired $y$ exactly, then one would like to know \emphalt{how close} one can get to $y$ from $x$. In other words, one needs to be able to consider \emphalt{approximate} conversion of $x$ into $y$. This is of central importance for example in Shannon's resource theory of communication, which we discuss in Section~\ref{sectocm}. In such a setup, describing convertibility as a binary relation with possible values ``yes, $x$ is convertible into $y$'' or ``no, $x$ is not convertible into $y$'' is not sufficiently expressive. We intend to tackle this problem of ``epsilonification''~\cite{resourcesI} in future work.

Epsilonification is also relevant for $\Chem$. If one follows the above prescription for constructing the order, one needs to make a sharp choice on which reactions~\eqref{chemreact} are to be considered and which ones are not. In an epsilonified setting, it should be possible to take more quantitative information about the individual reactions into account, such as the speed at which an individual reaction happens.
\end{rem}

Just as with most other mathematical structures, it is not only of interest to consider individual ordered sets, but also maps between ordered sets.

\begin{defn}
\label{orderedmap}
An \emph{ordered map} between ordered sets $A$ and $B$ is a function $f:A\to B$ such that
\beq
x\geq y \quad\Longrightarrow\quad f(x)\geq f(y).
\eeq
\end{defn}

Ordered maps automatically respect equality: since $x=y$ means $x\geq y$ and $y\geq x$, this implies $f(x)\geq f(y)$ and $f(y)\geq f(x)$, and therefore $f(x)=f(y)$.

In resource theories, one is often interested in assigning a real number to each resource object, with the idea being that this number measures the value of the resource object relative to the other ones. In other words, if $A$ is the ordered set formalizing the resource theory, then one is interested in ordered maps $A\to \R$, where $\R$ is an ordered set in the usual way. Such ordered maps are sometimes known as \emph{resource monotones} or just \emph{monotones}; the terminology possibly originates with Vidal, who introduced resource monotones in the study of quantum entanglement~\cite{entanglementmonotones}. By definition, a monotone $f$ preserves the convertibility relation: if $x\geq y$, then $f(x)\geq f(y)$. An important question is: under what conditions does the converse hold as well? Is there a monotone $f$ which not only preserves, but also \emphalt{reflects} the ordering in the sense that
\beqn
\label{ff}
x\geq y \quad\Longleftrightarrow \quad f(x) \geq f(y) \quad ?
\eeqn
Such a monotone is extremely useful, as it allows one to decide the convertibility of $x$ into $y$ unambiguously simply by comparing $f(x)$ with $f(y)$. It is automatically injective: $f(x)=f(y)$ means $f(x)\geq f(y)$ and $f(y)\geq f(x)$, which implies $x\geq y$ and $y\geq x$ by assumption, and therefore $x=y$ by the definition of equality~\eqref{equaldef}. Therefore, the existence of such a monotone is equivalent to the existence of an embedding of $A$ into $\R$.

More generally, one can ask for the existence of a \emphalt{family} of monotones $f_i$ indexed by some $i\in I$ such that
\beqn
\label{sepfamily}
x\geq y \quad\Longleftrightarrow \quad f_i(x) \geq f_i(y) \quad \forall i.
\eeqn
Because of the same argument as before, the existence of such a family is equivalent to an embedding $A\to \R^I$, where $\R^I$ is the set of all functions $I\to\R$ equipped with the pointwise ordering. It is easy to see that such a family of monotones always exists if one takes $I=A$~\cite[Proposition~5.2]{resourcesI}, and in fact it is sufficient to let them take values in the Booleans $\{0,1\}\subseteq \R$ only instead of all of $\R$. So, proving the existence of a family of monotones satisfying~\eqref{sepfamily} is a very simple matter. However, in practice it is of interest to find a reasonably small such family such that all its members are easily computable, and this is typically much more difficult. For example, if one considers the ordered set of Turing machines with $\mathcal{T}\geq\mathcal{S}$ whenever $\mathcal{T}$ can simulate $\mathcal{S}$, then no computable family of monotones with~\eqref{sepfamily} can exist, since this would solve the halting problem.

\begin{rem}
\label{conservationlaw}
Another interesting kind of monotone $f$ is when $-f$ is also a monotone. Such a monotone is usually called a \emph{conservation law}. A function $f:A\to\R$ is a conservation law if and only if $x\geq y$ implies that $f(x)=f(y)$, which means that the value of $x$ is conserved under any conversion, including conversions that are not reversible.
\end{rem}

\begin{ex}
\label{totalatoms}
In chemistry, the total number of atoms of any given element is a conservation law. For example, the number of hydrogen atoms is $6$ on both sides of the reaction
\beq
\mathrm{Zn + 2HCl} \lra \mathrm{ZnCl_2 + H_2},
\eeq
and it is similarly conserved in any other chemical reaction. (For simplicity, we do not consider nuclear reactions to be part of chemistry.)
\end{ex}

\newpage
\section{Resource theories: ordered commutative monoids}
\label{sectocm}

Until now, we have only considered the convertibility relation on resource objects. The mathematics relevant for this is the theory of ordered sets. While ordered sets can be highly complex in their structure, they have been extensively studied during the 20th century and we do not have anything new to say about them.

However, the convertibility relation is not the only piece of structure that a resource theory typically has: there also should exist the possibility of \emphalt{combining} resource objects, in the sense that for any two resource objects $x$ and $y$, there also is a resource object $x+y$ which describes their combination. We think of $x+y$ as the object which represents having access to both $x$ and $y$ together and at the same time, and such that both $x$ and $y$ get consumed in the process of converting $x+y$ into something else. However, other interpretations are certainly possible: for example, one could also take $x+y$ to represent a disjunctive combination of resource objects, and all our definitions and results will apply just as well, although certain things like functionals (Definition~\ref{deffunc}) will be more difficult to interpret.

When this additional structure is considered, the mathematics of ordered sets needs to be replaced by the mathematics of ordered commutative monoids, which turns out to be considerably richer. We propose that any resource theory (modulo Remark~\ref{epsilonification}) should be mathematically described by an ordered commutative monoid in the following sense:

\begin{defn}
\label{ocm}
An \emph{ordered commutative monoid} is an ordered set $A$ equipped with the following additional pieces of structure:\medskip
\begin{itemize}
\item a binary operation $+$,
\item a distinguished element $0\in A$,\medskip
\end{itemize}
and such that the following axioms hold:\medskip
\begin{itemize}
\item $+$ is associative and commutative,
\beq
x+(y+z) = (x+y)+z, \qquad x+y=y+x,
\eeq
\item $0\in R$ is a neutral element,
\beq
0 + x = x,
\eeq
\item addition respects the ordering,
\beqn
\label{monotoneplus}
x\geq y \quad\Longrightarrow\quad x+z\geq y+z.
\eeqn
\end{itemize}
\end{defn}

\begin{rem}
Up to the discussion on equality in the previous section, this coincides with our previous definition of a ``theory of resource convertibility''~\cite[Definition~4.1]{resourcesI}. Also in Lieb and Yngvason's approach to the foundations of thermodynamics~\cite{secondlaw1}, our definition corresponds to the axioms ``A1'' to ``A3'' together with their background assumptions on forming composite systems. Furthermore, mathematical structures like this or closely related ones have previously been studied:
\begin{itemize}
\item Definition~\ref{ocm} coincides with the definition of commutative pomonoid of~\cite{Fakhruddin,RafteryVanAlten}, the definition of commutative p.o.~monoid of~\cite{Luttik}, and with the definition of partially ordered monoid of~\cite{TsinakisZhang}\footnote{The connection to Petri nets made in~\cite{TsinakisZhang} is especially interesting: a Petri net can be understood as presentation of an ordered commutative monoid in terms of generators and relations, and this is what the adjunctions of~\cite{TsinakisZhang} secretly describe. This should be investigated in more detail elsewhere.}.
\item It also coincides with the notion of partially-ordered semimodule over the partially-ordered semiring $\N$ in the sense of~\cite{semirings}.
\item Preordered abelian semigroups~\cite[Section~1.1]{convexcones} only differ in using preorders instead of partial orders, i.e.~two elements are not necessarily regarded as equal when they are ordered both ways.
\item The positively ordered monoids of~\cite{injpoms} and the $po^+$-monoids of~\cite{seppom} differ in addition by also postulating $x\geq 0$ for all $x$.
\item The partially ordered abelian semigroups of~\cite{strongsemi} differ from ordered commutative monoids only in that their are not required to have a neutral element.
\item The preordered semigroups of~\cite{extstates} do neither require a neutral element nor postulate~\eqref{equaldef}.
\item Unital commutative quantales~\cite{Rosenthal} are ordered commutative monoids in which the ordering is a complete lattice, and the addition preserves joins.
\end{itemize}
\end{rem}

While the ordering in an ordered commutative monoid formalizes the convertibility relation and addition describes the possibility of combining resources, it remains to discuss the meaning of the three axioms in Definition~\ref{ocm}. First, associativity and commutativity of ``$+$'' say that any finite set of resources can be grouped together in a unique way, in the sense that any two ways of grouping the resource objects result in mutually interconvertible composite resource objects. Second, the zero element $0$ represents the void resource object, which when combined with any other resource object results in something mutually interconvertible with the second. In both cases, we speak about ``mutually interconvertible'', since this is how we defined equality in~\eqref{equaldef}. There is still no need to use any other notion of equality on $A$. Third, the monotonicity of addition~\eqref{monotoneplus} means that if $x$ can be converted into $y$, then we can also convert the composite object $x+z$ into the composite object $y+z$, e.g.~by converting $x$ into $y$ and doing nothing to $z$.

We do not assume that $x\geq 0$ necessarily holds for all $x\in A$. There are several reasons for this: first, assuming $x\geq 0$ for all $x$ would amount to assuming that any resource object can be converted into nothing, i.e.~can be discarded or be made to vanish at will. Although this indeed happens in many resource theories that are of interest in practice, this will not always be the case. For example, one can consider resource theories which contain objects that cannot just be ignored after their disposal, such as nuclear waste. Second, a related issue is that if $x\geq 0$ for all $x\in A$, then $A$ cannot have any nontrivial conservation law: any resource theory with a non-trivial conservation law, such as $\Chem$, will have resource objects $x$ with $x\not\geq 0$. Third, most of our results do not require the assumption $x\geq 0$; a notable exception is Theorem~\ref{rateformula}. Fourth, imposing $x\geq 0$ would make Definition~\ref{ocm} lose the appealing and useful self-duality property that reversing the ordering in an ordered commutative monoid yields another ordered commutative monoid.

\begin{rem}
\begin{enumerate}
\item We can understand the monotonicity property~\eqref{monotoneplus} as stating that the map $x\mapsto x+z$ is an ordered map $A\to A$ in the sense of Definition~\ref{orderedmap}.
\item In particular, this monotonicity property implies that the addition operation respects equality,
\beq
x = y \quad\Longrightarrow\quad x+z = y+z.
\eeq
\end{enumerate}
\end{rem}

\begin{rem}
\label{commmon}
Following Remark~\ref{set}, a commutative monoid is an ordered commutative monoid having the special property that the ordering is symmetric, in the sense that $x\geq y$ implies $y\geq x$.
\end{rem}

Other properties of ordered commutative monoids are as follows:

\begin{itemize}
\item The unit element $0$ is unique: if $0'$ is another unit element, then we have $0=0+0'=0'+0=0'$.
\item The monotonicity~\eqref{monotoneplus} also holds in a stronger form: if $x\geq y$ and $w\geq z$, then also $x+w\geq y+z$, since
\beq
x+w \stackrel{\eqref{monotoneplus}}{\geq} y+w \;=\; w+y \stackrel{\eqref{monotoneplus}}{\geq} z+y \;=\; y+z.
\eeq
\end{itemize}

When $A$ is an ordered commutative monoid, $x\in A$ and $n\in\N$, then we also write $nx$ as shorthand for the $n$-fold sum of $x$,
\beq
nx \defin \underbrace{x+\ldots+x}_{n \textrm{ times}}.
\eeq
This turns $A$ into a semimodule over the semiring $\N$, which means nothing other than that the distributive laws $(n+m)x=nx+mx$ and $n(x+y)=nx+ny$ hold, as well as $1x=x$ and $0x=0$.

\begin{rem}
\label{annihilator}
Let $A$ be an ordered commutative monoid and $x,y\in A$. Then the set
\beq
\ann(x,y) \defin \{ \: n\in \N \:|\: nx\geq ny\: \} 
\eeq
is a subset of $\N$. We call this subset the \emph{annihilator}, since in the special case of an abelian group, it coincides with the familiar notion of annihilator in algebra: an abelian group is trivially ordered, and hence $nx\geq ny$ is equivalent to $nx=ny$, and hence also to $n(x-y)=0$. So in this case, our $\ann(x,y)$ coincides with the set of all $n\in\N$ for which $n(x-y)=0$, which is the standard notion of annihilator. 

In general, the annihilator $\ann(x,y)$ is an ideal in $\N$, in the sense that it contains $0$, is closed under addition, and therefore also closed under multiplication by arbitrary natural numbers. It is itself a commutative monoid. While we could also regard it as an ordered commutative monoid with respect to the usual ordering on $\N$, it is not clear to us what the relevance of considering this particular ordering would be.
\end{rem}

We now get to our main examples of ordered commutative monoids and refer to~\cite{resourcesI} for many more.

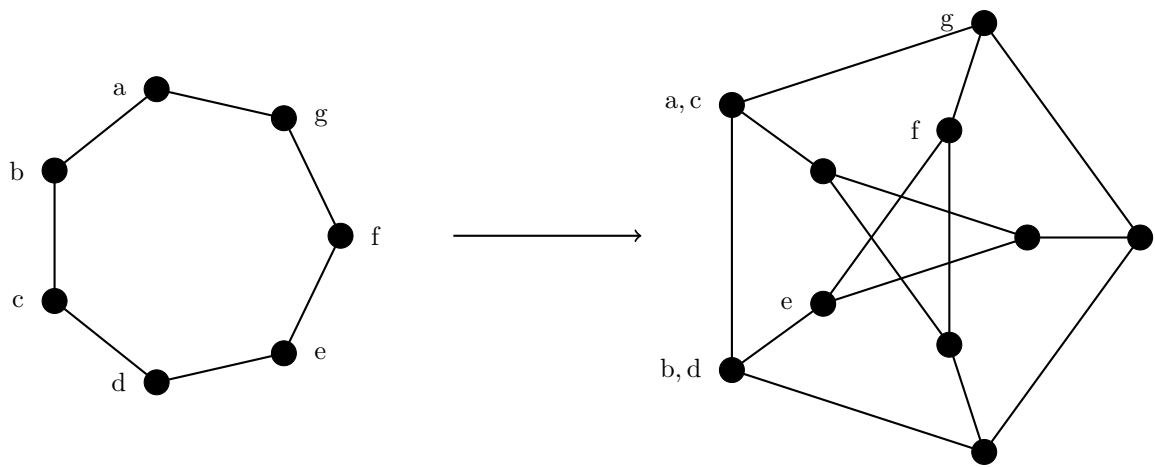
\begin{figure}
\centering
\begin{tikzpicture}
\tikzset{VertexStyle/.style={shape=circle,draw,fill}}
\SetVertexNoLabel
\begin{scope}[local bounding box=scope1]
	\grCycle[RA=2]{7}
	\tikzset{AssignStyle/.append style={left=8pt}}
	\AssignVertexLabel{a}{,,a,b,c,d}
	\tikzset{AssignStyle/.append style={right=8pt}}
	\AssignVertexLabel{a}{f,g,,,,,e}
\end{scope}
\draw[->,thick] (3.5,0) -- (6,0) ;
\begin{scope}[shift={($(scope1.east)+(7cm,0)$)}]
	\grPetersen[form=1,RA=3,RB=1.5]
	\tikzset{AssignStyle/.append style={left=8pt}}
	\AssignVertexLabel{a}{,g,{a,\thinspace c},{b,\thinspace d}}
	\tikzset{AssignStyle/.append style={left=8pt}}
	\AssignVertexLabel{b}{,f,,e}
\end{scope}
\end{tikzpicture}
\caption{Two graphs together with a graph map between them. The map is specified by taking every vertex of the source graph to the corresponding vertex of the target graph which carries the same label.}
\label{graphmap}
\end{figure}

\begin{ex}[The ordered commutative monoid of chemistry]
In the resource theory of chemistry from Example~\ref{rtchem}, we can add two collections of molecules by simply joining them up:
\beq
\mathrm{(C_{60} + 3H_2) + (4N_2 + O_2)} \defin \mathrm{C_{60} + 3H_2 + 4N_2 + O_2}.
\eeq
By definition of the convertibility relation, this addition preserves the ordering $\geq$. We also have a neutral element for ``$+$'' given by the empty collection of molecules, denoted by ``$0$''.  Hence $\Chem$ becomes an ordered commutative monoid.
\end{ex}

\begin{ex}[The ordered commutative monoid of graphs]
\label{graphs}
Another ordered commutative monoid that our formalism can be successfully applied to is $\Grph$, the ordered commutative monoid of graphs. In the upcoming sections, we will achieve a glance of some of its highly intricate structure.

In order to define $\Grph$, we start with the collection of all finite graphs\footnote{For us, graphs are \emphalt{simple graphs}, i.e.~undirected graphs that do not have self-loops or parallel edges. We also require a graph to be nonempty, i.e.~to have at least one vertex, while allowing its set of edges to be empty.}. If $\mathcal{G},\mathcal{H}\in\Grph$, then we put $\mathcal{H}\geq \mathcal{G}$ whenever there exists a graph map $\mathfrak{m}:\mathcal{G}\to \mathcal{H}$. Such a map consists of a function which sends every vertex of $\mathcal{G}$ to a vertex of $\mathcal{H}$ such that adjacent vertices get taken to adjacent vertices; Figure~\ref{graphmap} shows an example. We think of this as saying that $\mathcal{H}$ can in some sense ``simulate'' $\mathcal{G}$ and is therefore at least as useful as $\mathcal{G}$. By this, we have basically defined $\Grph$ as the preorder reflection (Appendix~\ref{enrichment}) of the category of graphs. Already as an ordered set, the structure of $\Grph$ is highly complex; for example, it is known that every finite ordered set arises in $\Grph$ as an ordered subset. This is a special case of the results of~\cite[Section~IV.3]{representations}.

Concerning how to combine graphs, there are many different possible ways to do this: one can take the disjoint union, the disjoint union plus all edges between the two components (graph join), or one of the many different products of graphs~\cite{prodgraphs}. We suspect that many of these choices will give rise to an interesting ordered commutative monoid of graphs. But ultimately, the choice of which combination operation to use has to be determined by the desired application. For our purposes, this will be the relation to communication channels in Proposition~\ref{cctogrph}, and that we want to ultimately recover some well-known graph invariants in terms of the ordered commutative monoid structure of $\Grph$, in particular the chromatic number or the Shannon capacity. This leads us to take the combination of graphs $\mathcal{G}$ and $\mathcal{H}$ to be given by the \emph{disjunctive product} $\mathcal{G}\ast \mathcal{H}$, also known as the \emph{co-normal product}. The disjunctive product $\mathcal{G}\ast \mathcal{H}$ is defined to be the graph whose vertices are pairs $(v,w)$ with $v$ a vertex of $\mathcal{G}$ and $w$ a vertex of $\mathcal{H}$, and such that $(v,w)$ is adjacent to $(v',w')$ if and only if $v$ is adjacent to $v'$ or $w$ is adjacent to $w'$,
\beq
(v,w) \sim (v',w') \quad :\Longleftrightarrow\quad v\sim v' \:\lor\: w\sim w'.
\eeq
We leave it to the reader to check that the ordering of graphs $\geq$ together with this binary operation $\ast$ results in an ordered commutative monoid whose neutral element is $0=\mathcal{K}_1$, the graph on one vertex. We denote this ordered commutative monoid by $\Grph$. We will use it as our running example of an ordered commutative monoid in the upcoming sections and also find some (more or less obvious and well-known) relations to important graph invariants.
\end{ex}

Note that we do not claim that the study of resource theories or ordered commutative monoids to be especially closely linked to graph theory. We rather want to illustrate graph theory as one particular area in which our language and way of thinking can be useful, and we hope that there will be plenty of others as well.

Those who are happy with $\Grph$ as an example and not particularly interested in information theory may skip the following example and continue with Definition~\ref{defhom}.

\begin{ex}[{The resource theory of communication~\cite[Example~2.6]{resourcesI}}]
We would like to discuss and analyze one particular example of a resource theory of particular practical relevance: the \emph{resource theory of communication}, as developed by Shannon~\cite{Shannon}. It is concerned with communication channels, which are mathematical abstractions of communicating at limited bandwidth and in the presence of potential transmission errors due to noise. One can use encoding and decoding of messages in order to simulate a target communication channel---typically the one which models perfect transmission without error---with the help of a given channel. One can also combine channels by using them in parallel. Hence we can describe the resource theory of communication in terms of an ordered commutative monoid which we denote $\CC$. Again loosely following ideas of Shannon~\cite{zeroerror}, we will also see that one can study $\CC$ in terms of a homomorphism to $\Grph$.

In more detail, a resource object in the theory of communication is a \emph{communication channel} $P:\inalph\to\outalph$, consisting of the following pieces of data:
\begin{itemize}
\item an \emph{input alphabet}, which is a finite set $\inalph$,
\item an \emph{output alphabet}, which is a finite set $\outalph$,
\item for every input symbol $a\in\inalph$, a probability $P(b|a)\in[0,1]$ to get the output symbol $b\in\outalph$ upon sending the input symbol $a\in\inalph$ through the channel, so that $\sum_b P(b|a)=1$.
\end{itemize}

In summary, we can understand the probabilities $P(b|a)$ as forming a conditional probability distribution, i.e.~a \emph{stochastic map}, of type $P:\inalph\to\outalph$\footnote{Note that the arrow notation does not mean that $P$ is a function with domain $\inalph$ and codomain $\outalph$, but we can interpret $P:\inalph\to\outalph$ as stating that $P$ is a morphism in the category of stochastic maps between the objects $\inalph$ and $\outalph$.}. We think of $P$ as the mathematical model of a communication link between two distant locations or parties. This communication link accepts an input symbol $a\in\inalph$ and transmits it to the other end; but due to noise or other modifications possibly afflicted to the symbol, the output symbol $b\in\outalph$ may be totally different and is not a deterministic function of $a$. We can only specify the probability of receiving a particular output $b$, given that a particular input $a$ has been sent. 

The idea now is that we can use such a communication channel in order to simulate other communication channels, by pre-processing of the input symbol (\emph{encoding}) and post-processing of the output symbol (\emph{decoding}). In other words, we can convert a channel $P : \inalph \to \outalph$ into a channel $Q : \inalphh\to\outalphh$ if and only if there are stochastic maps $\enc : \inalphh\to\inalph$ and $\dec : \outalph\to\outalphh$ such that
\beqn
\label{channeltrafo}
Q = \dec\circ P\circ \enc,
\eeqn
where the composition of stochastic maps on the right-hand side is given by matrix multiplication,
\beq
(\dec\circ P\circ \enc)(d|c) = \sum_{a,b} \dec(d|b) P(b|a) \enc(a|c).
\eeq
We can also write~\eqref{channeltrafo} pictorially:
\beq
\begin{tikzpicture}[scale=.5]
\draw[<-] (-21,0) -- (-19,0) ;
\draw (-19,-1) rectangle (-15,1) node [pos=.5] {$Q$} ;
\draw[<-] (-15,0) -- (-13,0) ;
\node at (-12,0) {$=$} ;
\draw[<-] (-11,0) -- (-9,0) ;
\draw (-9,-1) rectangle (-5,1) node [pos=.5] {$\dec$} ;
\draw[<-] (-5,0) -- (-2,0) ;
\draw (-2,-1) rectangle (2,1) node [pos=.5] {$P$} ;
\draw[<-] (2,0) -- (5,0) ;
\draw (5,-1) rectangle (9,1) node [pos=.5] {$\enc$} ;
\draw[<-] (9,0) -- (11,0) ;
\end{tikzpicture}
\eeq

We write $P\geq Q$ if and only if such a conversion of $P$ into $Q$ exists, and this is the ordering relation in the ordered commutative monoid which models the resource theory of communication. More permissive convertibility relations between channels can be considered alternatively, such as allowing shared randomness or quantum entanglement between the encoding and decoding operations~\cite{assistedchannel}. 

We combine two channels $P:\inalph\to\outalph$ and $Q:\inalphh\to\outalphh$ to a new channel by using $P$ and $Q$ in parallel,
\beq
(P\otimes Q)(bd|ac) \defin P(b|a)\cdot Q(d|c).
\eeq
Or in pictures:
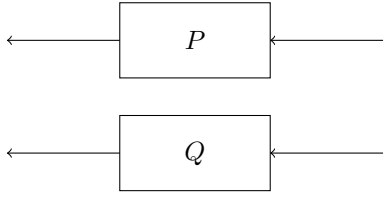
\beq
\begin{split}
\label{channelcomb}
\begin{tikzpicture}[scale=.5]
\draw[<-] (-5,2) -- (-2,2) ;
\draw (-2,1) rectangle (2,3) node [pos=.5] {$P$} ;
\draw[<-] (2,2) -- (5,2) ;
\draw[<-] (-5,-1) -- (-2,-1) ;
\draw (-2,-2) rectangle (2,0) node [pos=.5] {$Q$} ;
\draw[<-] (2,-1) -- (5,-1) ;
\end{tikzpicture}
\end{split}
\eeq
It is not hard to verify that communication channels thereby form an ordered commutative monoid whose neutral element is the trivial channel between any two one-element sets. We denote this ordered commutative monoid by $\CC$. This is the mathematical structure which captures the resource theory of communication with respect to exact conversion (as opposed to approximate). Unfortunately, stating Shannon's noisy coding theorem would require us to also consider small errors in the conversion of one channel into another. It is not currently clear to us how to capture this type of ``epsilonification'' (Remark~\ref{epsilonification}) in terms of an ordered commutative monoid, but we believe that the best way of doing so will involve a generalization of Definition~\ref{ocm}. Since no such generalization has as yet been found to recover Shannon's theorem, we disregard for now the possibility of approximate conversions between resource objects and continue to focus on the exact case.
\end{ex}

We now proceed with the general theory and introduce the analogue of Definition~\ref{orderedmap} for ordered commutative monoids.

\begin{defn}
\label{defhom}
A \emph{homomorphism} $f:A\to B$ of ordered commutative monoids is an ordered map,
\beq
x\geq y \quad \Longrightarrow \quad f(x) \geq f(y),
\eeq
which is moreover additive and preserves zero,
\beq
f(x + y) = f(x) + f(y), \qquad f(0) = 0.
\eeq
\end{defn}

\begin{rem}
If $B$ is cancellative (Definition~\ref{canc}), the requirement $f(0) = 0$ is redundant, since then one can cancel $f(0)$ on both sides of $f(0) = f(0+0) = f(0) + f(0)$.
\end{rem}

As an example, we construct a homomorphism $f:\CC\to\Grph$. By virtue of this homomorphism, some results about $\Grph$ can be translated into results about $\CC$. 

On the level of resource objects, every communication channel $P:\inalph\to\outalph$ has a \emph{distinguishability graph}. Its vertex set is the input alphabet $\inalph$, and two input symbols $a_1,a_2\in\inalph$ are defined to be adjacent if and only if they can be perfectly distinguished, meaning that for every output symbol $b\in\outalph$, we have $P(b|a_1)=0$ or $P(b|a_2)=0$. This defines the distinguishability graph $f(P)\in\Grph$. The following observation is essentially due to Shannon~\cite{zeroerror}:

\begin{prop}
\label{cctogrph}
This assignment $P\mapsto f(P)$ defines a homomorphism $f:\CC\to\Grph$.
\end{prop}

\begin{proof}
We need to check that $f$ respects both the convertibility relation and the combination operation, starting with the former. If we have a conversion between channels~\eqref{channeltrafo} which witnesses $P\geq Q$, then we define a graph map $\mathfrak{m}:f(Q)\to f(P)$ by sending each input symbol $c\in\inalphh$ to some $\mathfrak{m}(c)\in\inalph$ with $\enc(\mathfrak{m}(c)|c)>0$; the particular choice is irrelevant. This preserves adjacency due to the following reasoning: adjacency in $f(Q)$ for input symbols $c_1$ and $c_2$ means that $Q(d|c_1)=0$ or $Q(d|c_2)=0$ for all $d$. By $Q=\dec\circ P\circ\enc$, we can write this as
\beq
\sum_{a,b} \dec(d|b) P(b|a) \enc(a|c_1) = 0 \quad\textrm{ or }\quad \sum_{a,b} \dec(d|b) P(b|a) \enc(a|c_2) = 0.
\eeq
Since all terms in these sums are nonnegative, vanishing of such a sum implies that each term in the sum must vanish. Hence we have for all $d$,
\beq
\sum_b \dec(d|b) P(b|\mathfrak{m}(c_1)) = 0 \quad\textrm{ or }\quad \sum_b \dec(d|b) P(b|\mathfrak{m}(c_2)) = 0 .
\eeq
Using the same reasoning again,
\beq
\dec(d|b) P(b|\mathfrak{m}(c_1)) = 0 \quad\forall b\quad\textrm{ or }\quad \dec(d|b) P(b|\mathfrak{m}(c_2)) = 0 \quad\forall b.
\eeq
If for a given $b$ we choose $d$ with $\dec(d|b)>0$ arbitrarily, then we find that one of these products must vanish, and therefore $P(b|\mathfrak{m}(c_1))=0$ or $P(b|\mathfrak{m}(c_2))=0$. This means that $\mathfrak{m}:f(Q)\to f(P)$ is indeed a graph map. This finishes the proof of $f(P)\geq f(Q)$, which shows that $f$ respects the ordering.

Proving $f(P\otimes Q) = f(P)\ast f(Q)$ is easier: by the definition of equality~\eqref{equaldef}, it is sufficient to show that the graphs $f(P\otimes Q)$ and $f(P)\otimes f(Q)$ are isomorphic, since then in particular we have a map of graphs in each direction. First, the set of vertices is the cartesian product $\inalph\times\inalphh$ in both cases; second, adjacency of $(a_1,c_1)$ and $(a_2,c_2)$ in $f(P\otimes Q)$ means that for all $b\in\outalph$ and $d\in\outalphh$,
\beq
P(b|a_1)Q(d|c_1) = 0 \quad\textrm{ or }\quad P(b|a_2)Q(d|c_2) = 0 ,
\eeq
while adjacency in $f(P)\ast f(Q)$ means that for all $b\in\outalph$, we have $P(b|a_1)=0$ or $P(b|a_2)=0$, or for all $d\in\outalphh$, we have $Q(d|c_1)=0$ or $Q(d|c_2)=0$. These conditions are equivalent as well.
\end{proof}

Therefore we can analyze part of the structure of $\CC$ on the level of graphs. More concretely, we can make use of $f:\CC\to\Grph$ as a witness of non-convertibility: if we happen to find $f(P)\not\geq f(Q)$, then we can conclude $P\not\geq Q$. So we can think of $f$ as a $\Grph$-valued invariant of communication channels.

Returning to the general theory, ordered commutative monoids and their homomorphisms form a category $\OCM$. We can hence apply standard concepts from category theory:

\begin{defn}
\label{isom}
An \emph{isomorphism} between ordered commutative monoids $A$ and $B$ is a homomorphism $f:A\to B$ for which there exists $g:B\to A$ such that $g(f(x)) = x$ for all $x\in A$ and $f(g(y))=y$ for all $y\in B$. If an isomorphism exists, then we say that $A$ and $B$ are \emph{isomorphic}.
\end{defn}

\begin{prop}
\label{isocrit}
A homomorphism $f:A\to B$ is an isomorphism if and only if
\begin{enumerate}
\item\label{fullyfaithful} $f$ reflects the ordering,
\beq
f(x) \geq f(y) \quad\Longrightarrow\quad x \geq y,
\eeq
\item\label{essentiallysurjective} and for every $z\in B$ there is $x\in A$ with $f(x)=z$.
\end{enumerate}
\end{prop}

Condition~\ref{fullyfaithful} means in particular that $f$ is injective, i.e.~reflects equality: if $f(x)=f(y)$, then also $x=y$.

\begin{proof}
If we have an isomorphism as in Definition~\ref{isom}, then $f$ reflects the ordering since $f(x) \geq f(y)$ implies that
\beq
x = g(f(x)) \geq g(f(y)) = y.
\eeq
Moreover, for every $z\in B$ we have $f(g(z))=z$, which is~\ref{essentiallysurjective}.

Conversely, suppose that $f$ satisfies the hypotheses of the proposition. Then for every $z\in B$, we define $g(z)$ by choosing some $x\in A$ with $f(x)=z$. Then $f(g(z))=z$ holds by definition, while $g(f(x))=x$ follows from $f(g(f(x)))=f(x)$ and cancelling the outer application of $f$, which we can do since $f$ reflects equality. Finally, it remains to be shown that $g$ is a homomorphism. First, $g$ preserves the ordering since if we have $y\geq z$ in $B$, then we also have $f(g(y)) \geq f(g(z))$, and therefore $g(y)\geq g(z)$ by~\ref{fullyfaithful}. Similarly, we can show that
\beq
f(g(y+z)) = y+z = f(g(y)) + f(g(z)) = f(g(y) + g(z)),
\eeq
and therefore $g(y+z) = g(y) + g(z)$.
\end{proof}

Next we will encounter a simple example of an isomorphism. As in Remark~\ref{commmon}, we consider a commutative monoid as a special kind of ordered commutative monoid.

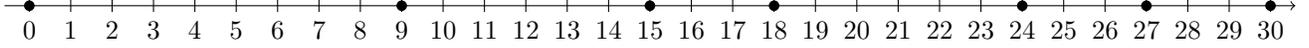
\begin{figure}
\begin{tikzpicture}[scale=.55]
\tikzstyle{point}=[fill,circle,inner sep=0pt,minimum size=4pt]
\draw[->] (-.6,0) -- (30.6,0) ;
\foreach \x in {0,...,30}
	\draw[xshift=\x cm] (0,-.15) node[below] {$\x$} -- (0,.15) ;
\foreach \x in {0,3,5,6,8,9,10}
	\node[point] at (3*\x,0) {} ;
\end{tikzpicture}
\caption{An example of a commutative submonoid of $\N$. Starting with $24$, it contains all natural numbers divisible by $3$.}
\label{fignumericalex}
\end{figure}

\begin{ex}
\label{numsg}
Motivated by Remark~\ref{annihilator}, we may consider general subsets $S\subseteq \N$ that contain $0$ and are closed under addition. This is an interesting and non-trivial class of commutative monoids. As a concrete example, consider the set
\beqn
\label{numericalex}
\{\: 0,9,15,18,24,27,30,\ldots \:\},
\eeqn
where at $24$ and after, the set is defined to contain all multiples of $3$; see Figure~\ref{fignumericalex} for an illustration. This set contains $0$ and is closed under addition, and hence forms a commutative submonoid of $\N$. Those commutative submonoids $S\subseteq\N$ for which the elements of $S$ are coprime, i.e.~for which $\gcd(S)=1$ holds, are known as \emph{numerical semigroups}~\cite{numericalsemigroup}. Possibly surprisingly, the condition $\gcd(S)=1$ is equivalent to $\N\setminus S$ being finite~\cite[Lemma~2.1]{numericalsemigroup}. The example~\eqref{numericalex} has $\gcd(S)=3$, and hence is not a numerical semigroup.

Any commutative monoid $S\subseteq\N$, and in particular any of the annihilators $\ann(x,y)$ from Remark~\ref{annihilator}, is either trivial, i.e.~$S\cong\{0\}$, or isomorphic to a numerical semigroup as follows. The greatest common divisor of all the elements of $S$,
\beq
d \defin \gcd(S) ,
\eeq
is a positive integer dividing every element of $S$. We think of $d$ as the generator of the principal ideal in $\Z$ which is generated by the ideal $S\subseteq\N$. Then,
\beq
S' \defin \{ \: n\in \N \:|\: dn\in S \:\}
\eeq
is a numerical semigroup, since $\gcd(S')=1$ holds by construction. In the example~\eqref{numericalex}, we obtain
\beq
S' = \{\: 0,3,5,6,8,9,10,\ldots \:\},
\eeq
and now indeed $\N\setminus S'=\{1,2,4,7\}$ is finite. In general, the map
\beqn
\label{ssprime}
S' \lra S, \qquad n\longmapsto dn
\eeqn
is an isomorphism of commutative monoids, and hence the $S$ is isomorphic to a numerical semigroup~\cite[Proposition~2.2]{numericalsemigroup}. In particular, every annihilator from Remark~\ref{annihilator},
\beq
\ann(x,y) \defin \{ \: n\in \N \:|\: nx\geq ny\: \} ,
\eeq
is isomorphic to a numerical semigroup. In this sense, studying the possible behaviours of the many-copy convertibility relation $nx\geq ny$ as a function of $n$ is equivalent to studying numerical semigroups.
\end{ex}

The real numbers $\R$ form an ordered commutative monoid with respect to the usual ordering and addition. As indicated at the end of Section~\ref{prelims}, we are particularly interested in homomorphisms of the form $A\to\R$, which correspond to consistent assignments of values or prices to each resource object. Due to this special significance, we introduce an abbreviated term for these homomorphisms, indicative of our upcoming considerations of ordered vector spaces.

\begin{defn}
\label{deffunc}
A \emph{functional} $f$ on an ordered commutative monoid $A$ is a homomorphism $f:A\to\R$.
\end{defn}

In resource-theoretic parlance, the functionals are precisely the \emph{additive monotones}. In the resource theories studied in quantum information theory, many additive monotones are given by quantities closely related to R\'enyi entropies~\cite{nonuniformity}. Additivity is a useful property for a monotone to have, since computing the value of the monotone on a composite resource object reduces to computing the value of each part. For example, if $f$ is an additive monotone and $x+y\geq z$, then we can conclude $f(x)+f(y)\geq f(z)$; hence the latter inequality is a necessary condition for $x+y\geq z$. The usefulness of additivity here lies in the fact that we only need to compute $f(x)$ and $f(y)$ individually, without making possibly more complicated considerations about composite resource objects (although these may be relevant for \emphalt{proving} additivity in the first place).

Moreover, if we have a functional $f$ which also reflects the ordering as in~\eqref{ff}, then we obtain an embedding of our ordered commutative monoid inside $\R$, in the sense that \emphalt{both} the ordering and the addition are given by that of real numbers. A similar statement applies to families of additive monotones, in a sense analogous to~\eqref{sepfamily}. We will get to this in Section~\ref{sectonedim}.

\begin{rem}
\label{multifun}
It is sometimes advantageous to consider monotones with a multiplicative rather than additive scaling in the following sense. The positive real numbers $\Rpos$ form an ordered commutative monoid with respect to multiplication and the usual ordering. Therefore we can also consider \emph{multiplicative functionals} or \emph{multiplicative monotones} in the form of homomorphisms $f:A\to\Rpos$. Instead of additivity, such an $f$ satisfies the multiplicativity equation
\beq
f(x+y) = f(x)f(y).
\eeq
This kind of functional is natural especially in those situations in which the combination operation ``$+$'' itself is of a multiplicative character. 

With respect to any base, the logarithm
\beqn
\label{log}
\log : \Rpos \to \R,
\eeqn
is an isomorphism of ordered commutative monoids whose inverse is the exponential function with respect to the same base. We can translate back and forth between additive and multiplicative functionals using these functions. Due to this correspondence, we focus on additive functionals and do not consider the multiplicative case separately.
\end{rem}

For example in $\Grph$ (Example~\ref{graphs}), the \emph{Lov\'asz number} $\vartheta$---one of the most widely studied graph invariants---is essentially a multiplicative functional on $\Grph$. The qualification ``essentially'' refers to the fact that for technical reasons, we need to work in terms of the ``complementary Lov\'asz number'', which is the graph invariant assigning to every graph $\mathcal{G}$ the Lov\'asz number $\vartheta(\overline{\mathcal{G}})$ of its complement graph $\overline{\mathcal{G}}$. In order to have less confusing notation, we denote this by $\lovasz(\mathcal{G})$.

Among the various equivalent definitions of $\lovasz$~\cite{Lovasz,sandwich}, we choose
\beq
\lovasz(\mathcal{G}) \defin \max \sum_v |\langle \psi,\phi_v\rangle|^2 ,
\eeq
where the maximum is taken over all families of vectors $(\phi_v)$ and $\psi$ in a sufficiently high-dimensional inner product space with $\langle\phi_v,\phi_v\rangle \leq 1$ and $\langle\psi,\psi\rangle\leq 1$, such that the $(\phi_v)$ are indexed by the vertices $v\in \mathcal{G}$ and assumed to satisfy the orthogonality constraint
\beq
v \sim w \quad\Longrightarrow\quad \langle\phi_v,\phi_w\rangle = 0.
\eeq

\begin{prop}
\label{lovprop}
This gives a multiplicative functional $\lovasz:\Grph\to\Rpos$ in the sense of Remark~\ref{multifun}.
\end{prop}

\begin{proof}
This is well-known. That $\lovasz$ is monotone under maps of graphs has been noticed e.g.~in~\cite{SvsH}. The multiplicativity is a standard result, although it is typically stated without taking complement graphs: the Lov\'asz number of a strong product of graphs is equal to the product of the individual Lov\'asz numbers~\cite[Theorem 7]{Lovasz},~\cite[Section 20]{sandwich}.
\end{proof}

Other multiplicative functionals on $\Grph$ include the fractional chromatic number (Example~\ref{moregraphinvariants}) and the projective rank~\cite{projrank}. For the fractional chromatic number, monotonicity is a simple consequence of~\cite[Proposition~3.2.1]{fgt}, while multiplicativity is~\cite[Corollary~3.4.2]{fgt}. For the projective rank, monotonicity is a special case of~\cite[Theorem~6.13.4]{graphhoms}, while multiplicativity is~\cite[Theorem~27]{entass}.

For another bit of general theory, we generalize another standard notion of algebra from the unordered setting to our ordered setup: the notion of \emph{generators} of an algebraic structure.

\begin{defn}
\label{generatingpair}
A \emph{generating pair} $(g_+,g_-)$ is a pair of elements $g_+,g_-\in A$ with $g_+\geq g_-$ such that for every $x\in A$ there is $n\in\N$ with
\beq
ng_+ \geq x + ng_- \quad\textrm{ and }\quad ng_+ + x \geq ng_-.
\eeq
\end{defn}

We cannot claim that this is necessarily the only or even the best definition of what it means to ``generate'' an ordered commutative monoid. Nevertheless, the existence of a generating pair is crucial for many of the technical results that we will prove in the remainder of this paper. We can restate the main condition in an equivalent manner by postulating that for all $x,y\in A$, there is $n\in\N$ with
\beq
ng_+ + x \geq y + ng_-.
\eeq

From a resource-theoretic perspective, a generating pair is a pair of \emph{universal resources} in the sense that every other resource $x$ can both be created, as in $ng_+\geq x + ng_-$, and be absorbed, as in $ng_+ + x \geq ng_-$, using a sufficiently big number of copies of the universal pair. In many resource theories of practical interest, such a universal pair exists.

\begin{rem}
\label{posgen}
It sometimes happens that $x\geq 0$ for all $x\in A$ (Definition~\ref{defpos}). In this case, it is easy to see that if $(g_+,g_-)$ is a generating pair, then so is $(g_+,0)$. Hence a generating pair exists if and only if there is $g_+\in A$ such that for every $x\in A$ there exists $n\in\N$ with
\beq
n g_+ \geq x.
\eeq
In this situation, we say that $g_+$ is a \emph{generator}.
\end{rem}

\begin{ex}
\label{grphgenpair}
In $\Grph$, we indeed have $x\geq 0$ for all $x$. There is a generator given by $g_+\defin \mathcal{K}_2$, the complete graph on two vertices. Every graph $\mathcal{G}$ is a subgraph of some $\mathcal{K}_{2^n} = \mathcal{K}_2^{\ast n} = n \mathcal{K}_2$, and with this $n$ we obtain $\mathcal{K}_{2^n} = n \mathcal{K}_2 \geq \mathcal{G}$.
\end{ex}

A generating pair often arises as part of an entire \emphalt{family} of pairs $(g_+^n,g_-^n)$ indexed by a size parameter $n\in\Npos$ in such a way that the universal family suitably scales as a function of $n$. This is captured by the following definition:

\begin{defn}
\label{generatingfamily}
A \emph{family of generating pairs} is a family of pairs $(g_+^n,g_-^n)_{n\in\Npos}$ with $g_+^n\geq g_-^n$, and
\beq
g_+^{mn} = g_+^m + g_+^n,\qquad g_-^{mn} = g_-^m + g_-^n,
\eeq
and such that for every $x$ there is $n\in\Npos$ with
\beqn
\label{genk}
g_+^n + x \,\geq\,  g_-^n \quad\textrm{ and }\quad g_+^n \,\geq\, x + g_-^n.
\eeqn
\end{defn}

Every pair $(g_+^n,g_-^n)$ in a family of generating pairs is indeed itself a generating pair. The main interest and usefulness of such families lies in considering the \emphalt{smallest} $n$ for which~\eqref{genk} holds, and this $n$ can be understood as a measure of the ``size'' of $x$. In general, one can consider the two parts of~\eqref{genk} separately and obtain two kinds of ``size'' measures. This kind of idea has been used e.g.~in nonequilibrium thermodynamics as the notion of ``entropy meter''~\cite{entropymeter}. But we can also apply it in many other contexts, such as to $\Grph$:

\begin{ex}
\label{grphex}
$\Grph$ has a family of generating pairs given by
\beq
g_+^n \defin \mathcal{K}_n ,\qquad g_-^n \defin 0 ,
\eeq
where $\mathcal{K}_n$ is the complete graph on $n$ vertices. We have indeed $g_+^{mn} = g_+^m + g_+^n$ since $\mathcal{K}_{mn} = \mathcal{K}_m \ast \mathcal{K}_n$, and moreover $g_-^{mn} = 0 = g_-^m + g_-^n$. But most importantly, every graph $\mathcal{G}$ satisfies $\mathcal{K}_{|\mathcal{G}|}\geq \mathcal{G}+0$, where $|\mathcal{G}|$ is the number of vertices of $\mathcal{G}$, and trivially also $\mathcal{K}_{|\mathcal{G}|} + \mathcal{G} \geq 0$.

It is of interest to consider the smallest $n$ for which $\mathcal{K}_n\geq \mathcal{G}$, i.e.~the smallest $n$ for which $\mathcal{K}_n$ can ``simulate'' $\mathcal{G}$ in the sense that there exists a graph map $\mathcal{G}\to \mathcal{K}_n$. In graph-theoretic parlance, this $n$ is commonly known as the \emph{chromatic number} of $\mathcal{G}$ and denoted $\chi(\mathcal{G})$. Similarly, we may be interested in the largest $m\in\Npos$ for which $\mathcal{G}\geq \mathcal{K}_m$, or for which $\mathcal{G}$ can ``simulate'' $\mathcal{K}_m$ in the sense of a graph map $\mathcal{K}_m\to \mathcal{G}$. In graph-theoretic terminology, this $m$ is commonly known as the \emph{clique number} of $\mathcal{G}$ and denoted $\omega(\mathcal{G})$.

If $\chi(\mathcal{G})=\omega(\mathcal{G})$, then we have $\mathcal{K}_{\chi(\mathcal{G})}\geq \mathcal{G}\geq \mathcal{K}_{\omega(\mathcal{G})}$, and therefore $\mathcal{G}=\mathcal{K}_{\chi(\mathcal{G})}$ as elements of $\Grph$, and hence even $\mathcal{G} = \mathcal{K}_{\chi(\mathcal{G})}$.\footnote{It may seem strange to write $\mathcal{G}=\mathcal{K}_{\chi(\mathcal{G})}$, although $\mathcal{G}$ itself may not actually be a complete graph. But conceptually, this is no different from how algebraic expressions can be written in different forms, such as $(x+y)^2 + (x-y)^2 = x^2 + 2xy + y^2 + x^2 - 2xy + y^2 = 2x^2 + 2y^2$, where we use the equality symbol as a way of denoting a certain kind of equivalence.} The \emph{perfect graphs}\footnote{Among several equivalent definitions, $\mathcal{G}$ is perfect if neither $\mathcal{G}$ itself nor its complement contains an induced odd cycle~\cite{perfect}.} are an interesting class of graphs with this property.
\end{ex}

In summary, upon considering the structure of $\Grph$ as an ordered commutative monoid together with a family of generating pairs, we have found that two of the most important graph invariants arise very naturally. While this is a relatively simple and well-known fact~\cite{graphhoms}, we nevertheless hope that it elucidates the potential of our ideas. For example, we can now rederive Lov\'asz's \emph{sandwich theorem} in an abstract and conceptually clear manner:

\begin{thm}[\cite{Lovasz,sandwich}]
For every graph $\mathcal{G}\in\Grph$, we have
\beq
\omega(\mathcal{G}) \leq \lovasz(\mathcal{G}) \leq \chi(\mathcal{G}).
\eeq
\end{thm}

Strictly speaking, this is only a part of the sandwich theorem, which also crucially comprises a complexity statement. As the proof shows, the theorem also holds with any other multiplicative functional $f$ in place of $\lovasz$ normalized as $f(\mathcal{K}_2)=2$.

\begin{proof}
Besides the above definitions of $\omega$ and $\chi$, we only need Proposition~\ref{lovprop} together with the normalization $\lovasz(\mathcal{K}_2) = 2$ and the fact that the $\mathcal{K}_n$ form a family of generating pairs as above.

Since $\lovasz$ is a multiplicative functional, we have $\lovasz(\mathcal{K}_{mn}) = \lovasz(\mathcal{K}_m) \lovasz(\mathcal{K}_n)$. Together with monotonicity, this implies that $\lovasz(\mathcal{K}_n) = n^r$ for some constant $r\in\Rplus$ due to a result of Erd\H{o}s~\cite{erdos}. The normalization condition $\lovasz(\mathcal{K}_2)=2$ requires $r=1$, and hence
\beqn
\label{lovkn}
\lovasz(\mathcal{K}_n) = n
\eeqn
for all $n$. Now the first inequality follows easily,
\beq
\omega(\mathcal{G}) = \lovasz(\mathcal{K}_{\omega(\mathcal{G})}) \leq \lovasz(\mathcal{G}).
\eeq
where the first step is by~\eqref{lovkn} and the second follows from $\mathcal{G}\geq \mathcal{K}_{\omega(\mathcal{G})}$. The proof of the other inequality is similar.
\end{proof}

From here on, there would be many different routes that our investigations could take. But for the present work, we would like to focus on the following question as a guiding principle: if we have
\beqn
\label{Mord}
f(x) \geq f(y) \textrm{ for all functionals } f,
\eeqn
then does this imply that $x\geq y$? The answer to this question will easily be seen to be negative, and hence we will consider the question in a revised form: what does the assumption~\eqref{Mord} tell us about the convertibility of $x$ and $y$? Is there a suitably relaxed or regularized notion of convertibility which can be defined without reference to functionals and is equivalent to~\eqref{Mord}? The answer to this revised question will turn out to be positive.

To summarize, our main question is:

\mainqstn

Next we will see a simple reason for why the answer is not just ``it tells us that $x\geq y$''.

\newpage
\section{The catalytic regularization: ordered abelian groups}
\label{sectoag}

A very common phenomenon in resource theories is that of \emph{catalysis}: there may be $x,y,z\in A$ such that $x\not\geq y$, but
\beqn
\label{catalytic}
x+z\geq y+z.
\eeqn
In this situation, we say that $z$ is a \emph{catalyst} for the conversion of $x$ into $y$. The possibility of catalysis in chemistry is arguably among the most fundamental insights into chemistry as a resource theory.

\begin{ex}
In chemistry, hydrogen peroxide decomposes into water and oxygen only rather slowly, 
\beqn
\label{hydroperox}
\mathrm{2H_2O_2} \stackrel{\textrm{slow}}{\lra} \mathrm{2H_2O} + \mathrm{O_2} .
\eeqn
But in the presence of manganese dioxide, the reaction is much quicker,
\beq
\mathrm{2H_2O_2} + \mathrm{MnO_2} \stackrel{\textrm{fast}}{\lra} \mathrm{2H_2O} + \mathrm{O_2} + \mathrm{MnO_2}.
\eeq
Hence manganese dioxide catalyzes the decomposition of hydrogen peroxide into water and oxygen. So if we define the ordered commutative monoid $\Chem$ such that the second reaction is considered viable but the first one is not, then manganese dioxide $\mathrm{MnO_2}$ is a catalyst also in the sense of~\eqref{catalytic} for the conversion of hydrogen peroxide into water and oxygen,
\beq
\mathrm{2H_2O_2} \not\geq \mathrm{2H_2O} + \mathrm{O_2} ,\qquad \mathrm{2H_2O_2} + \mathrm{MnO_2} \geq \mathrm{2H_2O} + \mathrm{O_2} + \mathrm{MnO_2}.
\eeq

In chemistry, there are also \emphalt{inhibitors}, which can decrease the rate at which a reaction happens when added to the reactants. For example, the presence of acetanilide $\mathrm{C_8H_9NO}$ slows down the above reaction~\eqref{hydroperox},
\beqn
\label{inhibit}
\mathrm{2H_2O_2 + C_8H_9NO} \stackrel{\textrm{very slow}}{\lra} \mathrm{2H_2O + O_2 + C_8H_9NO} .
\eeqn
Now even if one considers the reaction~\eqref{hydroperox} to be a viable basic conversion in $\Chem$, but not~\eqref{inhibit}, in the ordered commutative monoid $\Chem$ we still have
\beqn
\label{inhibitconvert}
\mathrm{2H_2O_2 + C_8H_9NO} \geq \mathrm{2H_2O + O_2 + C_8H_9NO} .
\eeqn
This is due to our definition of the ordering in Example~\ref{rtchem}: we imagine that the boxes that contain the molecules can be brought together at will, causing their inhabitants to react in exactly the way desired by the operator. With this in mind, the conversion~\eqref{inhibitconvert} can be achieved by taking the two hydrogen peroxide molecules on the left-hand side out of their boxes and waiting until they have decomposed into water and oxygen; the acetanilide inhibitor remains in its box.
\end{ex}

\begin{ex}
\label{tools}
In everyday life, catalysts are usually known as \emph{tools}. They enable conversions like this:
\beq
\textrm{timber} + \textrm{nails} \not\geq \textrm{table},\qquad \textrm{timber} + \textrm{nails} + \textrm{saw} + \textrm{hammer} \geq \textrm{table} + \textrm{saw} + \textrm{hammer}.
\eeq
So ``$\textrm{saw} + \textrm{hammer}$'' is the catalyst which enables the conversion of ``$\textrm{timber} + \textrm{nails}$'' into ``$\textrm{table}$''.
\end{ex}

On the mathematical level of ordered commutative monoids, catalysis is among the simplest phenomena concerned with how the ordering relation and the addition operation interact.

If $f$ is any functional (additive monotone), then~\eqref{catalytic} implies that $f(x)+f(z)\geq f(y)+f(z)$, and therefore $f(x)\geq f(y)$. Hence in order for $f(x)\geq f(y)$ to hold for all functionals $f$, it is sufficient that $x$ is convertible into $y$ \emphalt{catalytically}; in particular, it does generally not mean that $x\geq y$. This provides a simple partial answer to Question~\ref{mainqstn}~\cite[p.~31]{nonuniformity}.

In some resource theories, catalytic convertibility does imply convertibility without a catalyst, and then genuine catalysis is impossible. In mathematical terms:

\begin{defn}[{cf.~\cite[Definition~4.10]{resourcesI}}]
\label{canc}
An ordered commutative monoid $A$ is \emph{cancellative} if $x+z\geq y+z$ implies $x\geq y$.
\end{defn}

We have borrowed this term from algebra, where cancellativity of a commutative monoid means that $x+z=y+z$ implies $x=y$. If an ordered commutative monoid $A$ is cancellative in our sense, then it is also cancellative in this unordered sense: $x+z=y+z$ means that $x+z\geq y+z$ and $y+z\geq x+z$, which yields $x\geq y$ and $y\geq x$ by cancellativity, and therefore $x=y$. So in light of Remark~\ref{commmon}, our definition of cancellativity generalizes the usual notion of cancellativity to the ordered setting. This has been done before in~\cite{separativeosgs,seppom}.

Other names for the concept of Definition~\ref{canc} have been used. For example,~\cite{strongsemi} uses the term ``strong''. In~\cite{resourcesI}, we have used the term ``catalysis-free''.

For an example of a resource theory whose ordered commutative monoid turns out to be cancellative, see~\cite[Theorem~25]{asymmetry}.

\begin{ex}
\label{grphncanc}
It is challenging to figure out whether $\Grph$ is cancellative. A computer search using \textsc{Sage} has failed to produce any counterexample. Nevertheless, by anticipating some of the upcoming theory, we can construct a counterexample of moderate size.

The non-cancellativity follows from the upcoming Example~\ref{grphntf} together with Lemma~\ref{distribute}\ref{cancimpltf}, where we take the alternative binary operation $\lor$ on $\Grph$ to be given by the \emph{graph join}, which is the disjoint union of two graphs together with all edges connecting the two components. We leave it to the reader to verify that the graph join also makes $\Grph$ into an ordered commutative monoid, and that the disjunctive product ``$+$'' distributes over the join ``$\lor$''.

This argument can be turned into a concrete example as follows. Let $\pentagon$ denote the \emphalt{5-cycle}, which is the graph that just looks like a pentagon $\pentagon$. Example~\ref{grphntf} explains that $6\pentagon\geq 2\mathcal{K}_{11}$, although $3\pentagon\not\geq \mathcal{K}_{11}$. Plugging this into the proof of Lemma~\ref{distribute} results in the catalyst being given by~\eqref{zconstruction},
\beq
\mathcal{G} = 3\pentagon \lor \mathcal{K}_{11},
\eeq
i.e.~the graph which is the graph join of $\pentagon^{\ast 3}\defin\pentagon\ast\pentagon\ast\pentagon$ with the complete graph on $11$ vertices. The proof of Lemma~\ref{distribute} then shows that we have $3\pentagon + \mathcal{G} \geq \mathcal{K}_{11} + \mathcal{G}$, although $3\pentagon \not\geq \mathcal{K}_{11}$.
\end{ex}

Getting back to the general theory, we can make any ordered commutative monoid into a cancellative one by introducing a new ordering denoted ``$\succeq$'' and defined as
\beq
x\succeq y \quad :\Longleftrightarrow\quad \exists z\in A, \: x+z\geq y+z.
\eeq
If $A$ is the ordered commutative monoid describing a resource theory, then ``$\succeq$'' is the \emph{catalytic ordering} in which $x$ is declared convertible into $y$ if there exists a catalyst which facilitates the conversion. It is straightforward to show that $\succeq$ is indeed reflexive and transitive, and that this ordering coincides with the original ``$\geq$'' if $A$ already happened to be cancellative. Moreover, the addition $+$ is clearly monotone with respect to $\succeq$ as well. If we write $\canc(A)$ for the ordered commutative monoid given by $A$ equipped with this new ordering, then $\canc(A)$ is cancellative: assuming $x+z\succeq y+z$ means that there exists $w\in A$ with $x+z+w\geq y+z+w$ in the original ordering, which indeed implies $x\succeq y$. There is a canonical homomorphism $\eta_A:A\to \canc(A)$ which simply maps every element to itself, and this has the following universal property: if $B$ is any other ordered commutative monoid which is cancellative, and $f:A\to B$ is any homomorphism, then there is a unique homomorphism $\canc(f):\canc(A)\to B$ such that $f = \canc(f)\circ \eta_A$, i.e.~such that the diagram
\beq
\xymatrix{ A \ar[rr]^{\eta_A} \ar[dr]_f & & \canc(A) \ar@{-->}[dl]^{\canc(f)} \\
 & B }
\eeq
commutes. In categorical language, this universal property states that the category of cancellative ordered commutative monoids is a reflective subcategory of $\OCM$. 

We refer to~\cite[Proposition~B.3]{secondlaws} for an example of a successful characterization of the catalytic ordering in a concrete resource theory.

\begin{rem}
\label{cateq}
The induced equality relation on $\canc(A)$ may differ from the one on $A$, meaning that $\eta_A(x)=\eta_A(y)$ may not necessarily imply that $x=y$. For example, if there is an element $a\in A$ which is \emph{absorbing} in the sense that $x+a=a$ for all $x\in A$, then this equation $x+a=a$ will also hold in $\canc(A)$, and cancellativity then results in $x=0$ in $\canc(A)$. Therefore $\canc(A)$ is isomorphic to the trivial commutative monoid in which all elements are equal to $0$. A similar thing happens if $A$ is a (semi-)lattice and ``$+$'' is given by the meet or join operation.
\end{rem}

There is a particularly interesting class of ordered commutative monoids that are automatically cancellative:

\begin{defn}
\label{oag}
An \emph{ordered abelian group} $G$ is an ordered commutative monoid for which the monoid structure is a group, i.e.~such that for every $x\in G$ there exists a $(-x)\in G$ with $x+(-x)=0$.
\end{defn}

As before, we assume that the equality relation ``$=$'' on $G$ is the one induced from the ordering relation. We are interested in ordered abelian groups at this point because of the intimate relation to cancellativity. First, every ordered abelian group is automatically cancellative: $x+z \geq y+z$ implies $x+z+(-z) \geq y+z+(-z)$, and therefore $x\geq y$~\cite[Theorem~1]{strongsemi}.

Moreover, for any $x$ the inverse $(-x)$ is necessarily unique: if there was another inverse $(-x)'$, then we would get $x+(-x) = 0 = x+(-x)'$, and therefore $(-x) = (-x)'$ by cancellativity. This justifies removal of the brackets and writing $-x$ for \emphalt{the} negative inverse of $x$, and using the usual rules for how $+$ and $-$ interact. If $G$ is an ordered abelian group and $A$ any ordered commutative monoid, then any homomorphism $f:G\to A$ preserves negatives in the sense that $f(-x) = -f(x)$, even if $A$ does not have negatives for all of its elements. This is because $f(-x) + f(x) = f(-x+x) = f(0) = 0$, so that $f(-x)$ is an additive inverse of $f(x)$.

We can turn every cancellative ordered commutative monoid $A$ into an ordered abelian group $\oag(A)$ by formally throwing in negative inverses as follows. We define the elements of $\oag(A)$ to be pairs of elements of $A$,
\beq
\oag(A) \defin A\times A,
\eeq
and we think of a pair $(x,y)\in \oag(A)$ as representing a formal difference $x-y$. We declare two pairs $(x,y)$ and $(x',y')$ to be ordered $(x,y)\geq (x',y')$ if and only if\footnote{This definition corresponds to the ``generalized ordering'' of Lieb and Yngvason~\cite[p.~23]{secondlaw1}.}
\beqn
\label{oagord}
x + y' \geq x' + y
\eeqn
holds in $A$. Finally, we define the addition operation on pairs componentwise, so that
\beq
(x,y)+(x',y')\defin (x+x',y+y').
\eeq
It is straightforward to check that this satisfies Definition~\ref{oag} with neutral element $(0,0)\in \oag(A)$ and inverses $-(x,y) \defin (y,x)$. All of this should be intuitive if one keeps the intended interpretation of $(x,y)$ as a formal difference $x-y$ in mind.

Cancellativity of $A$ is relevant for proving that the ordering~\eqref{oagord} is transitive: the assumption $(x_1,y_1)\geq (x_2,y_2)\geq (x_3,y_3)$ means that
\beq
x_1 + y_2 \geq x_2 + y_1,\qquad x_2 + y_3 \geq x_3 + y_2,
\eeq
and this implies
\beq
x_1 + y_3 + y_2 \geq x_3 + y_1 + y_2.
\eeq
In order to conclude $(x_1,y_1)\geq (x_3,y_3)$, which corresponds to $x_1 + y_3\geq x_3 + y_1$, we need to appeal to cancellativity. The other axioms of an ordered abelian group are straightforward to verify and do not rely on the cancellativity of $A$.

There is a canonical homomorphism $\gamma_A:A\to \oag(A)$ which maps every element $x\in A$ to the pair $(x,0)\in \oag(A)$, and this homomorphism has the following universal property: if $G$ is any other ordered abelian group and $f:A\to G$, then there is a unique $\oag(f):\oag(A)\to G$ such that $\oag(f)\circ \gamma_A = f$, i.e.~such that the diagram
\beq
\xymatrix{ A \ar[rr]^{\gamma_A} \ar[dr]_f & & \oag(A) \ar@{-->}[dl]^{\oag(f)} \\
 & G }
\eeq
commutes. The uniqueness part follows from
\beq
\oag(f)( (x,y) ) = \oag(f)( (x,0) - (y,0) ) = \oag(f) ( (x,0) ) - \oag(f)( (y,0) ) = f(x) - f(y).
\eeq
A simple computation shows that defining $\oag(f)$ by this equation indeed yields a homomorphism, so this also proves existence. In categorical language, this universal property states that the category of ordered abelian groups, which we denote $\OAG$, is a reflective subcategory of the category of cancellative ordered commutative monoids. We also note that by~\eqref{oagord}, $\gamma_A$ automatically reflects the ordering,
\beq
\gamma_A(x) \geq \gamma_A(y) \quad \Longrightarrow \quad x\geq y,
\eeq
and in particular $\gamma_A$ must be injective. Hence we can regard $A$ as a submonoid of $\oag(A)$ carrying the induced ordering~\cite[Theorem~5]{strongsemi},~\cite[Proposition~5.2]{commsemigroups}. Appendix~\ref{enrichment} explains how this construction is a special case of a standard result of category theory.

In resource-theoretic terms, we interpret $\oag(A)$ as follows. If a resource theory does not allow genuine catalysis---or has already been regularized such that the convertibility relation is catalytic convertibility---then we can formally introduce resource objects $-x$ which stand for ``borrowing'' the resource object $x\in A$ from somewhere else, or for ``going into debt'' with respect to $x$. Since catalysis is impossible, we cannot achieve any additional conversions by first borrowing $x$, conducting a resource conversion, and then returning $x$, than without having borrowed $x$ in the first place. Hence introducing $-x$ does not modify the convertibility relation, and we gain the mathematical convenience of working with a group rather than a monoid. We may think of a formal difference $x-y$ of resource objects in terms of a bank account: we have a credit of $x$ at our disposal, but also a debt of $y$ which will eventually have to paid off. But just as having a resource $x$ as credit may not necessarily be beneficial---as the nuclear waste example shows---having a debt of $y$ is not necessarily unfavourable.

So if we start with an ordered commutative monoid $A$ which is not necessarily cancellative, we can make it cancellative first by constructing $\canc(A)$, and then turn $\canc(A)$ into an ordered abelian group $\oag(\canc(A))$. In resource-theoretic terms, this ordered abelian group is the \emph{catalytic regularization} of the original theory. There is a canonical homomorphism $\gamma_A \eta_A : A \to \oag(\canc(A))$ having the following universal property:  if $G$ is any ordered abelian group and $f:A\to G$, then there is a unique homomorphism $\oag(\canc(f)):\oag(\canc(A))\to G$ such that $\oag(\canc(f))\circ \gamma_A\eta_A = f$, i.e.~such that the diagram
\beq
\xymatrix{ A \ar[rr]^{\gamma_A\eta_A} \ar[dr]_f & & \oag(\canc(A)) \ar@{-->}[dl]^{\oag(\canc(f))} \\
 & G }
\eeq
commutes. In the unordered setting, this construction of turning a commutative monoid to an abelian group is commonly known as forming the \emphalt{Grothendieck group}\footnote{Not to be confused with the \emphalt{Grothendieck construction} from category theory, which doesn't have much to do with it.}, and is fundamental e.g.~for the development of $K$-theory~\cite[Section~II.1]{KtheoryTop},~\cite[Section~3.1]{KtheoryCstar}. While the conventional construction of the Grothendieck group requires taking the quotient of $A\times A$ with respect to an equivalence relation, this is implicit in our ordered setting by the definition of equality in terms of the ordering.

We could as well have performed both steps of this construction in one go by directly going from the ordered commutative monoid to the ordered abelian group. However, we find it instructive to split the construction into two separate parts which are simpler to understand individually. We nevertheless abbreviate $\oag(\canc(A))$ by $\oag(A)$, in order not to clutter notation too much.

In category-theoretic terms, we can state the universal property like this:

\begin{thm}
\label{ocmtooag}
$\OAG$ is a reflective subcategory of $\OCM$, i.e.~the canonical inclusion functor $\OAG \hookrightarrow \OCM$ has a left adjoint.
\end{thm}

The relevance of this result for Question~\ref{mainqstn} is that $\R$ is an ordered abelian group, and therefore any functional on $A$ can be uniquely extended to a functional on $\oag(A)$. In particular, in order to have $f(x)\geq f(y)$ for all functionals $f$, it is sufficient that $x\geq y$ in $\oag(A)$, i.e.~that $x$ can be converted into $y$ catalytically. But is $x\geq y$ in $\oag(A)$ also \emphalt{necessary} for $f(x)\geq f(y)$ for all $f$? In resource-theoretic terms, do additive monotones witness catalytic convertibility perfectly? We now know that this is really a question about ordered abelian groups. By adding or subtracting $y$ on both sides, an inequality of the form $x\geq y$ is equivalent to $x-y\geq 0$, and hence it is sufficient to consider the case that the element on the right-hand side is zero:

\begin{qstn}
\label{revisedmainqstn}
If $G$ is an ordered abelian group and $f(x)\geq 0$ for all functionals $f$, then does this imply that $x\geq 0$? If not, what does it tell us?
\end{qstn}

In the next section, we will see that the answer is again negative for a simple reason, and another step of regularization will be performed.

There is an alternative definition of ordered abelian group which we find useful. Since $x\geq y$ is equivalent to $x-y\geq 0$, it is sufficient to specify which elements are $\geq 0$ in order to specify the ordering relation on $G$ completely. We denote this class of elements by $G_+$ and call it the \emph{positive cone} of $G$. We can then also rewrite $x\geq y$ equivalently as $x-y\in G_+$, or as $x\in y+G_+$, or even as $x+G_+\subseteq y+G_+$. With this in mind, it is easy to see that Definition~\ref{oag} is equivalent to~\cite[p.~3]{poags}:

\begin{defn}[alternative]
\label{oagalt}
An \emph{ordered abelian group} is a set $G$ equipped with the following pieces of structure:
\begin{itemize}
\item a subset $G_+\subseteq G$, called the \emph{positive cone},
\item a binary operation $+$,
\item a distinguished element $0\in G$,
\end{itemize}
and such that the following axioms hold:
\begin{itemize}
\item $+$ is associative, commutative, has $0$ as a neutral element and inverses,
\item $0\in G_+$,
\item $G_+ + G_+\subseteq G_+$,
\item the positive cone is \emph{pointed}: if $x\in G_+$ and $-x\in G_+$, then $x=0$.
\end{itemize}
\end{defn}

The positive cone $G_+$ measures the difference between an ordering $x\geq y$ and an equality $x=y$. The cone being pointed encodes~\eqref{equaldef}, our definition of equality in terms of the ordering. With this definition, ordered abelian groups come up e.g.~in the $K$-theory of operator algebras~\cite[Definition~6.2.1]{Ktheoryop}.

We sketch how an ordered abelian group in the sense of Definition~\ref{oagalt} is also an ordered abelian group in the sense of Definition~\ref{oag}, leaving the other direction to the reader. $G$ becomes an ordered commutative monoid by taking $x\geq y$ to mean that $x-y\in G_+$. This is transitive thanks to $G_+ + G_+\subseteq G_+$: the assumption $x\geq y\geq z$ means that $x-y\in G_+$ and $y-z\in G_+$, which gives
\beq
(x-y) + (y-z) = x-z \:\in G_+,
\eeq
and hence $x\geq z$. The axiom $G_+ + G_+ \subseteq G_+$ also results in the monotonicity of addition~\eqref{monotoneplus}: the assumption $x\geq y$ means that $x-y\in G_+$, and therefore also
\beq
(x + z) - (y + z) \:\in G_+,
\eeq
and hence $x+z\geq y+z$. The existence of negatives as in Definition~\ref{oag} is by construction.

In the following sections, we will freely confuse this alternative definition of ordered abelian group with the original one. In particular, we will freely translate between $x\geq y$ and $x-y\in G_+$.

In resource-theoretic terms, the positive cone $G_+$ consists of all those resource objects which can be freely disposed of, i.e.~which can be converted into nothing. Equivalently, one can think of an $x\geq 0$ as a resource object such that having $x$ is at least as good as having nothing. Alternatively, it may be more intuitive to think in terms of $-G_+$, which is exactly the set of all those resource objects which can be freely produced, i.e.~which can be obtained from nothing. In the case $G=\oag(A)$, these ``resource objects'' are pairs $(x,y)$ consisting of a ``credit'' $x\in A$ and a ``debt'' $y\in A$ as before, together with the additional property that $y + z\geq x + z$ for some $z\in A$.

\begin{ex}
As shown in Example~\ref{grphncanc}, we have $3\pentagon - \mathcal{K}_{11}\in \oag(\Grph)_+$, although $3\pentagon\not\geq\mathcal{K}_{11}$ in $\Grph$.
\end{ex}

At the level of ordered abelian groups, we can give a simpler version of Definition~\ref{generatingpair}:

\begin{defn}
\label{oaggenerator}
A \emph{generator} in an ordered abelian group $G$ is an element $g\in G$ with $g\geq 0$ such that for every $x$ there is $n\in\N$ with $ng\geq x$.
\end{defn}

In particular, this implies that every $x\in G$ can be written as a difference $ng - (ng - x)$ of two elements of the positive cone.

Applying the generator property with $-x$ in place of $x$ shows that there also exists an $m\in\N$ with $x\geq -mg$, or equivalently $mg+x\geq 0$. Choosing $n$ such that both $ng\geq x$ and $ng+x\geq 0$ hold shows that the pair $(g,0)$ is a generating pair in the sense of Definition~\ref{generatingpair}. Conversely, if $(g_+,g_-)$ is a generating pair in an ordered abelian group, then $g_+-g_-$ is a generator. More generally, if $A$ is any ordered commutative monoid with a generating pair $(g_+,g_-)$, then $g_+-g_-$ is a generator in $\oag(A)$, as follows directly from the definitions.

A generator in an ordered abelian group is also known as an \emph{order unit}~\cite[p.~4]{poags}. We have chosen not to follow this terminology, since it does not extend nicely to the setting of ordered commutative monoids, where we need to consider pairs of elements as in Definition~\ref{generatingpair}.

\newpage
\section{The many-copies regularization: ordered $\Q$-vector spaces}
\label{sectovs}

Many resource theories display the phenomenon of \emph{economy of scale}: it may happen that the convertibility relation
\beq
x + x \geq y + y
\eeq
holds, even though $x\not\geq y$. This is another manner in which $f(x)\geq f(y)$ for all functionals $f$ does not necessarily imply that $x\geq y$, since already $x+x\geq y+y$ implies that $f(x+x)\geq f(y+y)$, and therefore $f(x)\geq f(y)$ by additivity. More generally, if $nx\geq ny$ for some $n\in\Npos$, then $f(x)\geq f(y)$, but this does not necessarily mean that $x\geq y$. In resource-theoretic terms, considering convertibility at the many-copy level may enable conversions which were impossible before. Modulo epsilonification (Remark~\ref{epsilonification}), this is a central theme in Shannon's theorems~\cite{Shannon}. Below, we will show that it also occurs in $\Grph$, but not in $\oag(\Grph)$.

Similar as to how the absence of nontrivial catalysis is captured by the mathematical notion of cancellativity, also the absence of an economy of scale effect is quite natural mathematically:

\begin{defn}
\label{deftf}
An ordered commutative monoid $A$ is \emph{torsion-free} if $nx\geq ny$ for some $n\in\Npos$ implies $x\geq y$.
\end{defn}

In~\cite[Definition~6]{strongsemi}, ordered commutative monoids with this property are called ``normal'', and it is shown that any totally ordered abelian group has this property.

In an ordered abelian group $G$, one can rephrase this definition as the requirement that $nx\in G_+$ for $n\in\Npos$ should imply $x\in G_+$. Ordered abelian groups with this property are often called ``unperforated''~\cite[Definition~6.7.1]{Ktheoryop} or ``semiclosed''~\cite[Definition~3.2]{logroups2}. So it may seem misleading to use the alternative term ``torsion-free'' for this notion, since torsion-freeness already has an established meaning in algebra: an abelian group is torsion-free if $nx=0$ implies $x=0$, and more generally a commutative monoid is torsion-free if $nx = ny$ implies $x=y$. However, our definition specializes to the conventional one in the unordered case (Remark~\ref{commmon}), similar to how our notion of cancellativity specializes to the conventional one.

\begin{rem}
$A$ is torsion-free if and only if for every $n\in\Npos$, the map $x\mapsto nx$ reflects the order in the sense of~\eqref{ff}.
\end{rem}

\begin{ex}
\label{grphntf}
$\Grph$ is not torsion-free, and we construct an example inspired by Shannon's theory of zero-error communication~\cite{zeroerror}. Let us consider the pentagon graph $\pentagon$, also called the \emphalt{$5$-cycle}, which just looks like a pentagon $\pentagon$. Some computation using \textsc{Sage} reveals the clique numbers
\beq
\omega(\pentagon \ast \pentagon) = 5, \qquad \omega(\pentagon \ast \pentagon \ast \pentagon) = 10.
\eeq
By Example~\ref{grphex}, this means that $2\pentagon \geq \mathcal{K}_5$ and $3\pentagon\geq \mathcal{K}_{10}$, and $5$ and $10$ are the largest numbers for which these inequalities hold. The first inequality implies $6\pentagon\geq 3 \mathcal{K}_5 = \mathcal{K}_{5^3} \geq \mathcal{K}_{11^2} = 2 \mathcal{K}_{11}$. Now if $\Grph$ was torsion-free, we would be led to conclude that $3\pentagon\geq \mathcal{K}_{11}$, which the above shows to be false.
\end{ex}

What follows is a powerful result relating torsion-freeness to cancellativity. As far as we know, it was originally developed in~\cite{manycopyvscatalytic} in the context of the resource theory of quantum entanglement. See the earlier Example~\ref{grphncanc} for an application.

\begin{lem}[\cite{manycopyvscatalytic}]
\label{distribute}
Let $A$ be an ordered commutative monoid equipped with an additional binary operation $\lor$ such that $A$ also is an ordered commutative monoid with respect to $\lor$, and such that $+$ distributes over $\lor$,
\beqn
\label{pldist}
x + (y\lor z) = (x + y) \lor (x + z).
\eeqn
Then:
\begin{enumerate}
\item\label{mctocat} If there is an $n\in\Npos$ with $nx \geq ny$ in $A$,\footnote{Here, $nx$ is still shorthand for the $n$-fold sum $x+\ldots+x$ and does not involve $\lor$.} then there also is $z\in A$ with $x+z\geq y+z$.
\item\label{cancimpltf} If $A$ is cancellative, then it is also torsion-free.
\item\label{oagtf} $\oag(A)$ is torsion-free.
\end{enumerate}
\end{lem}

For example, if $A$ has binary joins (as an ordered set), then $A$ is automatically an ordered commutative monoid with respect to these joins. If the addition additionally preserves these joins, then the lemma applies. The same holds if $A$ has binary meets. For a closely related statement on lattice-ordered groups, see~\cite[Lemma~2.1.2]{pogs}.

In the original application of~\cite{manycopyvscatalytic}, the addition $+$ is the tensor product of vectors $\otimes$, while $\lor$ is the direct sum $\oplus$. In this case, the distributivity requirement~\eqref{pldist} takes a familiar form,
\beq
x\otimes (y\oplus z) = (x\otimes y) \oplus (x\otimes z).
\eeq

\begin{proof}
\begin{enumerate}
\item We take the argument of~\cite{manycopyvscatalytic} and explain it in our notation. For $nx\geq ny$ with given $n\in\Npos$, define
\beqn
\label{zconstruction}
z \defin \bigvee_{k=1}^{n} \left( (k - 1) x + (n - k) y \right).
\eeqn
Then we have, by the assumed distributivity,
\begin{align*}
x + z & = x + \bigvee_{k=1}^{n} \left( (k - 1) x + (n - k) y \right) \stackrel{\eqref{pldist}}{=} \bigvee_{k=1}^n \left( kx + (n - k) y \right) \\
 & = \left[\bigvee_{k=1}^{n-1} \left( kx + (n - k) y \right) \right] \lor nx \:\stackrel{(*)}{\geq}\: ny \lor \left[\bigvee_{k=1}^{n-1} \left( kx + (n - k) y \right) \right] \\
 & = \bigvee_{k=1}^n \left( (k-1)x + (n - k + 1) y \right) \stackrel{\eqref{pldist}}{=} y + \bigvee_{k=1}^{n} \left( (k - 1) x + (n - k) y \right) = y + z.
\end{align*}
Here, the step marked $(*)$ is where the assumption $nx\geq ny$ comes in. The subsequent step also uses a reindexing $k\mapsto k-1$.
\item Assume $A$ to be cancellative. Then a putative inequality $nx\geq ny$ implies $x+z\geq y+z$ for some $z$ by part~\ref{mctocat}, and hence $x\geq y$ by assumption. Therefore $A$ is torsion-free.
\item Suppose that we have $n(x-y)\geq 0$ for $n\in\Npos$ and $x-y$ some generic element of $\oag(A)$ with $x,y\in A$. By definition of the ordering in $\oag(A)$, this means that there is $z\in A$ such that $nx + z\geq ny + z$ holds in $A$. By adding $(n-1)z$ on both sides, we also have $n(x+z)\geq n(y+z)$. But then by~\ref{mctocat}, we know that there is $w\in A$ with $x+z+w\geq y+z+w$. This results in $x-y\geq 0$ in $\oag(A)$.
\qedhere
\end{enumerate}
\end{proof}

From now on, much of the development in this section will be parallel to the one from the previous section. For simplicity, we restrict ourselves to the consideration of ordered abelian groups only. For a resource theory, this means that we assume that the catalytic regularization has already been performed.

We can make any ordered abelian group $G$ into a torsion-free one by introducing a new ordering denoted ``$\succeq$'' and defined as
\beq
x\succeq y \quad :\Longleftrightarrow\quad \exists n\in \Npos, \: nx\geq ny.
\eeq
If $G$ is the ordered abelian group describing a resource theory (after catalytic regularization, if necessary), then this relation is \emph{many-copy convertibility}. It is straightforward to show that $\succeq$ is indeed reflexive and transitive, and that this ordering coincides with the original ``$\geq$'' if $G$ already happened to be torsion-free. Moreover, the addition $+$ is clearly monotone with respect to $\succeq$ as well. If we write $\tf(G)$ for the ordered commutative monoid given by $G$ equipped with this new ordering, then $\tf(G)$ is evidently torsion-free: assuming $nx\succeq 0$ means that there exists $m\in \Npos$ with $mnx\geq 0$ in the original ordering, which indeed implies $x\succeq 0$. There is a canonical homomorphism $\eta_G:G\to \tf(G)$ which simply maps every element to itself, and this has the following universal property: if $H$ is any other ordered abelian group which is torsion-free, and $f:G\to H$ is any homomorphism, then there is a unique homomorphism $\tf(f):\tf(G)\to H$ such that $f = \tf(f)\circ \eta_G$, i.e.~such that the diagram
\beq
\xymatrix{ G \ar[rr]^{\eta_G} \ar[dr]_f & & \tf(G) \ar@{-->}[dl]^{\tf(f)} \\
 & H }
\eeq
commutes. In categorical language, the category of torsion-free ordered abelian groups is a reflective subcategory of $\OAG$.

\begin{rem}
\label{tfeq}
The induced equality relation on $\tf(G)$ may differ from the one on $G$, meaning that $\eta_G(x)=0$ may not necessarily imply that $x=0$. For example, if $G$ is a \emph{torsion group}, i.e.~if for every $x\in G$ there exists $n\in\Npos$ with $nx=0$, then we have $0\succeq x$ and $x\succeq 0$ for all $x$, resulting in $x=0$ in $\tf(G)$. Therefore $\tf(G)$ is isomorphic to the trivial group in which all elements are equal to $0$.
\end{rem}

\begin{defn}
\label{defdiv}
An ordered commutative monoid $A$ is \emph{divisible} if for all $x\in A$ and $n\in\Npos$, there is $y\in A$ with $x=ny$.
\end{defn}

\begin{defn}
\label{ovs}
An \emph{ordered $\Q$-vector space} $V$ is an ordered abelian group which is torsion-free and divisible.
\end{defn}

While in the previous section, being an ordered abelian group implied cancellativity, the analogous statement here is not the case: in general, divisibility does not imply torsion-freeness. For example, the abelian group $\Q$ is divisible, and if we define its positive cone to consist of $0$ together with all rational numbers $\geq 1$, then we obtain an ordered abelian group which is divisible, but not torsion-free.

It may sound a bit funny to use the term ``vector space'' for an ordered commutative monoid with extra \emphalt{properties}, rather than the extra \emphalt{structure} consisting of scalar multiplication. But it is a basic fact that any ordered $\Q$-vector space---as we defined it---is indeed a vector space over $\Q$ in the conventional sense. The missing piece of structure is scalar multiplication of any $x\in V$ by an arbitrary fraction $\tfrac{p}{q}\in\Q$ with $p\in\Z$ and $q\in\Npos$, and this can be defined as follows: first, we find a $y\in V$ with $qy=x$; by torsion-freeness, this $y$ is unique. Then we can define $\tfrac{p}{q}\cdot x\defin py$. It is straightforward to check that this indeed yields a $\Q$-vector space structure on $V$ in the usual sense.

If $V$ is an ordered $\Q$-vector space, $A$ any torsion-free ordered commutative monoid and $f:V\to A$ any homomorphism, then $f(\tfrac{p}{q}\cdot x) = \tfrac{p}{q} f(x)$, in the sense that the element $f(\tfrac{p}{q}\cdot x)$ solves the equation $qf(\tfrac{p}{q}\cdot x) = p f(x)$, and is necessarily the unique element to do so by torsion-freeness. In the language of \emphalt{property, structure and stuff}~\cite{propertystructurestuff}, this confirms that an ordered $\Q$-vector space is an ordered commutative monoid with extra properties, but no extra structure or stuff: the forgetful functor from ordered $\Q$-vector spaces to ordered commutative monoids is fully faithful.

Similarly to how every cancellative ordered commutative monoid embeds into an abelian group, every torsion-free ordered abelian group $G$ embeds into an ordered $\Q$-vector space $\ovs(G)$ in a universal way. We construct $\ovs(G)$ by taking $G$ and throwing in all formal $n$-th fractions of group elements (which may represent resource objects). More precisely, $\ovs(G)$ is defined as the ordered $\Q$-vector space
\beq
\ovs(G) \defin \Npos \times G,
\eeq
so that the elements of $\ovs(G)$ are pairs $(n,x)$ consisting of a positive natural number $n$ and an element $x\in G$. We think of such a pair as the formal $\tfrac{1}{n}$-th multiple or formal $n$-th part of $x$. Concerning the ordering, we declare that
\beqn
\label{ovsord}
(m,x)\geq (n,y) \quad:\Longleftrightarrow\quad n x \geq my .
\eeqn
Transitivity of this new relation follows from torsion-freeness of $\geq$. Since we then have $(n,x)\geq (mn,mx)\geq (n,x)$ for every $m\in\Npos$, we also obtain $(n,x)=(mn,mx)$.

The addition of two elements $(m,x)$ and $(n,y)$ in $\ovs(G)$ is defined by
\beq
(m,x) + (n,y) \defin (mn,nx + my),
\eeq
which corresponds to taking fractions with common denominator $mn$, as one would expect from the intuitive idea that the pairs represent formal fractions. We take the neutral element $0\in \ovs(G)$ to be given by the pair $(1,0)$. This yields an ordered commutative monoid which is torsion-free and divisible by construction, and therefore forms an ordered $\Q$-vector space.

There is a canonical homomorphism $\gamma_G:G\to \ovs(G)$ which maps every element $x\in G$ to the pair $(1,x)\in \ovs(G)$, and this has the following universal property: if $V$ is any other ordered $\Q$-vector space and $f:G\to V$, then there is a unique $\ovs(f):\ovs(G)\to V$ such that $\ovs(f)\circ \gamma_G = f$, i.e.~such that the diagram
\beq
\xymatrix{ G \ar[rr]^{\gamma_G} \ar[dr]_f & & \ovs(G) \ar@{-->}[dl]^{\ovs(f)} \\
 & V }
\eeq
commutes. The uniqueness part follows from
\beq
\ovs(f)( (n,x) ) = \ovs(f)\left( \tfrac{1}{n} \cdot (1,x)  \right) = \tfrac{1}{n}\cdot \ovs(f)( (1,x) ) = \tfrac{1}{n} f(x).
\eeq
A simple computation shows that this definition also works, which proves existence. We also note that $\gamma_G : G \to \ovs(G)$ clearly reflects the ordering in the sense of~\eqref{ff},
\beq
\gamma_G(x) \geq \gamma_G(y) \quad \Longrightarrow \quad x\geq y,
\eeq
and in particular is injective. Hence we regard $G$ as a subgroup of $\ovs(G)$ equipped with the induced ordering. In this way, we gain the mathematical convenience of working with a $\Q$-vector space rather than a group. This takes our investigations to the well-trodden ground of linear algebra, and we will make judicious use of this in the following sections. In resource-theoretic terms, $\ovs(G)$ is the \emph{many-copy regularization} of $G$.

If we start with an ordered abelian group $G$ which is not necessarily torsion-free, we can make it torsion-free by constructing $\tf(G)$, and then turn it into the ordered $\Q$-vector space $\ovs(\tf(G))$. Then there is a canonical homomorphism $\gamma_G \eta_G : G \to \ovs(\tf(G))$ having the following universal property:  if $V$ is any ordered $\Q$-vector space and $f:G\to V$, then there is a unique $\ovs(\tf(f)):\ovs(\tf(G))\to V$ such that $\ovs(\tf(f))\circ \gamma_G \eta_G = f$, i.e.~such that the diagram
\beq
\xymatrix{ G \ar[rr]^{\gamma_G\eta_G} \ar[dr]_f & & \ovs(\tf(G)) \ar@{-->}[dl]^{\ovs(\tf(f))} \\
 & V }
\eeq
commutes. At least in the setting of unordered abelian groups, this composite construction of turning an abelian group into a $\Q$-vector space is simply given by taking the tensor product of $\Q$ as abelian groups,
\beq
\ovs(\tf(G)) = \Q\otimes_\Z G.
\eeq
We leave open the question to which extent an analogous statement can also be made in our ordered setting. 

Again, we could have performed both steps of the $\ovs(\tf(G))$ construction at once and directly go from the ordered abelian group to the ordered $\Q$-vector space, but we find it instructive to split it into two separate parts which are simpler to grasp individually. Nevertheless, for the remainder of this paper we abbreviate $\ovs(\tf(G))$ by $\ovs(G)$. In categorical language, we can restate the universal property of $\ovs(G)$ like this:

\begin{thm}
\label{oagtoovs}
$\OVS$ is a reflective subcategory of $\OAG$.
\end{thm}

When $A$ is any old ordered commutative monoid, then we even abbreviate $\ovs(\oag(A))$ to $\ovs(A)$. Generally speaking, our notation is such that the outermost functor that we apply represents the highest degree of regularity, where the degree of regularity increases to the right in Figure~\ref{smcats}, and we omit all the inner functors from the notation. So by definition of both regularizations, $x\geq y$ in $\ovs(A)$ for $x,y\in A$ means that there exists $n\in\Npos$ and $z\in A$ such that with respect to the ordering in $A$,
\beqn
\label{catmc}
n(x + z) \geq n(y + z).
\eeqn
Equivalently, this happens whenever there exists $w\in A$ with
\beqn
\label{mccat}
nx + w \geq ny + w,
\eeqn
again with respect to the ordering in $A$.

In resource-theoretic terms, the elements of $\ovs(A)$ can be understood as all formal $\Q$-linear combinations of the original resource objects which make up the ordered commutative monoid $A$. We may think of such a linear combination as a \emph{portfolio} of resource objects\footnote{This terminology is due to Rob Spekkens.}, consisting of credit in some shares of some objects while at the same time owing some shares of others. The ordering relation on $\ovs(A)$ is then concerned with how one can convert such portfolios into each other. A conversion $x\geq y$ may be interpreted like this: if one owns a portfolio $y$ and has the option to trade $y$ for another portfolio $x$, then $x\geq y$ means that this trade is guaranteed to be favourable. Equivalently, if $x-y\geq 0$ holds and one has the chance of acquiring the portfolio $x-y$ for free, then one should definitely do so. 

The relevance of Theorem~\ref{oagtoovs} for our revised main question (Question~\ref{revisedmainqstn}) is that $\R$ is an ordered $\Q$-vector space, and therefore any functional on an ordered abelian group $G$ can be uniquely extended to a functional on $\ovs(G)$. If we start with an ordered commutative monoid $A$, then in order to have $f(x)\geq f(y)$ for all functionals $f$, it is therefore sufficient that $x\geq y$ in $\ovs(A)$, i.e.~that $x$ can be converted into $y$ catalytically and at the level of many copies. We can now reformulate the question even further:

\begin{qstn}
\label{revised2mainqstn}
If $V$ is an ordered $\Q$-vector space and $f(x)\geq 0$ for all functionals $f$, does this imply that $x\geq 0$? If not, what does it tell us?
\end{qstn}

So in resource-theoretic terms, we now ask: do additive monotones witness the catalytic many-copies convertibility relation~\eqref{catmc}--\eqref{mccat} perfectly? As we have shown, this is really a question about ordered $\Q$-vector spaces. In the next section, we will see that the answer is again negative, although this is now much less simple to see. A third and final step of regularization will be performed, this time of a much more analytic rather than algebraic nature.

There is a standard alternative definition of ordered $\Q$-vector space which we find useful: 

\begin{defn}[alternative]
\label{ovsalt}
An \emph{ordered $\Q$-vector space} is a set $V$ equipped with the following additional pieces of structure:
\begin{itemize}
\item a subset $V_+\subseteq V$ called the \emph{positive cone},
\item a binary operation $+$,
\item a scalar multiplication $\cdot : \Q\times V\to V$,
\item a distinguished element $0\in V$,
\end{itemize}
and such that the following axioms hold:
\begin{itemize}
\item these operations turn $V$ into a $\Q$-vector space,
\item $V_+$ is closed under addition and multiplication by nonnegative scalars,
\item\label{ovspointed} the positive cone is \emph{pointed}: if $x\in V_+$ and $-x\in V_+$, then $x=0$.
\end{itemize}
\end{defn}

Again, we think of the last requirement not as an axiom, but rather as the \emphalt{definition} of equality on $V$.

Getting back to the concept of generators, the following well-known property of points in an ordered $\Q$-vector space is often of interest:

\begin{lem}[{see also e.g.~\cite[Lemma 1.7]{cones}}]
\label{unitinterior}
For $g\geq 0$ in an ordered $\Q$-vector space $V$, the following are equivalent:
\begin{enumerate}
\item\label{ovsgenpair} $(g,0)$ is a generating pair in the sense of Definition~\ref{generatingpair}.
\item\label{ovsgenerator} $g$ is a generator in the sense of Definition~\ref{oaggenerator}.
\item\label{orderunit} $g$ is an \emph{order unit}, i.e.~for every $x\in V$ there exists $\lambda\in\Q$ such that $x+\lambda g\geq 0$.
\item\label{interiorpoint} $g$ is an \emph{interior point}, i.e.~for every $x\in V$ there exists\footnote{Recall our stipulation that lowercase Greek letters always designate \emphalt{rational} numbers.} $\eps>0$ such that $g+\eps x\geq 0$.
\end{enumerate}
\end{lem}

So all these seemingly different concepts become simple reformulations of one another at the level of ordered $\Q$-vector spaces.

\begin{proof}
We already know that~\ref{ovsgenpair} and~\ref{ovsgenerator} are equivalent.
\begin{itemize}
\item[\implproof{ovsgenerator}{orderunit}] For given $x$, we know that there exists $n\in\N$ with $ng\geq -x$, and hence $x + ng\geq 0$.
\item[\implproof{orderunit}{ovsgenerator}] Applying the assumption to $-x$ yields $-x+\lambda g \geq 0$. Together with $g\geq 0$, this implies that $\lceil\lambda\rceil g\geq x$.
\item[\implproof{orderunit}{interiorpoint}] In $x+\lambda g\geq 0$, we may assume without loss of generality that $\lambda>0$ since $g\geq 0$. Then $g+\lambda^{-1} x\geq 0$.
\item[\implproof{interiorpoint}{orderunit}] The assumption $g+\eps x\geq 0$ implies $x+\eps^{-1} g\geq 0$.\qedhere
\end{itemize}
\end{proof}

Many ordered $\Q$-vector spaces that arise in nature come equipped with a \emphalt{distinguished} order unit. Considering such \emph{order unit spaces} results in a rich mathematical theory, which is usually developed with $\R$ rather than $\Q$ as the ground field~\cite{orderunit}. For resource theories, there often may exist a canonical choice of generator or order unit as well---as in Example~\ref{grphgenpair}.

\newpage
\section{The seed regularization: Archimedean ordered $\Q$-vector spaces}
\label{sectaovs}

So after having performed the catalytic and many-copy regularizations, does $f(x)\geq f(y)$ for all functionals $f$ imply that $x\geq y$? Again, the answer is negative, as the following example shows: consider the ordered $\Q$-vector space $V\defin \Q^2$ with positive cone
\beqn
\label{nonarch}
V_+ \defin \{\: (0,0) \:\} \:\cup\: \{\: (\alpha,\beta)\in\Q^2 \:|\: \alpha > 0 \:\}.
\eeqn
In this example, we claim that $f( (0,1) ) = 0$ for any functional $f$, and that therefore $(0,1)$ cannot be distinguished from the origin $(0,0)$ by functionals only. The reason for this is that $(\eps,1)\geq 0$ for any $\eps>0$, and therefore
\beq
f( (0,1) ) + \eps f( (1,0) ) = f( (\eps,1) ) \geq 0,
\eeq
which implies $f( (0,1) ) \geq 0$ upon taking $\eps\to 0$. So we have $f( (0,1) )\geq 0$ for all functionals $f$, although $(0,1)\not\geq 0$. Since we must also have $f( (0,-1) ) = - f( (0,1) ) \geq 0$ by symmetry, we can even conclude that $f( (0,1) ) = 0$. 

Ordered $\Q$-vector spaces that display this phenomenon may arise from thermodynamic resource theories as follows. Let the $\beta$ coordinate be given by a physical quantity which is additive under combination of systems, such as the volume of a gas. Let $\alpha$ be a quantity like the negative of the entropy of the system. Then, in a perfectly adiabatic transition, $\beta$ can be changed arbitrarily without incurring any cost in $\alpha$. However, real-world transitions are never perfectly adiabatic, because realizing adiabatic transitions requires an infinite amount of time; hence, any actual transition will incur a nonzero but arbitrarily small cost in $\alpha$. This leads to a positive cone of the form~\eqref{nonarch}. So the considerations of this section may be of particular relevance to thermodynamics.

We now consider the conditions under which a phenomenon like this is guaranteed not to occur. This is more intricate than the simple algebraic conditions of cancellativity and torsion-freeness of the previous two sections, and we need to deal with some functional analysis.

\begin{defn}
\label{defacc}
A subset $U\subseteq V$ of an ordered $\Q$-vector space $V$ is said to be
\begin{enumerate}
\item \emph{convex} if $x,y\in U$ and rational $\lambda\in [0,1]$ gives $\lambda x + (1-\lambda)y\in U$,
\item\label{absconvdef} \emph{absolutely convex} if it is convex and $V_+\subseteq U$,
\item\label{absorbdef} \emph{absorbent} if for every $x\in V$, there exists $\mu\in \Qpos$ with $\mu x \in U$.
\end{enumerate}
\end{defn}

Part~\ref{absconvdef} is again a definition which usually carries that name only in the unordered setting, in which $V_+=\{0\}$, so that the additional condition simply states that $0\in U$.

The absolutely convex absorbent sets form a neighbourhood basis of $0$ in a topology on $V$ which turns $V$ into something close to a topological $\Q$-vector space: addition and multiplication by nonnegative scalars are both continuous, but inversion $x\mapsto -x$ may not be.

The resource-theoretic interpretation of a convex set $U$ is that it is a set of portfolios which is closed under combining portfolios, but only in a way which does not increase the total amount of content in the portfolios by taking a $\lambda$-th share of the first portfolio and a $(1-\lambda)$-th share of the second. A set of portfolios is absolutely convex if it contains in addition all those portfolios which are at least as good as the empty portfolio. Finally, a set of portfolios is absorbent if it contains some (typically small) multiple of every other portfolio.

\begin{ex}
\label{epsU}
If $V$ is an ordered $\Q$-vector space with a generator $g$ and $\eps>0$, then the set
\beqn
\label{Ueps}
U_\eps\defin \{\: x\in V \:|\: x + \eps g \geq 0 \:\}
\eeqn
is absolutely convex absorbent.
\end{ex}

Absolutely convex absorbent sets of the form~\eqref{Ueps} have the special property that they are \emph{upwards closed}: if $y\in U_\eps$ and $x\geq y$, then also $x\in U_\eps$. Equivalently, we have $U_\eps + V_+ \subseteq U_\eps$. A general absolutely convex absorbent set need not have this property, but it will at least be very close to doing so:

\begin{lem}
\label{almostupwardsclosed}
If $U$ is absolutely convex absorbent, $x\in U$ and $y\geq 0$, then for every $\lambda\in(0,1)$ we have $\lambda x + y \in U$. In other words, $\lambda U + V_+ \subseteq U$.
\end{lem}

This approximates upward closure in the limit $\lambda\to 1$ and shows that not taking upward closure as part of the definition of absolutely convex absorbent sets makes little difference.

\begin{proof}
We have $(1-\lambda)^{-1} y\geq 0$ and hence $(1-\lambda)^{-1} y\in U$. Therefore the convex combination
\beq
\lambda\cdot x + (1-\lambda) \cdot (1-\lambda)^{-1}y = \lambda x + y
\eeq
must be in $U$ as well.
\end{proof}

We now claim that if $x\in V$ is such that $x\in U$ for all absolutely convex absorbent $U$, then $f(x)\geq 0$ for all functionals $f$. The reason is simple: for every rational $\eps>0$, the set of all $y\in V$ with $f(y)\geq -\eps$ is absolutely convex absorbent. This is yet another sufficient condition which will finally also turn out to be necessary.

\begin{defn}
\label{aovs}
An ordered $\Q$-vector space $W$ is \emph{Archimedean} if its positive cone is the intersection of all absolutely convex absorbent sets,
\beq
W_+ = \bigcap\:\{\: U \!\textrm{ absolutely convex absorbent}\:\}.
\eeq
\end{defn}

In other words, $W$ is Archimedean if and only if $x\in U$ for every $U$ implies that $x\in W_+$.

\begin{rem}
If one develops the theory of locally convex vector spaces over $\Q$ in analogy to that over $\R$, then $W$ is Archimedean if and only if $W_+$ is topologically closed in the finest locally convex topology on $W$.
\end{rem}

\begin{ex}
\label{fdarch}
If $W$ is finite-dimensional, then it is Archimedean if and only if $W_+$ is topologically closed as a subset of $\Q^{\dim(W)}$ in the standard (Euclidean) topology. In particular,~\eqref{nonarch} is not Archimedean.
\end{ex}

For another important class of examples, we have the following criterion:

\begin{prop}
\label{ovsarchcrit}
If an ordered $\Q$-vector space $W$ has a generator $g$, then it is Archimedean if and only if $x+\eps g\geq 0$ for all $\eps>0$ implies $x\geq 0$.
\end{prop}

In fact, for order unit spaces, this is generally taken to be the \emphalt{definition} of Archimedeanicity~\cite{orderunit}.

\begin{proof}
If $W$ is Archimedean, then we need to show that $x + \eps g\geq 0$ for all $\eps>0$ implies $x\in U$ for all absolutely convex absorbent $U$. This follows if every $U$ contains some $U_\eps$ of the form~\eqref{Ueps}. And indeed, for a given $U$ choose $\eps>0$ such that $-2\eps g\in U$. For any $y\in U_\eps$, we have $y+\eps g\geq 0$ by definition of $U_\eps$. Therefore by Lemma~\ref{almostupwardsclosed} with $\lambda=\tfrac{1}{2}$, we conclude $y\in U$. This proves $U_\eps\subseteq U$.

Conversely, assume that the implication in the claim holds. Then if $x\in U$ for all absolutely convex absorbent $U$, we in particular have $x\in U_\eps$ for all $\eps$, and therefore $x\geq 0$ by the assumed implication.
\end{proof}

\begin{prop}
\label{ocmarchcrit}
If $A$ is an ordered commutative monoid with a generating pair $(g^+,g^-)$, then $\ovs(A)$ is Archimedean if and only if the following implication holds: if for all $\eps>0$ there are $n,k\in\Npos$ with $k\leq\eps n$ and 
\beq
nx+kg_+\geq ny+kg_-,
\eeq
then there are $m\in\Npos$ and $z\in A$ with
\beq
mx+z\geq my+z.
\eeq
\end{prop}

\begin{proof}
If the criterion of Proposition~\ref{ovsarchcrit} holds for some $x$, then it will also hold for every positive scalar multiple of $x$. Hence among all the formal $\Q$-linear combinations that make up $\ovs(A)$, it is sufficient to apply the criterion to all formal integer linear combinations, i.e.~the elements of $\oag(A)$. A generic such element is of the form $x-y$ for $x,y\in A$. Then $x-y\geq 0$ in $\ovs(A)$ is directly equivalent to the consequent stated, that there are $m\in\Npos$ and $z\in A$ with $mx+z\geq my+z$.

On the other hand, we need to translate the antecedent condition of Proposition~\ref{ovsarchcrit}, which is $(x-y)+\eps g\geq 0$ in $\ovs(A)$ for all $\eps>0$, into the condition stated in the claim. First, if the condition in terms of $nx+kg_+\geq ny+kg_-$ holds, then we certainly also have $nx+n\eps g_+\geq ny+n\eps g_-$, possibly after multiplying $n$ by a large enough factor such as to make $n\eps$ an integer. This leads to $(x-y)+\eps (g_+ - g_-)\geq 0$ in $\ovs(A)$, which is the desired inequality with $g=g_+ - g_-$ the generator. We will show the other direction later in the proof of Theorem~\ref{ocmhb}.
\end{proof}

\begin{prob}
We do not know whether $\ovs(\Grph)$ is Archimedean. In light of Lemma~\ref{distribute}\ref{oagtf} and the graph join of Example~\ref{grphncanc}, the criterion of Proposition~\ref{ocmarchcrit} becomes equivalent to the following question: if $\mathcal{G}$ and $\mathcal{H}$ are graphs such that for every $\eps>0$ there exist $n\in\Npos$ and a graph map $\mathcal{H}^{\ast n}\lra \mathcal{G}^{\ast n}\ast \mathcal{K}_{\lfloor 2^{n\eps}\rfloor}$, does this imply that there exists a graph $\mathcal{L}$ and a graph map $\mathcal{H}\ast\mathcal{L}\lra\mathcal{G}\ast\mathcal{L}$?
\end{prob}

Being Archimedean has some nice consequences:

\begin{lem}
\label{approach}
If $W$ is an Archimedean ordered $\Q$-vector space and $x,y\in W$ are such that $x+\eps_n y\geq 0$ for a sequence $(\eps_n)_{n\in\N}$ with $\eps_n\to 0$, then $x\geq 0$.
\end{lem}

\begin{proof}
First, suppose that $\eps_n=0$ for some $n$. In this case, we already have $x\geq 0$ as an assumption and hence there is nothing to prove. So we assume from now on that all $\eps_n$ have definite sign.

Second, suppose that the sequence $(\eps_n)$ changes sign somewhere, i.e. that $\eps_n>0$ but $\eps_m<0$ for some $n,m\in\N$. Then we have $x+\eps_n y\geq 0$ and $x - |\eps_m|y \geq 0$ by assumption. These inequalities superpose to $(|\eps_m| + \eps_n)x \geq 0$, which implies $x\geq 0$ as claimed.

Hence the remaining case is that all of the $\eps_n$'s have the same sign. By replacing $y\mapsto -y$ if necessary, we assume without loss of generality that $\eps_n>0$ for all $n$. We then prove that $\tfrac{x}{2}\in U$ for every absolutely convex absorbent $U$, which implies $\tfrac{x}{2}\in W_+$ by assumption and hence also $x\in W_+$. Since $U$ is absorbent, we know that $-\lambda y\in U$ for some $\lambda\in\Qpos$. But we also have $x+\eps_n y\in U$ by assumption, and therefore the point
\beq
\frac{\eps_n}{\eps_n + \lambda} (-\lambda y) + \frac{\lambda}{\eps_n + \lambda} (x + \eps_n y)  = \frac{\lambda}{\eps_n+\lambda} x
\eeq
also lies in $U$ by convexity of $U$. As long as $\eps_n\leq\lambda$, the coefficient in front of $x$ is at least $\tfrac{1}{2}$, and then we can again use convexity of $U$ and $0\in U$ to show that $\tfrac{x}{2}\in U$.
\end{proof}

We write $\AOVS$ for the category of Archimedean ordered $\Q$-vector spaces. If $V_+\subseteq V$ is the positive cone of an ordered $\Q$-vector space, then we can ``Archimedeanize'' it by keeping the underlying set $\aovs(V)\defin V$ and taking the positive cone to be given by
\begin{align*}
\aovs(V)_+ \defin  & \bigcap \: \{\: U \!\textrm{ absolutely convex absorbent}\:\} \\
 = & \left\{ \: x\in V \:|\: x\in U \textrm{ for every absolutely convex absorbent } U \:\right\}.
\end{align*}
This $\aovs(Q)$ is easily checked to be an ordered $\Q$-vector space as well, and it has the property that its absolutely convex absorbent sets are the same as the original $V$'s: if $U$ is absorbent and absolutely convex with respect to $V_+$, then it also contains $\aovs(V)_+$ by definition of the latter as an intersection in which $U$ occurs. This implies in particular that $\aovs(V)$ itself is indeed Archimedean, as the notation suggests.

\begin{ex}
In our initial example~\eqref{nonarch}, we have
\beq
\aovs(V)_+ = \{\: (\alpha,\beta) \in \Q^2 \:|\: \alpha \geq 0 \:\} .
\eeq
Since this contains the whole line $\{0\}\times\Q$, the resulting notion of equality induced by the ordering as in~\eqref{equaldef} collapses this entire line to a single point. Hence $\aovs(V)$ actually turns out to be one-dimensional. So also in this third step of regularization, the resulting notion of equality may change.
\end{ex}

The definition of $\aovs(V)_+$ in terms of absolutely convex absorbent sets is very difficult to deal with in practice due to the abundance of absolutely convex absorbent sets. Hence it is useful to have some other ways of constructing $\aovs(V)_+$ which apply in certain situations. For example if $V$ is finite-dimensional, then $\aovs(V)_+$ is just the closure of $V_+$ in the standard topology. But since we expect $V$ to be infinite-dimensional in most situations of interest, this will probably not be very useful. 

Motivated by Lemma~\ref{approach}, one can try to construct $\aovs(V)_+$ in general by taking it to contain all those $x$ for which there exists $y$ and a sequence $\eps_n\to 0$ with $x+\eps_n y\geq 0$. This results in the so-called \emph{sequential closure} of $V_+$, which unfortunately may be significantly smaller than $\aovs(V)_+$~\cite{closures}. But if a generator exists, then everything is fine:

\begin{lem}[{e.g.~\cite[Section~2.3]{orderunit}}]
\label{seqclose}
If $V$ is an ordered $\Q$-vector space with a generator $g$ (Lemma~\ref{unitinterior}), then $x\in\aovs(V)_+$ if and only if $x+\eps g\geq 0$ for all $\eps>0$.
\end{lem}

\begin{proof}
If one chooses any sequence $(\eps_n)_{n\in\N}$ of positive rationals with $\eps_n\to 0$, then the ``if'' part is an instance of Lemma~\ref{approach}. For the ``only if'' part, we consider the absolutely convex absorbent $U_\eps$ from~\eqref{Ueps}. By definition of $\aovs(V)_+$, every $x\in\aovs(V)_+$ is contained in all of these.
\end{proof}

Replacing $V$ by $\aovs(V)_+$ is what we call the \emph{seed regularization}. In the case when $\aovs(V)_+$ coincides with the sequential closure, the resource-theoretic interpretation of this regularization is in terms of \emph{seeds}. Originating in the field of pseudorandom number generators, the notion of ``seed'' refers to the possibility of consuming (or absorbing) a small amount of a valuable resource object $g$ in order to facilitate the conversion of some other resource object $x$ into $y$; with this definition, the seed regularization has been employed in~\cite{thermaleq,revshannon}. Its main point is that the required amount of $g$ is sublinear $o(n)$ in the number of copies being converted, and it is in this sense that it ``seeds'' the conversion. For example, $g$ may be a catalyst for the conversion of $x$ into $y$. In this case, one only needs a single copy of it in order to convert arbitrarily many $x$'s into $y$'s, since the same catalyst can be used any number of times; we will make use of this further down.

We can formulate a universal property of $\aovs(V)$ analogous to the universal properties that we encountered in the previous two sections. There is a canonical homomorphism $\eta_V : V \to \aovs(V)$ given by the identity function on the underlying vector spaces. If $W$ is any Archimedean ordered $\Q$-vector space and $f:V\to W$, then there is a unique $\aovs(V):\aovs(V)\to W$ such that $\aovs(f)\circ \eta_V = f$, i.e.~such that the diagram
\beq
\xymatrix{ V \ar[rr]^{\eta_V} \ar[dr]_f & & \aovs(V) \ar@{-->}[dl]^{\aovs(f)} \\
 & W }
\eeq
commutes. To see this, note that uniqueness is clear, since $\gamma_V$ is the identity function on the level of underlying vector spaces. So all that we need to show is that if $f$ is a homomorphism and $x\in V$ lies in all absolutely convex absorbent sets, then $f(x)\geq 0$. But since $W$ is Archimedean, for showing $f(x)\geq 0$ it is enough to prove that $f(x)$ lies in any absolutely convex absorbent set $U\subseteq W$. Because $f^{-1}(U)\subseteq V$ is also absolutely convex absorbent and therefore contains $x$ by assumption, we can indeed conclude $f(x)\in U$, as was to be shown.

So in category-theoretic terms, we have:

\begin{thm}
\label{ovstoaovs}
$\AOVS$ is a reflective subcategory of $\OVS$.
\end{thm}

\begin{rem}
\label{QvsR}
Here and in all of the following, we work with vector spaces over $\Q$. One can go further and turn $\aovs(V)$ into an $\R$-vector space by completing with respect to the uniform structure that it carries naturally. However, since our current goal is mostly to give a complete answer to Question~\ref{mainqstn}, we leave the investigation of this additional step to an $\R$-vector space structure to future work. 
\end{rem}

The following result is a Hahn--Banach theorem which finally enables us to answer the question as to which elements of an ordered commutative monoid can be distinguished by functionals (Question~\ref{mainqstn}). It will be an important technical tool that we will make repeatedly make use of in the remainder of this paper.

\begin{thm}[Hahn--Banach separation for cones]
\label{aovshb}
For an Archimedean ordered $\Q$-vector space $W$ and $x\in W$, the following are equivalent:
\begin{enumerate}
\item\label{acapos} $x\geq 0$,
\item\label{fpos} $f(x)\geq 0$ for all functionals $f$.
\end{enumerate}
\end{thm}

Since the Hahn--Banach theorem is among the most basic results of functional analysis and many variants of it have been proven, we do not expect this result to be original, although we have not found it anywhere in this precise form (over $\Q$). In the resource-theoretic interpretation it could be related to an earlier application of the Hahn--Banach theorem in the foundations of thermodynamics~\cite{clausius}.

\begin{proof}
Only the direction~\ref{fpos}$\Rightarrow$\ref{acapos} is nontrivial, and we prove its contrapositive: if $x\not\geq 0$, then there exists $f$ with $f(x)<0$.

The starting situation is that we have the positive cone $W_+\subseteq W$ and $x\not\in W_+$. Due to Archimedeanicity, we can witness $x\not\in W_+$ by an absolutely convex absorbent set $U\subseteq W$ with $x\not\in U$. Associated to $U$ we have the Minkowski gauge
\beq
p \: : \: W \lra \R, \qquad y \longmapsto \inf \{\: \lambda\in\Qplus \: |\: y\in \lambda U\:\}.
\eeq
Since $U$ is absorbent, the infimum is guaranteed to exist. We also have $p(x) \geq 1$ since $x\not\in U$. Furthermore, $p$ is positively homogeneous, $p(\mu y) = \mu p(y)$ for all $\mu>0$, and subadditive, $p(y+z)\leq p(y) + p(z)$, thanks to convexity of $U$. On the one-dimensional subspace $\Q x$, we can define a linear map $g:\Q x\to \R$ by $g(\mu x)\defin \mu$, and then we have $g\leq p$ on this subspace, which follows for positive $\mu$ from $p(x) \geq 1$ and for negative $\mu$ from nonnegativity of $p$. By the Hahn--Banach extension theorem for sublinear functionals Theorem~\ref{hbext} from the appendix, this linear map extends to $g:W\to\R$ such that $g(x)=1$ and $g\leq p$. For $y\geq 0$, we have $p(y)=0$, and therefore $g(y)\leq 0$. In conclusion, $f:=-g$ is the desired functional with $f(x)=-g(x)=-1 < 0$.
\end{proof}

We can think of the theorem as saying that for every Archimedean ordered $\Q$-vector space $W$, there exists a family of functionals $f_i:W\to\R$ indexed by some $i\in I$ which jointly reflects the order, i.e.~$x\geq 0$ if and only if $f_i(x)\geq 0$ for all $i$. Equivalently speaking, we have an embedding $X\lra\R^I$ given by $x\longmapsto (i\longmapsto f_i(x))$. Or in category-theoretic terms, we can say that the object $\R$ is a cogenerator in $\AOVS$.

In order to get a complete answer to Question~\ref{mainqstn}, we would now like to translate Theorem~\ref{aovshb} back into the world of ordered commutative monoids by traversing all three regularizations backwards. However, we have not yet found a general and reasonably nice way to do this; the difficulties are in characterizing the Archimedeanization of an arbitrary ordered $\Q$-vector space in a clean manner. So we state our results only under the additional assumption that a generating pair exists, in which case we can resort to Lemma~\ref{seqclose} for an appealing description of the Archimedeanization.

\begin{thm}
\label{ovshb}
Let $V$ be an ordered $\Q$-vector space with generator $g\in V$. Then the following are equivalent for any $x\in V$:
\begin{enumerate}
\item $x\geq 0$ in $\aovs(V)$.
\item $x+\eps g\geq 0$ for all $\eps>0$.
\item $f(x)\geq 0$ for all functionals $f$.
\end{enumerate}
\end{thm}

\begin{proof}
Combine Theorem~\ref{aovshb} with Lemma~\ref{seqclose}.
\end{proof}

\begin{thm}
\label{oaghb}
Let $G$ be an ordered abelian group with generator $g$. Then the following are equivalent for any $x\in G$:
\begin{enumerate}
\item\label{oagaovs} $x\geq 0$ in $\aovs(G)$.
\item\label{oagcond} For all $\eps>0$, there exist $n,k\in\Npos$ with $k\leq\eps n$ and $nx+kg\geq 0$.
\item\label{oagfun} $f(x)\geq 0$ for all functionals $f$.
\end{enumerate}
\end{thm}

\begin{proof}
\begin{enumerate}
\item[\implproof{oagaovs}{oagcond}] By Theorem~\ref{ovshb}, we know that $x+\eps g\geq 0$ holds for all $\eps>0$ in $\ovs(G)$. By the definition of the ordering on $\ovs(G)$, this implies that there exists an $n\in\Npos$ with $n\eps\in\N$ such that $nx + n\eps g \geq 0$. So it is even possible to take $k=n\eps$.
\item[\implproof{oagcond}{oagfun}] Straightforward.
\item[\implproof{oagfun}{oagaovs}] Theorem~\ref{aovshb}. \qedhere
\end{enumerate}
\end{proof}

In~\ref{oagcond}, one is allowed to make use of a number of copies of the generator which is sublinear $o(n)$ in the desired number of conversions $n$. This also applies at the level of ordered commutative monoids:

\begin{thm}
\label{ocmhb}
Let $A$ be an ordered commutative monoid with generating pair $(g_+,g_-)$. Then the following are equivalent for any $x,y\in A$:
\begin{enumerate}
\item\label{ocmaovs} $x\geq y$ in $\aovs(A)$.
\item\label{ocmcond} For all $\eps>0$, there exist $n\in\N$ and $k\in\N$ with $k\leq\eps n$ and $nx + kg_+ \geq ny + kg_-$.
\item\label{ocmfun} $f(x) \geq f(y)$ for all functionals $f$.
\end{enumerate}
\end{thm}

\begin{proof}
\begin{enumerate}
\item[\implproof{ocmaovs}{ocmcond}] By Theorem~\ref{oaghb} and the construction of $\oag(A)$, the assumption $x\geq y$ means that for all $\eps>0$, there exist $n,k\in\N$ with $k\leq\tfrac{\eps}{2} n$ and
\beq
n(x-y) + k(g_+-g_-)\geq 0
\eeq
in $\oag(A)$, and we can assume that $k\geq 2$ without loss of generality. In terms of $A$ itself, this means that there exists $z\in A$ with
\beqn
\label{nxmgz}
nx + kg_+ + z \geq ny + kg_- + z,
\eeqn
and we now need to show that we can get rid of $z$. Since $(g_+,g_-)$ is a generating pair, we have $l\in\N$ with $lg_+ + z \geq lg_-$ and $lg_+\geq z + lg_-$. The main trick for getting rid of $z$ is to apply~\eqref{nxmgz} repeatedly: for every $j\in\{1,\ldots,l\}$,~\eqref{nxmgz} tells us that
\beq
j(nx + kg_+) + (l-j)(ny + kg_-) + z \geq (j-1)(nx + kg_+) + (l-j+1)(ny+kg_-) + z ,
\eeq
where the right-hand side is exactly the left-hand side, except that the value of $j$ has decreased by one. So we can chain all these inequalities together, from the left-hand side with $j=l$ until the right-hand side with $j=1$, and obtain
\beqn
\label{knxmgz}
lnx + lkg_+ + z \geq lny + lkg_- + z.
\eeqn
But now since $lg_+ + z \geq lg_-$ and $lg_+\geq z + lg_-$, we also have
\begin{align*}
lnx + l(k+1)g_+ & \geq lnx + lkg_+ + z + lg_- + lg_+ \\
& \geq lny + lkg_- + z + lg_+ + lg_- \\
& \geq lny + l(k+2)g_-.
\end{align*}
Putting $n'\defin ln$ and $k'\defin l(k+2)$ therefore yields
\beq
n' x + k'g_+ \geq n' y + k'g_-,
\eeq
which is the desired result thanks to $k'=l(k+2)\leq l(k+k)\leq \eps ln = \eps n'$.
\item[\implproof{ocmcond}{ocmfun}] Straightforward. 
\item[\implproof{ocmfun}{ocmaovs}] Theorem~\ref{aovshb}.
\qedhere
\end{enumerate}
\end{proof}

The main step in this proof is the derivation of~\eqref{knxmgz} from~\eqref{nxmgz}. This part generalizes~\cite[Lemma~2.3]{extstates}, and also is an elaboration on an argument of Lieb and Yngvason~\cite[Theorem 2.1]{secondlaw1} which was also used independently by Devetak, Harrow and Winter~\cite[Lemma~4.6]{quantumshannon}. The interpretation of this argument is that the catalytic regularization is redundant after one has performed the seed regularization, since the use of $(g_+,g_-)$ as a seed lets us buy any catalyst which can then be used over and over again.

\newpage
\section{Classifying additive monotones: the structure of functionals}
\label{sectfun}

The results of the previous section highlight the distinguished role played by functionals. Theorem~\ref{ocmhb} delineates accurately how much functionals can ``see'' of the structure of an ordered commutative monoid with a generating pair. But in practice, how can we possibly check whether a condition like $f(x)\geq f(y)$ holds for \emphalt{all} functionals $f$? Doing so requires a good understanding of the structure of the functionals on the given ordered commutative monoid. In this section, we develop some basic methods for analyzing this structure.

\begin{defn}
A functional $f$ is \emph{extremal} if for any other functional $f'$ with $f\geq f'$, we have $f'=\lambda f$ for some constant $\lambda\in\Rplus$.
\end{defn}

This definition uses the pointwise ordering on functionals, in which $f\geq f'$ stands for $f(x)\geq f'(x)$ for all $x\in A$. Then $f\geq f'$ is equivalent to the statement ``$f-f'$ is a functional itself''. 

One can regard the functionals on an ordered commutative monoid $A$ as forming a convex cone within the space of linear maps $\aovs(A)\to\R$. It is straightforward to show that a functional is extremal if and only if it spans an extremal ray of this convex cone.

\begin{ex}
If $W$ is a finite-dimensional Archimedean ordered $\Q$-vector space for which $W_+\subseteq W$ is a polyhedral cone, then a functional on $f:W\to\R$ is extremal if and only if its kernel is a facet of $W_+$. The reason for this is that the facets correspond precisely to the extremal rays of the dual cone, which is the convex cone of functionals.
\end{ex}

In general, extremal functionals on an Archimedean ordered $\Q$-vector space need not exist. For example, if $W$ is an Archimedean ordered $\Q$-vector space of dimension $\geq 1$ and with trivial positive cone $W_+=\{0\}$, then no functional is extremal. However, we will see that if $W$ has a generator, then sufficiently many extremal functionals exist to span all other functionals in a suitable sense. For now, we work on an Archimedean ordered $\Q$-vector space $W$ with generator $g$, and later lift our results to the world of ordered commutative monoids.

\begin{defn}
On an Archimedean ordered $\Q$-vector space with generator $g$, a functional $f$ is \emph{normalized} if $f(g)=1$.
\end{defn}

For every functional $f$, we have either $f(g)=0$, which implies $f=0$ since $g$ is a generator, or $f(g)>0$, which implies that $f$ is a scalar multiple of a normalized functional.

The set of normalized functionals is a convex subset of the set of all functionals. We equip the set of all functionals and the set of normalized functionals with the weakest topology which makes all evaluation maps
\beq
f\longmapsto f(x)
\eeq
continuous; this is the \emph{weak-* topology}. Then the set of normalized functionals is compact by the Banach--Alaoglu theorem. From now on, we also use other standard notions and results from convex functional analysis~\cite{tvsSW,tvsKN}.

\begin{lem}
A normalized functional is extremal if and only if it is an extreme point of the compact convex set of normalized functionals.
\end{lem}

\begin{proof}
Suppose $f$ is extremal, and $f=t f_1 + (1-t) f_2$ for normalized functionals $f_1,f_2$ and $ t\in [0,1]$. Then in particular $f\geq t f_1$, and hence by extremality there exists $s\in\Rplus$ with $t f_1=s f$. Normalization of $f$ and $f_1$ then implies that $s = t$, and therefore also $(1-t)f_2 = f - t f_1 = (1-t)f$. Thus $f$ cannot be written as a convex combination of normalized functionals in a nontrivial manner, which makes it into an extreme point.

Conversely, suppose $f$ is an extreme point and $f\geq f'$ for some functional $f'$. Then $f-f'$ is also a functional. We distinguish three cases: first, if $f'(g)=0$, then $f'=0$ since $g$ is a generator, and then there is nothing more to prove. Second, if $f'(g) = 1$, then also $(f-f')(g)=0$, and therefore $f-f'=0$ again because $g$ is a generator, and there is nothing more to be shown. Third, if $0<f'(g)<1$, then we obtain the equation
\beq
f = f'(g) \cdot \frac{f'}{f'(g)} + (1-f'(g)) \cdot \frac{f-f'}{1-f'(g)},
\eeq
which decomposes $f$ into a convex combination of other normalized functionals. By assumption, this means that $\tfrac{f'}{f'(g)} = f$, and hence $f'=f'(g)\cdot f$. In all cases, we have again exhibited $f'$ as a nonnegative scalar multiple of $f$.
\end{proof}

\newcommand{\ext}{\overline{\mathrm{ex}}}

We write $\ext(W,g)$ for the closure in the weak-* topology of the set of normalized extremal functionals on $W$ with respect to the generator $g$.

\begin{thm}
\label{rieszrepthm}
For every functional $f$ there exists a regular Borel measure $\mu$ on $\ext(W,g)$ such that
\beqn
\label{rieszrepeq}
f(x) = \int \hat{f}(x) \;\mathrm{d}\mu(\hat{f}).
\eeqn
\end{thm}

Roughly speaking, the theorem says that every functional is a nonnegative linear combination of normalized extremal functionals, where ``nonnegative linear combination'' in general has to be understood in the sense of an integral. This statement seems closely related to Choquet theory, which studies points in compact convex sets and how they can be written as integrals over extreme points or boundary points, but we have not been able to derive it in this exact form using standard Choquet theory. We also do not know whether a similar statement holds with respect to the set of normalized extremal functionals $\mathrm{ex}(W,g)$ without taking the closure.

\newcommand{\ev}{\mathrm{ev}}

\begin{proof}
We consider the $\R$-vector space $C(\ext(W,g))$ of real-valued continuous functions on $\ext(W,g)$. For every $x\in W$, the evaluation map
\beq
\ev_x : \ext(W,g)\lra \R,\qquad \hat{f}\longmapsto \hat{f}(x)
\eeq
is continuous by definition of the topology on $\ext(W,g)$, and therefore is an element of $C(\ext(W,g))$. Moreover, the map which sends every $x\in W$ to the associated evaluation map,
\beqn
\label{WtoC}
W \lra C(\ext(W,g)),\qquad x\longmapsto \ev_x,
\eeqn
is $\Q$-linear and order-preserving. Here, we equip $C(\ext(W,g))$ with the pointwise ordering, in which $F\geq 0$ if and only if $F(\hat{f})\geq 0$ for all $\hat{f}\in\ext(W,g)$. That~\eqref{WtoC} is order-preserving then means that if $x\geq 0$, then $\ev_x\geq 0$. We furthermore claim that~\eqref{WtoC} reflects the order, which means that $\ev_x\geq 0$ implies $x\geq 0$. The assumption $\ev_x\geq 0$ means that $\hat{f}(x)\geq 0$ for all $\hat{f}\in\ext(W,g)$; but by the Krein-Milman theorem, the closed convex hull of all these $\hat{f}$ coincides with the set of all normalized functionals, and therefore $f(x)\geq 0$ for all normalized $f$, which implies $f(x)\geq 0$ for all functionals $f$. The claim $x\geq 0$ then follows from Theorem~\ref{aovshb}.

Since~\eqref{WtoC} reflects the order, it follows that it also reflects equality, i.e.~is injective. In particular, we can identify $W$ with its image in $C(\ext(W,g))$. We furthermore equip $C(\ext(W,g))$ with the sublinear function
\beqn
\label{supp}
p(F) \defin \max \{\: F(\hat{f}) \:|\: \hat{f}\in\ext(W,g) \:\} ,
\eeqn
which assigns to every continuous function its maximal value. This behaves like an ordered version of the supremum norm $||\cdot||_\infty$.

On the subspace $W$ of $C(\ext(W,g))$, we have
\begin{align}
\label{pdefs}
\begin{split}
p(x) \defin p(\ev_x) & = \max \{\: \hat{f}(x) \:|\: \hat{f}\in\ext(W,g) \:\} = \max \{\: f(x) \:|\: f(g) = 1 \:\} \\
 & = \inf \{\: \lambda\in\Q \:|\: \lambda g\geq x \:\}
\end{split}
\end{align}
where the last two equations follow again from the Krein-Milman theorem and Theorem~\ref{aovshb}. Now let us be given an arbitrary normalized functional $f$. By~\eqref{pdefs}, we have
\beq
f(x) \leq p(x),
\eeq
and therefore the Hahn--Banach extension theorem in the form of Theorem~\ref{hbext} allows us to extend $f$ from a $\Q$-linear map $W\to\R$ satisfying this inequality to a $\Q$-linear map $C(\ext(W,g))\to\R$ satisfying the analogous inequality. By this very inequality, this extension must automatically be $\R$-linear. Also, the extension maps every $F\leq 0$ to a nonpositive number, and hence every $F\geq 0$ to a nonnegative number. Therefore by the Riesz representation theorem, there exists a regular Borel measure $\mu$ on $\ext(W,g)$ such that this linear map is of the form
\beq
C(\ext(W,g)) \lra \R,\qquad F\longmapsto \int F(\hat{f})\; \mathrm{d}\mu(\hat{f}).
\eeq
Upon restriction to the subspace $W$, where we know the linear map to be given by $x\longmapsto f(x)$, we therefore obtain the desired equation~\eqref{rieszrepeq}.
\end{proof}

We now extend this result from Archimedean ordered $\Q$-vector spaces to ordered commutative monoids $A$ with a generating pair $(g_+,g_-$).

\begin{defn}
A functional $f:A\to\R$ is \emph{normalized} if
\beqn
\label{ocmnorm}
f(g_+) = f(g_-) + 1.
\eeqn
\end{defn}

In the $K$-theory of operator algebras, normalized functionals on an ordered abelian group with a generator are known as \emphalt{states}~\cite[Definition~6.8.1]{Ktheoryop}.

\begin{lem}
\label{normalizable}
Every nonzero functional $f$ is a scalar multiple of a normalized functional.
\end{lem}

\begin{proof}
By $g_+\geq g_-$, we know $f(g_+)\geq f(g_-)$. If this is a strict inequality, then we can rescale $f$ such that the two sides differ by $1$ and the resulting functional is normalized. Otherwise, if equality $f(g_+) = f(g_-)$ holds, then $f$ is the zero functional: for any $x\in A$ we have $n\in\N$ with $ng_+ \geq x + ng_-$ and $ng_+ + x \geq ng_-$, and applying $f$ to these inequality results in $0\geq f(x)$ and $f(x)\geq 0$.
\end{proof}

Similar to before, we equip the convex set of normalized functionals with the weakest topology which makes the evaluation maps
\beq
f\longmapsto f(x)
\eeq
continuous for all $x\in A$. Again this is the weak-* topology, and it coincides with the weak-* topology which one obtains by uniquely extending every functional to $\aovs(A)$ and equipping the normalized functionals with the weak-* topology there:

\begin{lem}
For every $y\in\aovs(A)$, also the evaluation map
\beq
f\longmapsto f(y)
\eeq
is continuous.
\end{lem}

\begin{proof}
We first prove the statement for $\oag(A)$. In this case, a generic element is given by $x-y$ for $x,y\in A$. The associated evaluation map is
\beq
f\longmapsto f(x-y) = f(x) - f(y),
\eeq
which is continuous since $f\mapsto f(x)$ and $f\mapsto f(y)$ are.

Now for $\ovs(A)$, every element is of the form $\tfrac{1}{n}\cdot z$ for $z\in\oag(A)$ and $n\in\Npos$. The associated evaluation map is
\beq
f\longmapsto f\left(\tfrac{1}{n}\cdot z\right) = \tfrac{1}{n}\cdot f(z) ,
\eeq
which is continuous since $f\mapsto f(z)$ is.

Finally since the elements of $\aovs(A)$ are precisely the elements of $\ovs(A)$, the claim follows.
\end{proof}

In conclusion, this means that we have an equality of topological spaces
\beq
\ext(A,g_+,g_-) = \ext(\aovs(A),g_+-g_-).
\eeq
If we now apply Theorem~\ref{rieszrepthm} to $\aovs(A)$, we obtain its generalization to ordered commutative monoids:

\begin{cor}
\label{ocmfunrep}
For every functional $f$ there exists a regular Borel measure $\mu$ on $\ext(A,g_+,g_-)$ such that 
\beq
f(x) = \int \hat{f}(x) \;\mathrm{d}\mu(\hat{f}).
\eeq
\end{cor}

Again, we think of this as roughly saying that every functional is a nonnegative linear combination of extremal functionals.

\begin{ex}
\label{renyi}
There is another ordered commutative monoid of relevance to quantum information theory: $\ProbMaj$, the ordered commutative monoid of finite probability spaces ordered by majorization. The elements of $\ProbMaj$ are finite probability spaces, which are pairs $(\inalph,P)$ consisting of a finite set $\inalph$ and a probability distribution $P:\inalph\to[0,1]$. For two finite probability spaces $(\inalph,P)$ and $(\outalph,Q)$, we take their combination $(\inalph,P)+(\outalph,Q)$ to be given by the product space
\beq
(\inalph\times\outalph, P\times Q),
\eeq
which corresponds to sampling from $P$ and $Q$ independently. We write $n\defin\max(|\inalph|,|\outalph|)$ and sort the individual probabilities in nonincreasing order,
\beq
P_1 \geq P_2 \geq\ldots\geq P_n,\qquad Q_1\geq Q_2 \geq\ldots\geq Q_n,
\eeq
where the distribution on the smaller set needs to be appended by zeroes. Then we put $(\inalph,P)\geq (\outalph,Q)$ whenever
\beq
\sum_{j=1}^k P_j \geq \sum_{j=1}^k Q_j
\eeq
holds for all $k=1,\ldots,n$. This defines the ordered commutative monoid $\ProbMaj$. In the case of equal cardinality $|\inalph|=|\outalph|$, the ordering of $\ProbMaj$ is the \emph{majorization order}~\cite{majorize}. Extending to the case of unequal cardinality by appending zeroes is a natural choice, and it is precisely in this form that majorization comes up in quantum entanglement theory: $\ProbMaj$ is the ordered commutative monoid describing the resource theory of two-party pure state entanglement~\cite{entmajor}.

$\ProbMaj$ has a generating pair given by $g_+\defin 0$ and taking $g_-$ to be given by a two-element set, such as $\{\textrm{heads},\textrm{tails}\}$, equipped with the uniform distribution which assigns probability $\tfrac{1}{2}$ to each outcome. For every parameter value $t\in[0,\infty]$, the R\'enyi entropy $H_t : \ProbMaj\to\R$ is a functional given by\footnote{The R\'enyi entropies with parameter $t<0$ are not functionals, since they take infinite values on distributions of non-full support and therefore are not even maps of the form $\ProbMaj\to\R$. If desired, one could try to fix this by only allowing distributions with full support as resource objects.}
\beq
H_t( (\inalph,P) ) \defin \frac{1}{1-t}\, \log_2 \left( \sum_{a\in\inalph} P(a)^t \right).
\eeq
The deep results of Klimesh~\cite{catmajorK} and Aubrun and Nechita~\cite{catmajorAN} suggest that the following might be true:
\begin{itemize}
\item All R\'enyi entropies $H_t$ are normalized extremal functionals.
\item There are no other normalized extremal functionals besides the R\'enyi entropies.
\item Moreover, the set of R\'enyi entropies is closed in the weak-* topology.
\end{itemize}
Taken together, these three statements would imply that $\ext(\ProbMaj,g_+,g_-)$ coincides with the set of R\'enyi entropies. If this turns out to be true, then Corollary~\ref{ocmfunrep} shows that every functional $f:\ProbMaj\to\R$ is of the form
\beq
f(x) = \int_0^\infty H_t(x) \;\mathrm{d}\mu(t)
\eeq
for some regular Borel measure $\mu$ on the extended half-line $[0,\infty]$. Here, the topology on $[0,\infty]$ would a priori be the weak-* topology introduced earlier; but if the above is true, then it seems reasonable to expect this topology to coincide with the usual topology on $[0,\infty]$, which is the one-point compactification of $\Rplus=[0,\infty)$.

Majorization is also highly relevant in the resource theories studied in the context of thermodynamics~\cite{nonuniformity,secondlaws}. However, the ordered commutative monoid there is slightly different, and the difference lies in how the ordering is defined in the case of unequal cardinality $|\inalph|\neq|\outalph|$. Nevertheless, it might be possible to make an analogous conjecture here as for $\ProbMaj$, the only difference being that the R\'enyi entropies $H_t$ should be replaced by the R\'enyi divergences $D_t$ relative to the uniform distribution (on the same sample space as the distribution under consideration).
\end{ex}

\begin{ex}
\label{grphfun}
For $\Grph$, we have certain examples of multiplicative functionals like the complementary Lov\'asz number from Proposition~\ref{lovprop}, the fractional chromatic number from Example~\ref{moregraphinvariants} and the projective rank, and their logarithms are genuine additive functionals. Since e.g.~the complementary Lov\'asz number and the fractional chromatic number coincide on complete graphs but differ on some other graphs, such as the pentagon graph $\pentagon$, at least two of these functionals are linearly independent. Hence the cone of functionals is at least two-dimensional, and therefore $\aovs(\Grph)$ is at least two-dimensional as well.

We do not know anything else about the structure of functionals on $\Grph$. In particular, we do not have any other bounds on these dimensions and we also do not know whether either of these functionals is extremal.
\end{ex}

\newpage
\section{Rates and the rate formula: the structure of two-dimensional slices}
\label{sectrates}

In many situations, it is not just of interest to try to convert one copy of $x$ into one copy of $y$,
\beq
x \stackrel{?}{\geq} y,
\eeq
but one would like to obtain as many copies of $y$ from $x$ as possible,
\beqn
\label{xsupny}
\sup\, \{\: m\in\N \: |\: x \geq m y \:\} = \mathop{?}
\eeqn
Similarly, one may be interested in the minimal number of copies of $x$ that are required for producing one copy of $y$,
\beqn
\label{nxinfy}
\inf\, \{\: n\in\N \: |\: n x\geq y \:\} = \mathop{?}
\eeqn
In a mass production setting, these two questions become subsumed by a third one of a similar flavour. In this case, one has many copies of $x$ available and tries to turn these into as many copies of $y$ as possible. So then the problem is to maximize the ratio $\tfrac{m}{n}$ for which $nx\geq my$,
\beqn
\label{defmaxrate}
R_{\max}(x\to y) \defin \sup \left\{\: \frac{m}{n} \: \Big|\: nx \geq my \:\right\}.
\eeqn
Here, the supremum ranges over all $m,n\in\N$, and we use the conventions $\tfrac{m}{0}\defin \infty$ for $m>0$ and $\tfrac{0}{0}\defin 0$, where the latter ensures that $m=n=0$ does not contribute to the supremum. This quantity is the \emph{maximal rate} of converting $x$ into $y$. It represents the average number of $y$'s which one can maximally extract from one copy of $x$. If no number of copies of $x$ can be converted to any number of copies of $y$, then we have $R_{\max}(x\to y)=0$; this will typically happen e.g.~for $x=0$. On the other hand, the fractions $\frac{m}{n}$ with $nx\geq my$ may also be unbounded, in which case we have $R_{\max}(x\to y)=\infty$. This arises either because there is some $n$ such that $nx$ can be converted into any number of copies of $y$, e.g.~if $x\geq 0\geq y$, or for the weaker reason that the maximal number of $y$'s that can be extracted from $nx$ grows superlinearly in $n$. The latter can also be phrased as saying that the number of $x$'s that is necessary for producing $my$ grows sublinearly in $m$.

Roughly speaking, a rate maximization of the form~\eqref{defmaxrate} is what Shannon's noisy coding theorem~\cite{Shannon} in the resource theory of communication is about: it gives a concrete answer to the question of how many copies of the noiseless channel which communicates one bit perfectly one can obtain from many copies of a given channel. However, due to the lack of a suitable way to deal with the ``epsilons'' in the noisy coding theorem, this is currently still a non-example for us (Remark~\ref{epsilonification}). But there are interesting examples that do fit our definition:

\begin{ex}
\label{moregraphinvariants}
In $\Grph$, let $x=\mathcal{G}$ be any graph and $y=\mathcal{K}_2$ be the complete graph on $2$ vertices. Then the rate $R_{\max}(\mathcal{G}\to \mathcal{K}_2)$ is the supremum over all fractions $\tfrac{m}{n}$ for which there exists a graph map $\mathcal{K}_{2^m}\to\mathcal{G}^{\ast n}$. In terms of the clique number $\omega$ from Example~\ref{grphex}, for a given $n$ the maximal feasible $m$ is precisely
\beq
m = \lfloor \log_2 \omega(\mathcal{G}^{\ast n}) \rfloor,
\eeq
and hence the maximal rate is
\beq
R_{\max}(\mathcal{G}\to \mathcal{K}_2) = \sup_n \frac{\lfloor \log_2 \omega(\mathcal{G}^{\ast n}) \rfloor}{n} = \lim_{n\to\infty} \frac{\lfloor \log_2 \omega(\mathcal{G}^{\ast n}) \rfloor}{n},
\eeq
where the supremum coincides with the limit due to Fekete's lemma. Since omitting the flooring $\lfloor\cdot\rfloor$ makes a difference of at most $\tfrac{1}{n}$, and as such preserves both convergence and the limit of the sequence, we can omit the flooring and write
\beq
R_{\max}(\mathcal{G}\to \mathcal{K}_2) = \lim_{n\to\infty} \frac{\log_2 \omega(\mathcal{G}^{\ast n})}{n} = \log_2 \lim_{n\to\infty} \sqrt[n]{\omega(\mathcal{G}^{\ast n})} = \log_2 \Theta(\overline{\mathcal{G}}) .
\eeq
The right-hand side is precisely the logarithm of the graph invariant known as the \emph{Shannon capacity} $\Theta$ of the complement graph $\overline{\mathcal{G}}$. In this way, we have recovered another famous graph invariant, at least up to the trivial transformations of taking the complement and the logarithm.

Also the rate $R_{\max}(\mathcal{K}_2\to\mathcal{G})$ turns out to recover a well-studied graph invariant. Similar to above, it is the supremum over all fractions $\tfrac{m}{n}$ for which there exists a graph map $\mathcal{G}^{\ast m}\to \mathcal{K}_{2^n}$. In terms of the chromatic number $\chi$ from Example~\ref{grphex}, this means that for a given $m$, the minimal feasible $n$ is precisely
\beq
n = \lceil \log_2 \chi(\mathcal{G}^{\ast m}) \rceil,
\eeq
and hence the maximal rate is
\beq
R_{\max}(\mathcal{K}_2\to\mathcal{G}) = \lim_{m\to\infty} \frac{m}{\lceil \log_2 \chi(\mathcal{G}^{\ast m}) \rceil} = \left(\lim_{m\to\infty} \frac{\lceil \log_2 \chi(\mathcal{G}^{\ast m}) \rceil}{m} \right)^{-1}.
\eeq
Similar to before, we can omit the ceiling $\lceil \cdot\rceil$ and write
\beq
R_{\max}(\mathcal{K}_2\to\mathcal{G}) = \left(\lim_{m\to\infty} \frac{\log_2 \chi(\mathcal{G}^{\ast m})}{m} \right)^{-1} = \left( \log_2 \lim_{m\to\infty} \sqrt[m]{\chi(\mathcal{G}^{\ast m})} \right)^{-1}.
\eeq
This shows that this maximal rate is also an incarnation of a famous graph invariant: it is the reciprocal of the logarithm of the \emph{fractional chromatic number} of $\mathcal{G}$~\cite[Corollary~3.4.3]{fgt}. Of course, all of this also holds with any other complete graph in place of $\mathcal{K}_2$, and only the base of the logarithm changes.
\end{ex}

As opposed to~\eqref{xsupny} and~\eqref{nxinfy}, we might alternatively be interested in turning $x$ into as \emphalt{few} copies of $y$ as possible,
\beqn
\label{xinfny}
\inf\, \{\: m\in\N \: |\: x \geq m y \:\} = \mathop{?} 
\eeqn
or similarly in absorbing as \emphalt{many} copies of $x$ as possible into $y$,
\beqn
\label{nxsupy}
\sup\, \{\: n\in\N \: |\: n x \geq y \:\} = \mathop{?} 
\eeqn
These are relevant questions when $x$ is a resource object which is considered undesirable and we try to get rid of instances of $x$ by turning them into $y$'s. Correspondingly, there is a \emph{minimal rate} at which $x$'s can be converted into $y$'s,
\beqn
\label{defminrate}
R_{\min}(x\to y) \defin \inf \left\{\: \frac{m}{n} \: \Big|\: nx \geq my \:\right\} .
\eeqn
Here we use the conventions $\tfrac{m}{0}\defin \infty$ for $m>0$ and also $\tfrac{0}{0}\defin \infty$, where the latter ensures that $m=n=0$ does not contribute to the infimum. Just as the maximal rate can be $0$ or $\infty$ in certain cases, we have $R_{\min}(x\to y)=\infty$ if no number of copies of $x$ can be converted to any number of copies of $y$. Similarly, $R_{\min}(x\to y) = 0$ if already a finite number of $y$'s is sufficient for absorbing any number of $x$'s, or if the number of $x$'s that can be absorbed by $my$ grows superlinearly in $m$, or equivalently if the number of $y$'s required for absorbing $nx$ grows sublinearly in $n$.

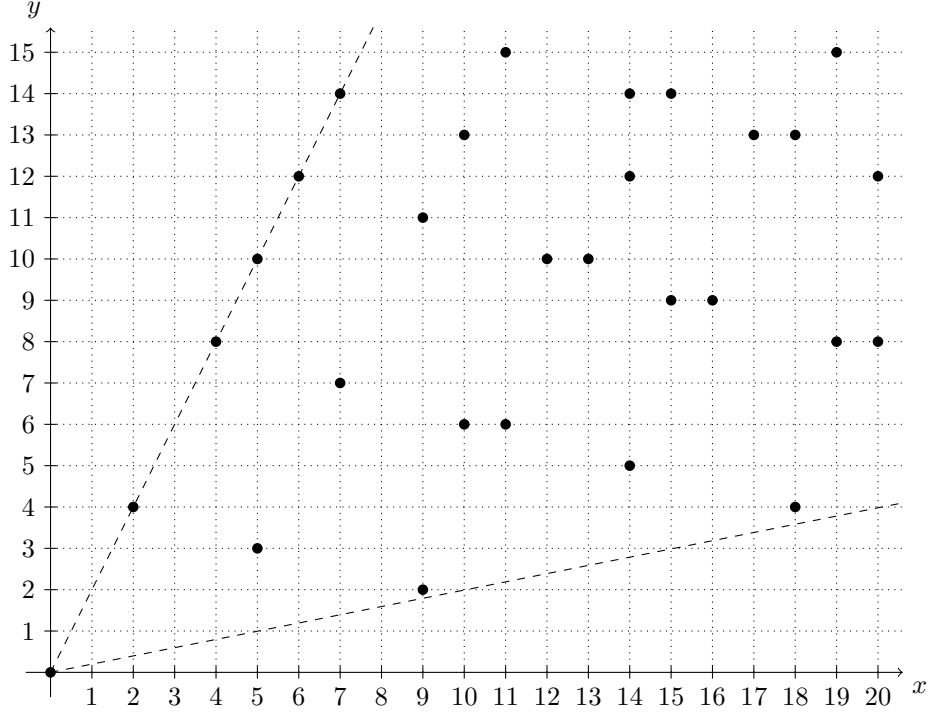
\begin{figure}
\begin{tikzpicture}[scale=.55]
\tikzstyle{point}=[fill,circle,inner sep=0pt,minimum size=4pt]
\draw[->] (-.6,0) -- (20.6,0) node [below right] {$x$} ;
\draw[->] (0,-.6) -- (0,15.6) node [above left] {$y$} ;
\foreach \x in {1,...,20} {
	\draw[xshift=\x cm] (0,-.15) node[below] {$\x$} -- (0,.15) ;
	\draw[xshift=\x cm,dotted] (0,0) -- (0,15.6) ; }
\foreach \y in {1,...,15} {
	\draw[yshift=\y cm] (-.15,0) node[left] {$\y$} -- (.15,0) ;
	\draw[yshift=\y cm,dotted] (0,0) -- (20.6,0) ; }
\foreach \xy in {(0,0),(2,4),(4,8),(5,3),(5,10),(6,12),(7,7),(7,14),(9,2),(9,11),(10,6),(10,13),(11,6),(11,15),(12,10),(13,10),(14,5),(14,12),(14,14),(15,9),(15,14),(16,9),(17,13),(18,4),(18,13),(19,8),(19,15),(20,8),(20,12)}
	\node[point] at \xy {} ;
\draw[dashed] (0,0) -- (7.8,15.6) ;
\draw[dashed] (0,0) -- (20.6,4.1) ;
\end{tikzpicture}
\caption{An example of what a commutative submonoid of $\N^2$ may look like. The maximal and minimal rates are the slopes of the upper and lower bounding rays, respectively.}
\label{ratefig}
\end{figure}

\begin{rem}
\label{2dslice}
There is an intuitive geometrical interpretation of both maximal and minimal rates as follows. Generalizing Remark~\ref{annihilator}, we may consider all pairs $(n,m)\in\N^2$ for which $nx\geq my$. This set of pairs forms a submonoid of $\N^2$ under addition; see Figure~\ref{ratefig} for what this may look like. For $nx\geq my$, the corresponding fraction $\tfrac{m}{n}$ is then the slope of the line or ray connecting the origin to the point $(n,m)\in\N^2$, which is an element of the submonoid. So the maximal rate, as the supremum over all such slopes, can be identified with the ray tightly bounding the submonoid from above, and similarly the minimal rate corresponds to the ray tightly bounding the submonoid from below. All of this is determined by the ordered commutative submonoid spanned by the elements $x$ and $y$ under consideration, and we think of this submonoid as a ``two-dimensional slice''.
\end{rem}

\begin{rem}
Due to certain technical issues that we encounter below, we need to point out that these definitions of maximal and minimal rate may not yet be the most appropriate ones, and we investigate an improved definition below. So while this section provides the results and methods expected of a basic theory of rates, the precise technical details of the definitions and results may be subject to revision in future work.
\end{rem}

In cases when the arguments $x$ and $y$ are clear from the context, we omit their mention and simply write $R_{\max}$ and $R_{\min}$ for the maximal and minimal rates. Here are some basic observations on these rates with our present definitions.

\begin{lem}
\label{carnotlem}
For every $x,y,z\in A$, we have
\begin{align}
\label{carnot}
\begin{split}
R_{\max}(x\to z) &\geq R_{\max}(x\to y) R_{\max}(y\to z) ,\\[4pt]
R_{\min}(x\to z) &\leq R_{\min}(x\to y) R_{\min}(y\to z),
\end{split}
\end{align}
\end{lem}

\begin{proof}
We only prove the first inequality since the second is analogous.

For $\eps>0$, we choose $n,m\in\N$ such that $\tfrac{m}{n} \geq R_{\max}(x\to y) - \eps$ and $nx\geq my$; and similarly, we choose $n',m'\in\N$ with $\tfrac{m'}{n'} \geq R_{\max}(y\to z) - \eps$ and $n'y\geq m'z$. Thereby we also obtain
\[
n' nx\geq n' my = mn'y \geq mm'z ,
\]
and hence
\[
R_{\max}(x\to z) \geq \frac{mm'}{n'n} = \frac{m}{n} \cdot \frac{m'}{n'} \geq \left( R_{\max}(x\to y) - \eps \right) \cdot \left( R_{\max}(y\to z) - \eps \right) .
\]
This yields the claim in the limit $\eps\to 0$.
\end{proof}

\begin{lem}
\label{ratedicho}
For every $x\in A$, we have either $R_{\max}(x\to x) = 1$ or $R_{\max}(x\to x) = \infty$. Similarly, either $R_{\min}(x\to x) = 1$ or $R_{\min}(x\to x) = 0$.
\end{lem}

\begin{proof}
We only prove the first part since the second is analogous.

We certainly have $R_{\max}\geq 1$ thanks to $x\geq x$. Furthermore, plugging in $z=y=x$ in~\eqref{carnot} shows that we necessarily must have $R_{\max}\geq R_{\max}^2$. With $R_{\max}\geq 1$, this implies $R_{\max}=1$ or $R_{\max}=\infty$.
\end{proof}

Putting the previous two lemmas together yields immediately:

\begin{cor}
\label{carnotcor}
If all rates in $A$ are finite, then
\beqn
\label{carnotineqmax}
R_{\max}(x\to y) \cdot R_{\max}(y\to x) \leq 1.
\eeqn
If all rates are strictly positive, then
\beqn
\label{carnotineqmin}
R_{\min}(x\to y) \cdot R_{\min}(y\to x) \geq 1.
\eeqn
\end{cor}

This is reminiscent of ``Carnot-style'' reasoning in thermodynamics: composing a conversion of $x$'s into $y$'s with a conversion of $y$'s into $x$'s cannot lead to an impossible multiplication of $x$'s into more $x$'s. In thermodynamics, an $x$ could stand for the combination of a hot and a cold reservoir of a certain finite size, while $y$ would correspond to the same two reservoirs in equilibrium together with a certain amount of extracted work.

\begin{lem}
\label{ratenhom}
\begin{enumerate}
\item For every $n\in\Npos$,
\beq
R_{\max}(nx\to y) = \tfrac{1}{n}\cdot R_{\max}(x\to y), \qquad R_{\min}(nx\to y) = \tfrac{1}{n}\cdot R_{\max}(x\to y).
\eeq
\item For every $m\in\Npos$,
\beq
R_{\max}(x\to my) = m\cdot R_{\max}(x\to y), \qquad R_{\min}(x\to my) = m\cdot R_{\min}(x\to y).
\eeq
\end{enumerate}
\end{lem}

\begin{proof}
Straightforward.
\end{proof}

The concepts of maximal and minimal rates are special cases of a general notion of rate. The definition is this:

\begin{defn}
\label{defrate}
\begin{enumerate}
\item\label{defratea} A nonnegative extended real number $r\in\Rplus\cup\{\infty\}$ is a \emph{rate} from $x$ to $y$ if every neighbourhood of $r$ contains a fraction $\tfrac{m}{n}$ with $nx\geq my$.
\item The set of all rates is the \emph{rate region}.
\end{enumerate}
\end{defn}

Here and in the following, when we speak of a \emph{fraction $\tfrac{m}{n}$}, we refer to a pair of natural numbers $m,n\in\N$, not both of which are zero. This terminology seems appropriate in our current context, where we want $\tfrac{m}{n}$ to be interpreted as the rational number (or $\infty$) which it denotes. In all cases, it is permitted for one of $m$ or $n$ to be zero. Sometimes we want to require both $m>0$ and $n>0$; we then speak of a \emph{finite fraction}.

\begin{rem}
\begin{enumerate}
\item An equivalent definition of rate would be that a rate is a number $r\in\Rplus\cup\{\infty\}$ for which there exists a sequence of fractions $\left(\tfrac{m_j}{n_j}\right)_{j\in\N}$ which converges to $r$ and such that $n_j x \geq m_j y$ for all $j$.
\item The rate region is also determined by the two-dimensional slice of Remark~\ref{2dslice}, and the rates are exactly those slopes of rays in Figure~\ref{ratefig} which can be approximated arbitrarily closely by rays through points in the submonoid.
\item If the only case in which $nx\geq my$ holds is with $n=m=0$, then no rate exists. This is an unpleasant special case in which our current definitions do not play together so well: $R_{\max}=0$ and $R_{\min}=\infty$ are still well-defined values, although no rate exists.
\end{enumerate}
\end{rem}

Due to the following observation, to compute rates it is sufficient to compute the minimal and maximal rates:

\begin{prop}
\label{rateregion}
The rate region is the closed interval $[R_{\min}, R_{\max}]$.
\end{prop}

\begin{proof}
Due to the definition of rate, every rate is bounded below by the minimal rate and bounded above by the maximal rate. Hence it only needs to be shown that every $r\in [R_{\min},R_{\max}]$ is indeed a rate. Having such an $r$ requires $R_{\min}\leq R_{\max}$, so that there is at least one fraction $\tfrac{m}{n}$ with $nx\geq my$. Then if $r=R_{\min}$ or $r=R_{\max}$, the claim follows from the definition of minimal and maximal rate.

Otherwise, we have $R_{\min}<r<R_{\max}$, so that there exist finite fractions $\tfrac{m_-}{n_-}<r$ and $\tfrac{m_+}{n_+}>r$ such that 
\beqn
\label{nmass}
n_- x \geq m_- y, \qquad  n_+ x \geq m_+ y.
\eeqn
We now take the unique $t\in[0,1]$ with
\beqn
\label{ralpha}
r = t \frac{m_-}{n_-} + (1-t) \frac{m_+}{n_+}
\eeqn
and approximate $t$ by a finite fraction $\tfrac{k}{l}\leq 1$ with
\beq
\left|t - \frac{k}{l}\right| \leq \eps.
\eeq
By adding the $kn_+$-th multiple of the first inequality in~\eqref{nmass} to the $(l-k)n_-$-th multiple of the second, we obtain
\beq
k n_+ n_- x + (l-k) n_+ n_- x \geq k n_+ m_- y + (l-k) n_- m_+ y,
\eeq
and hence we know that the number
\beq
\frac{k n_+ m_- + (l - k) n_- m_+}{k n_+ n_- + (l - k) n_+ n_-} = \frac{k}{l}\cdot \frac{m_-}{n_-} + \left(1-\frac{k}{l}\right)\cdot \frac{m_+}{n_+}
\eeq
is a rate as well. As $\eps\to 0$, this indeed converges to $r$ as desired.
\end{proof}

\begin{lem}
\label{rbound}
If $r$ is a rate from $x$ to $y$, then for all functionals $f$,
\beqn
\label{fvsr}
f(x) \geq r f(y)
\eeqn
\end{lem}

In the case $r=\infty$, this inequality is to be interpreted as $0\geq f(y)$.

\begin{proof}
If $r$ is a rate, then in every neighbourhood of $r$ we can find a fraction $\tfrac{m}{n}$ with $n x\geq m y$. Then we also have $n f(x)\geq m f(y)$. For $n>0$, this implies the conclusion upon dividing by $n$ and letting $\tfrac{m}{n}$ tend to $r$. For $n=0$, we must have $m>0$, since otherwise $\tfrac{m}{n}$ would not be a fraction. So also in this case, we conclude the expected $0\geq f(y)$.
\end{proof}

The first part of the argument also shows how rates behave under homomorphisms:

\begin{lem}
\label{ratehom}
If $f:A\to B$ is a homomorphism and $r$ is a rate from $x$ to $y$ in $A$, then $r$ is also a rate from $f(x)$ to $f(y)$ in $B$. 
\end{lem}

In an ordered $\Q$-vector space, we can rewrite the definition of maximal rate in a form which looks closer to~\eqref{xsupny} and~\eqref{nxinfy},
\beqn
\label{rmaxalt}
R_{\max}(x\to y) = \sup\, \{\: \beta\in\Qplus \: |\: x \geq \beta y \:\} = \left( \inf\, \{\: \alpha\in\Qplus \: |\: \alpha x\geq y \:\} \right)^{-1},
\eeqn
where one needs to understand the supremum of an empty set as being $0$, in order to correctly cover the case when $nx\geq my$ holds only when $n=0$. Similarly, we can rewrite the definition of minimal rate as
\beqn
\label{rminalt}
R_{\min}(x\to y) = \inf\, \{\: \beta\in\Qplus \: |\: x \geq \beta y \:\} = \left( \sup\, \{\: \alpha\in\Qplus \: |\: \alpha x\geq y \:\} \right)^{-1},
\eeqn
where again the supremum of the empty set is declared to be $0$. The idea behind these expressions is that in an ordered $\Q$-vector space, one can rewrite the minimal and maximal rates in a form similar to~\eqref{xsupny}--\eqref{nxinfy} and~\eqref{xinfny}--\eqref{nxsupy}. We will encounter this kind of expression also in~\eqref{infsup}.

In the light of Theorem~\ref{aovshb}, these definitions suggest that Lemma~\ref{rbound} should also have a converse in the case of an Archimedean ordered $\Q$-vector space. If this was the case, then Lemma~\ref{rbound} would not only yield useful upper bounds on maximal rates and lower bounds on minimal rates, but the converse would even prove these bounds to be tight. However, with our current definition of rate, such a converse does not exist even in the well-behaved case of finite $r$:

\begin{ex}
\label{funnyrate}
Let $W=\Q^3$ be the Archimedean ordered $\Q$-vector space in which $(\alpha,\beta,\gamma)\geq 0$ if and only if
\beq
|\sqrt{2} \alpha + \beta| \leq \gamma.
\eeq
Since this condition can be written as two separate linear inequalities, it defines a polyhedral cone, which is automatically topologically closed in $\Q^3$ and therefore Archimedean (Example~\ref{fdarch}).

On the $(\alpha,\beta,0)$-plane, the ordering is trivial: $(\alpha,\beta,0)\geq 0$ would require $\sqrt{2}\alpha+\beta=0$, which is impossible with $\alpha,\beta\in\Q$. In particular, there is no rate from $e_1=(1,0,0)$ to $e_2=(0,1,0)$. On the other hand, we claim that
\beqn
\label{fsqrt2}
f(e_1) = \sqrt{2} f(e_2)
\eeqn
holds for all functionals $f$, thereby showing that the converse of Lemma~\ref{rbound} is not true even for Archimedean ordered $\Q$-vector spaces. To see this, we choose an $\eps>0$ and an $\eps$-approximation $\lambda\in\Q$ to $\sqrt{2}$,
\beq
\sqrt{2} - \eps \leq \lambda \leq \sqrt{2} + \eps .
\eeq
By definition of the positive cone, we have
\beq
e_1 - \lambda e_2 + \eps e_3 = (1,-\lambda,\eps) \geq 0, \qquad -e_1 + \lambda e_2 + \eps e_3 = (-1,\lambda,\eps) \geq 0.
\eeq
Hence for any functional $f$,
\beq
f(e_1) - \lambda f(e_2) \geq -\eps f(e_3),\qquad f(e_1) - \lambda f(e_2) \leq \eps f(e_3).
\eeq
The claim~\eqref{fsqrt2} follows in the limit $\eps\to 0$.
\end{ex}

However, the converse to Lemma~\ref{rbound} \emphalt{does} hold for Archimedean ordered $\Q$-vector spaces under the additional assumption $y\geq 0$:

\begin{prop}
\label{ratecrit}
If $W$ is an Archimedean ordered $\Q$-vector space with $x,y\in W$ and $y\geq 0$, then $r\in\R_{\geq 0}\cup\{\infty\}$ is a rate from $x$ to $y$ if and only if 
\beqn
\label{ineqrf}
f(x)\geq r f(y)
\eeqn
for all functionals $f$.
\end{prop}

In the $r=\infty$ case, we again interpret the inequality as stating that $0\geq f(y)$ for all $f$. However, Theorem~\ref{aovshb} tells us that this happens if and only if $0\geq y$, which is an uninteresting case due to the assumption $y\geq 0$.

\begin{proof}
The ``only if'' part is covered by Lemma~\ref{rbound}. For the ``if'' part, we consider a given $r$ which is not a rate and find a functional $f$ which violates~\eqref{ineqrf}. Since $y\geq 0$, we have $R_{\min}=0$, and so the only way for $r$ not to be a rate is if $r>R_{\max}$. So we can choose $\lambda\in\Q$ with $R_{\max}<\lambda<r$ and consider the point $x - \lambda y\in W$. Since $\lambda$ is not a rate either, we know that $x - \lambda y\not\geq 0$. Hence Theorem~\ref{aovshb} gives us a functional $f$ with $f(x - \lambda y) < 0$, and therefore
\beq
f(x) < \lambda f(y) \leq r f(y),
\eeq
where the second inequality uses $\lambda<r$ and $f(y)\geq 0$.
\end{proof}

This results in a formula for rates on Archimedean ordered $\Q$-vector spaces:

\begin{thm}
\label{aovsrate}
If $W$ is an Archimedean $\Q$-vector space and $x,y\geq 0$ in $W$, then
\beq
 R_{\min}(x\to y) = 0, \qquad R_{\max}(x\to y) = \inf_f \frac{f(x)}{f(y)}, 
\eeq
where the infimum ranges over all functionals $f$ with $f(y)\neq 0$.
\end{thm}

\begin{proof}
Since $y\geq 0$, it is clear that $R_{\min} = 0$. Concerning $R_{\max}$, we know from Proposition~\ref{ratecrit} that $R_{\max}$ is the largest number for which $f(x)\geq r f(y)$ holds for all functionals $f$. Since $f(x)\geq 0$ and $f(y)\geq 0$, this inequality is equivalent to $r \leq \tfrac{f(x)}{f(y)}$ for all functionals $f$ with $f(y)\neq 0$.
\end{proof}

In the case that a generating pair exists, we do have an improved definition of rate:

\begin{defn}
\label{defrrate}
\begin{enumerate}
\item Let $A$ be an ordered commutative monoid with generating pair $(g_+,g_-)$. A nonnegative extended real number $r\in\Rplus\cup\{\infty\}$ is a \emph{regularized rate} from $x$ to $y$ if for every $\eps>0$ and for every neighbourhood of $r$ there exist a fraction $\tfrac{m}{n}$ in the neighbourhood and $k\in\N$ such that $k\leq\eps \max(m,n)$ and
\beqn
\label{rregineq}
nx + kg_+ \geq my + kg_-.
\eeqn
\item The set of all regularized rates is the \emph{regularized rate region}.
\end{enumerate}
\end{defn}

Every rate is also a regularized rate, since then one can simply take $k=0$. However, not every regularized rate is a rate: in Example~\ref{funnyrate}, we have a generating pair given by $g_+=e_3$ and $g_-=0$, and we have $\sqrt{2}$ as a regularized rate from $e_1$ to $e_2$ which is not a rate. As we will see below, whether a number $r$ is a regularized rate or not does not depend on the particular choice of generating pair.

Fortunately, in many cases there is no difference between rates and regularized rates:

\begin{prop}
\label{ratevsregrate}
Let $x$ and $y$ be elements of an ordered commutative monoid with generating pair $(g_+,g_-)$. If
\beqn
\label{maxgmin}
R_{\max}(x\to g_+) > 0,\qquad R_{\min}(g_-\to y) < \infty,
\eeqn
then $r$ is a regularized rate from $x$ to $y$ if and only if it is a rate from $x$ to $y$.
\end{prop}

\begin{proof}
Since any rate is trivially also a regularized rate, we only need to prove that if $r$ is a regularized rate, then it is also a rate. If $r$ is a regularized rate, then for every neighbourhood of $r$ and every $\eps>0$ we have a fraction $\tfrac{m}{n}$ in that neighbourhood and $k\leq \eps \max(m,n)$ with
\beq
nx + kg_+ \geq my + kg_-.
\eeq
By the assumption~\eqref{maxgmin}, we also have fractions $\tfrac{m_+}{n_+}$ with $m_+>0$ and $n_+ x\geq m_+ g_+$, and $\tfrac{m_-}{n_-}$ with $n_->0$ and $n_- g_- \geq m_- y$. This implies
\beq
(m_+ n_- n + n_+ n_- k ) x \geq m_+ n_- n x + m_+ n_- k g_+ \geq m_+ n_- m y + m_+ n_- k g_- \geq ( m_+ n_- m + m_+ m_- k)  y 
\eeq
Hence $x$ can be converted into $y$ at a rate of
\beq
\frac{m_+ n_- m + m_+ m_- k}{m_+ n_- n + n_+ n_- k} = \frac{m + \frac{m_-}{n_-}k}{n + \frac{n_+}{m_+} k}.
\eeq
Since the $m_\pm$ and $n_\pm$ all remain fixed as one takes the limit $n\to\infty$ or $m\to\infty$, this rate also tends to $r$ as both the original neighbourhood around $r$ and $\eps$ get smaller and smaller.
\end{proof}

\begin{ex}
\label{rreggrph}
Continuing Example~\ref{moregraphinvariants}, we can also express some regularized rates in $\Grph$ in terms of graph invariants. With $g_+=\mathcal{K}_2$ and $g_-=0$ as before, we indeed have a positive rate from any graph $\mathcal{G}$ with at least one edge to $\mathcal{K}_2$, since having an edge guarantees the existence of a graph map $\mathcal{K}_2\to\mathcal{G}$ and therefore $\mathcal{G}\geq \mathcal{K}_2$. Moreover, for any graph $\mathcal{H}$ we trivially have a rate $<\infty$ for going from $\mathcal{H}$ to $0$ since $\mathcal{H}\geq 0$. Hence Proposition~\ref{ratevsregrate} tells us that if $\mathcal{G}$ is a graph with at least one edge, then $\Rreg_{\max}(\mathcal{G}\to\mathcal{H})=R_{\max}(\mathcal{G}\to\mathcal{H})$. So up to taking the logarithm, complement and the reciprocal, $\Rreg_{\max}(\mathcal{G}\to \mathcal{K}_2)$ is the Shannon capacity of $\mathcal{G}$ as in Example~\ref{moregraphinvariants}, and $\Rreg_{\max}(\mathcal{K}_2\to\mathcal{G})$ is the fractional chromatic number.
\end{ex}

Many of the properties of rates also apply to regularized rates, and it is straightforward to prove the analogues of Lemma~\ref{carnotlem} to Lemma~\ref{ratenhom} for regularized rates. But regularized rates are better behaved in general, since the analogue of Proposition~\ref{ratecrit} holds even without the additional assumption $y\geq 0$:

\begin{prop}
\label{rregprop}
In an Archimedean ordered $\Q$-vector space $W$ with a generator $g$, a number $r\in\Rplus\cup\{\infty\}$ is a regularized rate from $x$ to $y$ if and only if
\beqn
\label{rregbound}
f(x)\geq r f(y)
\eeqn
for every (extremal) functional $f$.
\end{prop}

Putting ``extremal'' in brackets means that one can either consider all functionals or the extremal functionals only; the statement holds in either case. 

\newenvironment{heartproof}[1]{\begin{proof}[#1]\renewcommand*{\qedsymbol}{\(\heartsuit\)}}{\end{proof}}

\begin{heartproof}{\textsc{Proof.}}
For the ``only if'' part, we need to derive~\eqref{rregbound} from the assumption that $r$ is a regularized rate. We first consider the case that $r$ can be approximated by fractions as in Definition~\ref{defrrate} which are all finite. Then for given $\eps>0$, we have a finite fraction $\tfrac{m}{n}$ and $k\in\N$ as in the definition, where we can take $g_+=g$ and $g_-=0$. Applying $f$ to~\eqref{rregineq} results in
\beq
n f(x) + k f(g) \geq mf(y) ,
\eeq
which, due to $n>0$, can be rearranged to
\beq
f(x) + \frac{k}{n} f(g) \geq \frac{m}{n}\, f(y).
\eeq
In the limit $\tfrac{m}{n}\to r$ and $\eps\to 0$, this yields the claim~\eqref{rregbound}.

The only case in which this argument does not apply is if $r=\infty$ and the only fractions which witness this are of the form $\tfrac{m}{0}$. In this case, we start with $k\leq \eps m$ and $kg\geq my$, and therefore $\eps g\geq\tfrac{k}{m}g\geq y$, from which $0\geq y$ follows in the limit $\eps\to 0$ thanks to Lemma~\ref{approach}.

The ``if'' part is more complicated. Our earlier results like Theorem~\ref{aovshb} do not seem to cover all the cases that come up, and we need to dig a little bit deeper by resorting to Theorem~\ref{hbext} directly. We prove that if $r$ is not a regularized rate, then there is a functional $f$ which violates~\eqref{rregbound}. If $r=\infty$, then the assumption that $r$ is not a regularized rate is equivalent to $0\not\geq y$ again by Lemma~\ref{approach}, and we conclude by Theorem~\ref{aovshb}. Hence we can focus on the $r<\infty$ case from now on.

Consider the subspace $W_0 \defin \lin_{\Q}\{x,y\}$.
% It is at most two-dimensional by definition. If it is zero- or one-dimensional, then there is a linear dependence of the form $\alpha x + \beta y = 0$. If $\alpha \neq 0$, then we can divide by $\alpha$ and find that $x$ is a multiple of $y$, or $x=\gamma y$ for some $\gamma\in\Q$. Since $r$ is not a rate, we necessarily have $r\neq\gamma$. If $r>\gamma$, then we know that $0\not\geq y$, since otherwise we would have $R_{\max}=\infty$, which would also make $r$ into a rate since $\gamma$ is a rate and the rate region is an interval (Proposition~\ref{rateregion}). Therefore we can find a functional $f$ with $f(y)>0$, which then also violates~\eqref{rregbound}. Analogous reasoning deals with the case $r<\gamma$. Second, if $\alpha = 0$, then $\beta\neq 0$, and hence $y=0$. In this case, our to-be-proven claim states that if some $r$ is not a regularized rate from $x$ to $0$, then there is a functional $f$ with $f(x)<0$. This is true because if $x\geq 0$, then every $r$ is a regularized rate, and hence $x\not\geq 0$ yields $f(x)<0$ for suitable $f$.
% So we now assume $x$ and $y$ to be linearly independent. 
We will apply Theorem~\ref{hbext} to $W_0$ with respect to the sublinear map
\beqn
\label{rregp}
p : W \lra \R,\qquad z\longmapsto \inf\{\: \mu\in\Q \:|\: \mu g \geq z \:\}. 
\eeqn
Here, the infimum is guaranteed to be finite for two reasons: first, we know that there exists a $\mu\in\Q$ with $\mu g\geq z$. Second, the set of all such $\mu$ is bounded below: choosing $\kappa\in\Q$ with $\kappa g\geq -z$ yields, for $\mu g \geq z$,
\beq
(\mu + \kappa) g\geq z - z = 0.
\eeq
So if there is a $\mu$ which is less than $-\kappa$, then we conclude $-g\geq 0$. But if this is the case, then we have $W=0$, in which case there is nothing to prove. Hence we can assume $\mu \geq -\kappa$, which is the desired lower bound making $p$ finite. We think of $p$ as analogous to~\eqref{supp}, where $p$ was defined as the maximal value of a continuous function.

The remainder of the proof consists in constructing a linear map $f_0:W_0\to\R$ which violates~\eqref{rregbound} and such that $f_0$ is $p$-dominated. Then it can be extended to a $p$-dominated linear map $f:W\to\R$ by Theorem~\ref{hbext}. The inequality $f(z)\leq p(z)$ for all $z\in W$ also guarantees that $f$ is a functional, i.e.~that $f(W_+)\subseteq \Rplus$, since $z\geq 0$ implies $p(-z)\leq 0$, and hence $-f(z) = f(-z)\leq p(-z) \leq 0$. As an extension of $f_0$, this $f$ also violates~\eqref{rregbound}.

In order to construct $f_0$, we consider the affine line
\beq
L \defin \{\: x - \lambda y \:|\: \lambda \in \Q \:\},
\eeq
and distinguish six (not entirely disjoint) cases:
\begin{enumerate}
\item $L$ intersects $W_+$ at $\lambda > r$, i.e.~there exists $\lambda\in\Q$ greater than $r$ such that $x-\lambda y \geq 0$.

In this case, we choose $\mu>r$ with $x-\mu y\not\geq 0$, which exists because $r$ is not a rate. We then obtain an $f$ with $f(x-\mu y) < 0$ from Theorem~\ref{aovshb} directly. Then $f(x) < \mu f(y)$, but $f(x) \geq \lambda f(y)$. Since $f$ is linear along $L$ and $r<\mu<\lambda$, we conclude $f(x) < r f(y)$, as desired.
\item $L$ intersects $W_+$ at $\lambda < r$. This is analogous to the previous case.
\item $L$ intersects $W_+$ exactly at $\lambda = r$, i.e.~$x \geq r y$. This can only happen if $r$ is rational, and is in contradiction with the assumption that $r$ is not a rate.
\item $L$ does not intersect $W_+$, but almost intersects $W_+$ in the $\lambda\to\infty$ direction. By this we mean that for every $\eps\in\Qpos$ there exists $\lambda > r$ such that $x - \lambda y + \eps g \geq 0$. 

In this case, for every $n\in\N$ we choose $\lambda_n\in\Qpos$ with $x - \lambda_n y + \tfrac{1}{n} g \geq 0$ and $\kappa\in\Q$ such that $\kappa g\geq x$. This results in
\beq
-y + \frac{\kappa + \tfrac{1}{n}}{\lambda_n} g \geq 0 .
\eeq
Since $\lambda_n\to\infty$ as $n\to\infty$, we conclude $-y \geq 0$ from Lemma~\ref{approach}, and record this for later use.

Now we choose a rational $\mu>r$ with $x-\mu y\not\geq 0$. In particular, we have $p(\mu y-x) > 0$ because $W$ is Archimedean. On the one-dimensional subspace $\Q\cdot(x-\mu y)$, we can specify $f_0$ by taking $f_0(\mu y - x)\defin p(\mu y-x)$ and extending linearly. Then $f_0$ is $p$-dominated on this one-dimensional space thanks to positive homogeneity and
\beq
f_0(\mu y - x) = p(\mu y - x),\qquad f_0(x - \mu y) = - p(\mu y - x) \leq p(x - \mu y),
\eeq
where the last inequality is due to subadditivity and $p(0)=0$. A first application of Theorem~\ref{hbext} then lets us extend this to the desired $f_0$ on $W_0$, and it remains to show that this extension satisfies $f_0(x) < r f_0(y)$. To this end, we use $-y\geq 0$, which implies $f_0(y)\leq 0$ and the claim follows from $f_0(x) < \mu f_0(y)\leq r f_0(y)$.
\item $L$ does not intersect $W_+$, but almost intersects $W_+$ in the $\lambda\to -\infty$ direction. This is analogous to the previous case.
\item $L$ does not intersect $W_+$ at all, not even almost. This means that there exists $\eps>0$ such that $x-\lambda y + \eps g \not\geq 0$ for all $\lambda\in\Q$. We can reformulate this as $p(\lambda y - x)\geq \eps$ for all $\lambda$.

We then claim that $p(y)\geq 0$, and also $p(-y)\geq 0$. For if we had $p(y)<0$, then there would be $\kappa>0$ such that $0 \geq y + \kappa g$, and hence also $-y\geq \gamma \cdot (-x)$ for a certain $\gamma>0$. This would result in $x-\gamma^{-1}y\geq 0$, in contradiction to the assumption that $L$ does not intersect $W_+$. The argument for $p(-y)\geq 0$ is analogous.

Furthermore, in the present case we also know that $x$ and $y$ must be linearly independent. Hence we can specify the linear map $f_0 : W_0\to\R$ by specifying its value on $x$ and on $y$, and for these we take
\beq
f_0(x)\defin - \inf_{\lambda\in\Q} p(\lambda y - x),\qquad f_0(y)\defin 0 .
\eeq
Here, our current assumption guarantees that the first value is indeed negative, so that these assignments yield a violation of~\eqref{rregbound}.

We still need to check that $f_0(z) \leq p(z)$ for all $z\in W_0$, i.e.~that $-\alpha\, \inf_\lambda p(\lambda y - x) \leq p(\alpha x + \beta y)$ for all $\alpha,\beta\in\Q$. By positive homogeneity of $p$, it is sufficient to consider the cases $\alpha\in\{-1,0,+1\}$.
\begin{enumerate}
\item If $\alpha=-1$, we need to prove $\inf_\lambda p(\lambda y - x)\leq p(\beta y - x)$, which is trivial.
\item If $\alpha=0$, we need to prove $p(\beta y)\geq 0$ for all $\beta$, which follows from $p(y)\geq 0$ and $p(-y)\geq 0$ by positive homogeneity.
\item If $\alpha=+1$, we need to prove $-\inf_\lambda p(\lambda y - x) \leq p(x+\beta y)$, which is equivalent to $p(\lambda y - x) + p(\beta y + x) \geq 0$ for all $\beta$ and $\lambda$. This is due to subadditivity and $p((\lambda + \beta)y)\geq 0$.
\end{enumerate}
In all cases, the desired bound holds, so that $f_0$ has all the required properties.
\end{enumerate}
Finally, the restriction to extremal functionals is not essential, since we can restrict to normalized functionals without loss of generality and then resort to the Krein-Milman theorem as in the proof of Theorem~\ref{rieszrepthm}.
\end{heartproof}

\begin{rem}
In light of this result, we hope that if one performs the additional step of regularization to Archimedean ordered $\R$-vector spaces (Remark~\ref{QvsR}), then the necessary and sufficient condition for $r<\infty$ to be a regularized rate is that $x\geq ry$. In fact, this could be a potential future \emphalt{definition} of regularized rate, which would differ from the present one in that it would not be allowed to be infinite, but it would be allowed to be negative.
\end{rem}

\begin{rem}
\label{minmaxregrates}
In particular, Proposition~\ref{rregprop} implies that the regularized rate region is a closed interval as well, and we denote its endpoints by $\Rreg_{\min}$ and $\Rreg_{\max}$, respectively. Another consequence is that the regularized rate region does not depend on the choice of generator, although this could have been derived more easily directly from Definition~\ref{defrrate}. We also note that Lemma~\ref{ratehom} applies to regularized rates as well.
\end{rem}

This result generalizes to all ordered commutative monoids $A$ with a generating pair if we can show that the regularized rate region does not change under regularizing $A$. This indeed turns out to be the case:

\begin{prop}[{generalizes~\cite[Corollary~6.8.4]{Ktheoryop}}]
\label{rrind}
A number $r\in\Rplus\cup\{\infty\}$ is a regularized rate for $x$ to $y$ in an ordered commutative monoid $A$ with a generating pair if and only if it is so in $\aovs(A)$.
\end{prop}

The same can clearly not hold for the ``naive'' notion of rate of Definition~\ref{defrate}, since $A$ may not be cancellative, so that already the rate region of $\oag(A)$ may be bigger than that of $A$ itself. But as shown in the proof of Theorem~\ref{ocmhb}, allowing the ``seeds'' $kg_+$ and $kg_-$ in the definition of regularized rate automatically removes that obstruction by taking catalysis into account.

\begin{proof}
If $r$ is a regularized rate in $A$, then it clearly is one in $\aovs(A)$ as well.

The main part of the proof consists in showing the converse. We take the generator of $\aovs(A)$ to be $g_+-g_-$. If $r$ is a regularized rate in $\aovs(A)$, then for every neighbourhood of $r$ and every $\eps>0$ there exist a fraction $\tfrac{m}{n}$ and $k\in\N$ with $k\leq \eps \max(m,n)$ and such that $n x + k g_+ \geq m y + k g_-$. By Theorem~\ref{ocmhb}, in terms of $A$ the latter means in particular that we can find $j,l\in\Npos$ with $l\leq \eps j\max(m,n)$ and such that
\beq
jnx + (jk+l) g_+ \geq jmy + (jk+l)g_-.
\eeq
Then the fraction $\tfrac{jm}{jn}$ still lives in the same neighbourhood of $r$, and moreover we have
\beq
jk+l\leq \eps j\max(m,n) + \eps j\max(m,n) = 2\eps\max(jm,jn).
\eeq
This is all that is required for showing that $r$ is a regularized rate in $A$.
\end{proof}

This result implies immediately:

\begin{cor}
\label{ocmrrate}
Proposition~\ref{rregprop} generalizes to any ordered commutative monoid $A$ with a generating pair: $r$ is a regularized rate from $x$ to $y$ if and only if
\beq
f(x) \geq r f(y)
\eeq
holds for all (extremal) functionals $f$.
\end{cor}

This yields another one of our main results, which is a formula for regularized rates:

\begin{thm}
\label{rateformula}
In an ordered commutative monoid with a generating pair and $x,y\geq 0$, we have
\beq
\boxed{\Rreg_{\min}(x\to y) = 0, \qquad \Rreg_{\max}(x\to y) = \inf_f \frac{f(x)}{f(y)}}
\eeq
where the infimum ranges over all (extremal) functionals $f$ with $f(y)\neq 0$.
\end{thm}

\begin{proof}
As for Theorem~\ref{aovsrate}.
\end{proof}

\begin{ex}
In $\Grph$, we can apply the rate formula together with Example~\ref{rreggrph} to yield a formula for the Shannon capacity of a graph:
\[
\log_2 \Theta(\overline{\mathcal{G}}) = \Rreg_{\max}(\mathcal{G}\to \mathcal{K}_2) = \inf_f f(\mathcal{G}),
\]
where the infimum ranges over all (extremal) functionals $f:\Grph\to\R$ which satisfy the normalization $f(\mathcal{K}_2) = 1$. In terms of multiplicative functionals, exponentiating this equation proves that the Shannon capacity of a graph $\mathcal{G}$ is given by
\beq
\label{scapform}
\boxed{ \Theta(\mathcal{G}) = \inf_f f(\overline{\mathcal{G}}) }
\eeq
where the infimum now ranges over all graph invariants that are monotone under graph maps, multiplicative under disjunctive products, and normalized such that $f(\mathcal{K}_2)=2$.
\end{ex}

While the rate formula of Theorem~\ref{rateformula} is rather unwieldy in general, we hope that it will provide a useful method to compute rates in some situations:

\begin{ex}
\label{probmajconj}
If our conjectural characterization of extremal functionals in $\ProbMaj$ (Example~\ref{renyi}) is correct, then the regularized maximal rates can in this example be computed as
\beq
\Rreg_{\max}((\inalph,P)\to(\outalph,Q)) = \inf_{t\in[0,\infty]} \frac{H_t(P)}{H_t(Q)}.
\eeq
For given $P$ and $Q$, this quotient of R\'enyi entropies is straightforward to evaluate and minimize.
\end{ex}

\newpage
\section{Notions of one-dimensionality: numerical ordered commutative monoids}
\label{sectonedim}

It is an interesting question under which conditions the ordering $x\geq y$ can be completely characterized by a single functional in the sense that
\beqn
\label{numeq}
x\geq y \quad\Longleftrightarrow\quad f(x)\geq f(y).
\eeqn
In this section, we study this question abstractly and derive two characterizations. These are similar to results of von Neumann and Morgenstern in decision theory and of Lieb and Yngvason in thermodynamics.

\begin{defn}
\label{numdefn}
An ordered commutative monoid $A$ is \emph{numerical} if there exists $f$ with~\eqref{numeq}.
\end{defn}

We can understand~\eqref{numeq} as saying that $f$ embeds $A$ into the real line $\R$. This makes $A$ \emph{one-dimensional} in a certain sense, and the study of numerical ordered commutative monoids is equivalent to the study of numerical ordered submonoids of $\R$. The term ``numerical'' reflects this idea: the elements of $A$ can be regarded as real numbers. In resource-theoretic terms, there is a measure of value of a resource object which completely reflects the convertibility relation, and we can therefore identify every resource object with its value. This embedding into $\R$ is similar to how the numerical semigroups of Example~\ref{numsg} are submonoids of $\N$; the important difference is that we consider $\N$ and hence also any numerical semigroup to carry the trivial ordering, while here $\R$ is equipped with the standard ordering.

For being numerical, it is obviously necessary for $A$ to be \emph{totally ordered}: for every $x,y\in A$, we need to have $x\geq y$ or $y\geq x$. For an Archimedean ordered $\Q$-vector space, this turns out to be sufficient:

\begin{prop}
\label{numaovs}
Let $W$ be an Archimedean ordered $\Q$-vector space. Then the following conditions are equivalent:
\begin{enumerate}
\item\label{Wembed} $W$ is numerical.
\item\label{Wfun} There is $f:W\to\R$ such that every other functional is a scalar multiple of $f$.
\item\label{Wuniquefun} If $f:W\to\R$ is nonzero, then every other functional is a scalar multiple of $f$.
\item\label{Wcomplconv} The complement $W\setminus W_+$ is convex.
\item\label{Wpartition} $W$ is totally ordered, i.e.~$W_+\cup (-W_+) = W$.
\end{enumerate}
\end{prop}

\begin{proof}
\begin{enumerate}
\item[\implproof{Wembed}{Wfun}] 
We start with the case $W=\R$ and take $f$ to be the identity functional $\R\to\R$. Let $f':\R\to\R$ be some other functional. Then we claim that $f'(x) = f'(1) x$ for all $x\in\R$. First, since $\R$ is an ordered $\Q$-vector space and every homomorphism of $\Q$-vector spaces is automatically $\Q$-linear, we conclude $f'(x) = f'(1) x$ for all $x\in\Q$. Then the claim follows from approximating every $x\in\R$ by a sequence of rational numbers from above and another sequence from below, using monotonicity of $f'$.

For general numerical $W$, let $f:W\to\R$ be a functional which reflects the order. Then either $W\cong\{0\}$, or $f$ realizes $W$ as a dense subset of $\R$. In the first case there is nothing to prove, so we focus on the second. Any other $f':W\to\R$ extends uniquely to a functional $\R\to\R$ by density, and this we already know to be a multiple of the identity.
\item[\implproof{Wfun}{Wuniquefun}] Straightforward.
\item[\implproof{Wuniquefun}{Wcomplconv}] If $W\cong\{0\}$, then there is nothing to prove. Otherwise, there is $x\in W$ with $x\not\geq 0$, and hence a functional $f:W\to\R$ with $f(x)<0$ by Theorem~\ref{aovshb}. In particular, $f$ is nonzero, and hence every other functional is a scalar multiple of $f$. But then again by Theorem~\ref{aovshb}, we have $W_+ = f^{-1}(\Rplus)$ and $W\setminus W_+ = f^{-1}(\R\setminus\Rplus)$. Now the claim follows because $\R\setminus\Rplus$ is convex, and the inverse image of a convex set under a $\Q$-linear map is convex too\footnote{Recall that for us, ``convex'' always means convex with respect to rational coefficients.}.
\item[\implproof{Wcomplconv}{Wpartition}] 
For $x\in W\setminus W_+$ we have $-x\in W_+$, for otherwise $-x\in W\setminus W_+$ would imply $0\in W\setminus W_+$ by convexity.

\item[\implproof{Wpartition}{Wembed}]
In case that every $x\in W_+$ also satisfies $-x\in W_+$, then we have $W\cong\{0\}$ and $f:x\mapsto 0$ does the job.

Otherwise, we fix $g\in W_+$ with $-g\not\in W_+$. For every $x\in W$, we define
\beqn
\label{infsup}
f(x) \defin \inf \{\: \lambda\in\Q \:|\: \lambda g\geq x \:\} = \sup \{\: \mu\in\Q \:|\: x \geq \mu g \:\} ,
\eeqn
where the infimum is an expression that we have already met in~\eqref{rregp}. To see the equality between the infimum and the supremum, the inequality ``$\geq$'' is clear since $\lambda g\geq x\geq \mu g$ implies $(\lambda - \mu) g\geq 0$, and therefore $\lambda - \mu\geq 0$ by $-g\not\geq 0$. Now suppose that there was a gap between the infimum and the supremum. This gap would contain some $\nu\in\Q$, resulting in $\nu g\not\geq x$ and $x\not\geq\nu g$. This contradicts the assumption that $x-\nu g$ must lie in either $W_+$ or $-W_+$. Next, we show that $f(x)\in\R$, i.e.~that it is neither $-\infty$ nor $+\infty$. That it is not $-\infty$ follows from the definition as an infimum: if there was a sequence $(\lambda_n)_{n\in\N}$ with $\lambda_n\leq 0$ and $\lambda_n\to -\infty$, then this would imply that $g\leq 0$ by $g + |\lambda_n|^{-1} x \leq 0$ and Lemma~\ref{approach}. That it is not $+\infty$ follows similarly from the supremum characterization. We have shown in passing that $g$ is a generator.

Hence $f:W\to\R$ is well-defined, and we claim that it is a homomorphism. First, $f(x)\geq 0$ for $x\geq 0$ is clear by the definition as a supremum. Second, $f(x+y)\geq f(x)+f(y)$ also follows from the definition as a supremum; the other inequality direction follows from the infimum characterization. In order to prove~\eqref{numeq}, it remains to show that $f(x)\geq 0$ implies $x\geq 0$. But this in turn is another direct consequence of Lemma~\ref{approach}: $f(x)\geq 0$ means that there is a sequence $(\mu_n)_{n\in\N}$ with $\mu_n\leq 0$ and $\mu_n\to 0$ and such that $x + |\mu_n| g\geq 0$.
\qedhere
\end{enumerate}
\end{proof}

We would now like to extend the previous result to the world of ordered commutative monoids. This is less straightforward: in order for an ordered commutative monoid $A$ to be numerical, it is not sufficient that the order is total. The reason is that in order for~\eqref{numeq} to hold, it is necessary for the canonical homomorphism $A\to\aovs(A)$ to reflect the order as well, since any functional $f$ factors through $\aovs(A)$.

\begin{ex}
Take $A\defin\R^2$ with the usual addition and the \emphalt{lexicographic ordering}, in which the positive cone consists of all those $(x,y)$ which satisfy either $x>0$ or $x=0$ and $y\geq 0$. Since either $(x,y)$ or $-(x,y)$ satisfy this condition, this $A$ is an ordered $\Q$-vector space which is totally ordered. However, $\aovs(A)_+$ also contains all points $(x,y)$ with $x=0$ and $y<0$, so that the homomorphism $A\to\aovs(A)$ does not reflect the order. Thus $A$ is not numerical.
\end{ex}

In the following, we write ``$x>y$'' as shorthand for ``$x\geq y$ and $y\not\geq x$''. Even in a totally ordered commutative monoid, $x_1 > y_1$ and $x_2 > y_2$ does not necessarily imply $x_1 + x_2 > y_1 + y_2$: for example, in the totally ordered commutative monoid consisting of $\{0,1,2\}$ with its usual order and the usual addition truncated at $2$, we have $2>1$, but $2+2 \not> 1+1$. This is the main reason for why we generally prefer to work with ``$\geq$'' rather than ``$>$''; we introduce the latter mainly in order to formulate the following theorem.

\begin{thm}
\label{numocm}
Let $A$ be an ordered commutative monoid. Then the following conditions are equivalent:
\begin{enumerate}
\item\label{Anum} $A$ is numerical via $f:A\to\R$.
\item\label{AvNM} $A$ is totally ordered, and satisfies the conditions:
\begin{enumerate}
\item\label{vNMcanc} if $x>y$, then $x + z > y + z$ for all $z$;
\item\label{vNMarch} if $x>y>z$, then there exists $\eps\in(0,1)$ such that for all $n,k\in\N$ with $k\leq \eps n$,
\beq
(n-k) x + k z > n y > k x + (n - k) z.
\eeq
\end{enumerate}
\item\label{ALY} $A$ is totally ordered, and if $x,y,z_+,z_-\in A$ are such that for all $\eps>0$ there exist $n,k\in\N$ with $k\leq \eps n$ and
\beq
nx + kz_+ \geq ny + kz_-,
\eeq
then $x\geq y$.
\end{enumerate}
If these conditions hold, then any functional $f':A\to\R$ is a scalar multiple of the above $f$.
\end{thm}

\begin{proof}
\begin{enumerate}
\item[\implproof{Anum}{AvNM}] Straightforward.
\item[\implproof{AvNM}{ALY}] 
In the totally ordered setting,~\ref{vNMcanc} is the contrapositive of cancellativity: if $x+z\leq y+z$, then $x\leq y$. So we know that $A$ is cancellative. We prove the contrapositive of the second condition in~\ref{ALY} by showing that if $x<y$, then there is $\eps>0$ such that for all $n,k\in\N$ with $k\leq \eps n$, we also have $n x + k z_- < n y + k z_+$. 

If $z_+\leq y$, then $x<y$ clearly implies $(n-k) x + k z_+ < (n-k) y + k y = n y$ whenever $k < n$. If on the other hand $z_+ > y$, then the same inequality $(n - k) x + k z_+ < n y$ follows from~\ref{vNMarch} for suitable $\eps>0$. So this inequality holds in all cases for all $n,k\in\N$ with $k\leq \eps n$. In the same way, we can derive $n x < (n - k) y + k z_-$. Adding these two inequalities results in $(2n - k) x + k z_+ < (2n - k) y + k z_-$. We can weaken this to $2n x + k z_+ < 2n y + k z_-$, which is enough after a suitable adjustment to $\eps$.

\item[\implproof{ALY}{Anum}] By Proposition~\ref{numaovs}, it is enough to prove that the canonical homomorphism $A\to\aovs(A)$ reflects the order.

First, we choose $g_+,g_-\in A$ with $g_+\geq g_-$, but $g_-\not\geq g_+$; if such a pair does not exist, then we have $A\cong\{0\}$, in which case there is nothing to prove. As the notation suggests, we claim that $(g_+,g_-)$ is a generating pair. So for given $w\in A$, we need to show that there exists $n\in\N$ with $ng_+\geq w + ng_-$. This is indeed the case, for by totality we would otherwise have $ng_+ \leq w + ng_-$ for all $n\in\N$, which by assumption implies $g_+\leq g_-$, in contradiction with $g_-\not\geq g_+$. Similarly we can show $ng_+ + w \geq ng_-$ for some $n\in\N$. The claim that $x\geq y$ in $\aovs(A)$ implies $x\geq y$ in $A$ now follows from Theorem~\ref{ocmhb} and the assumed implication.
\end{enumerate}\medskip

The final claim about $f':A\to\R$ follows from the Proposition~\ref{numaovs}\ref{Wuniquefun} since $f'$ factors uniquely through $\aovs(A)$.
\end{proof}

\begin{ex}
$\Grph$ is not totally ordered, and therefore not numerical. For example, the pentagon $\pentagon$ and $\mathcal{K}_4$ are incomparable.
\end{ex}

\begin{ex}
In~\ref{AvNM}, the cancellativity requirement~\ref{vNMcanc} is essential. Consider $A=\{0,1\}$ with ordering $0<1$ and addition $1+1=1$. Then $A$ is totally ordered, and~\ref{vNMarch} vacuously holds since $A$ contains no sequence of three strictly ordered elements. However, $A$ is not numerical. 
\end{ex}

\subsection*{The von Neumann--Morgensterm theorem} 

We have formulated condition~\ref{numocm}\ref{AvNM} such as to bear strong resemblance to the \emph{utility theorem} of von Neumann and Morgenstern in decision theory and the foundations of economics~\cite[Section~3]{games},~\cite[Chapter~5]{choice}. There, an agent's set of possible outcomes carries a mathematical structure very similar to an ordered commutative monoid. First, preferring some outcomes over others turns the set of outcomes into an ordered set. Second, every convex combination of outcomes represents a ``lottery'' between outcomes, and this lottery can itself be regarded as an outcome, so that the set of outcomes comes equipped with a convex combination operation compatible with the ordering relation. The theorem of von Neumann and Morgenstern provides necessary and sufficient conditions for such a structure to allow for a convex-linear function to $\R$ which preserves and reflects the order; such a function is a \emph{utility function}. Although their axioms are very close to the conditions~\ref{AvNM}, the mathematical setup is a bit different and it is not immediately clear how to derive the von Neumann--Morgenstern theorem from our result or vice versa.

We can also interpret our Theorem~\ref{numocm} in terms of an agent's preferences over outcomes. Instead of combining outcomes via probabilistic mixtures or ``lotteries'', our result applies when one considers actual conjunctive combinations of outcomes, so that one can express things like ``I prefer having a car over a bicycle, but I'd rather have a bicycle and a house instead of only a car''. The theorem provides necessary and sufficient conditions for the existence of a utility function in this kind of setup.

\subsection*{The Lieb--Yngvason theorem}

In the foundations of thermodynamics, Lieb and Yngvason were concerned with the notion of adiabatic accessibility between states of a physical system. Modulo some inessential technicalities, their axioms ``A1'' to ``A5''~\cite{secondlaw1} say that states in thermodynamics form an ordered semimodule over the semiring $\Rplus$, where the ordering is given by adiabatic accessibility, addition corresponds to placing two systems side by side, and scalar multiplication corresponds to scaling the size of a system. The theorem of Lieb and Yngvason provides necessary and sufficient conditions for such a semimodule to allow for an affine-linear function to $\R$ which preserves and reflects the order; such a function is an \emph{entropy function}. At the mathematical level, the theorem of Lieb and Yngvason is very similar to the earlier result of von Neumann and Morgenstern\footnote{In fact, the Lieb--Yngvason theorem has already been applied in decision theory as well~\cite{knightuncertain}.}; the~\ref{AvNM}$\Rightarrow$\ref{ALY} implication of our proof could indicate how to obtain the Lieb--Yngvason theorem as a corollary of the von Neumann--Morgenstern theorem.

In thermodynamics, using scalar multiplication by $\Rplus$ is an idealization corresponding to the macroscopic limit, since actual physical systems cannot be divided arbitrarily into smaller and smaller parts. It should therefore be interesting to investigate Theorem~\ref{numocm} as an improved version of the Lieb--Yngvason applicable also to systems of any size. In fact, we conjecture that thermodynamics with microscopic systems should start with a plain ordered commutative monoid $A$, and taking the macroscopic limit then corresponds to working with the regularization $\aovs(A)$, at least modulo the issues of Remarks~\ref{epsilonification} and~\ref{QvsR}. It is conceivable that this is precisely how Lieb--Yngvason thermodynamics arises from the single-shot thermodynamics of~\cite{secondlaws}.

A more detailed investigation on the applicability of our results in thermodynamics and decision theory and on the relation to the theorems of von Neumann--Morgenstern and Lieb--Yngvason will have to be performed in dedicated work. The remainder of this section could also be relevant for these investigations.

\bigskip

While results along the lines of Theorem~\ref{numocm} are interesting, we strongly expect them to be too limited for many purposes. Being numerical is an extremely strong condition on an ordered commutative monoid, and it is likely to be too strong to hold in many cases of interest. A more generous but still useful requirement is for $\aovs(A)$ to be numerical. In the following, we would like to study in more detail under which conditions this happens. It turns out to be closely related to the notion of regularized rate from Section~\ref{sectrates}.

\begin{defn}
\label{defpos}
An ordered commutative monoid $A$ is \emph{positive} if $x\geq 0$ for all $x\in A$.
\end{defn}

If $A$ is positive and numerical, then it embeds into $\Rplus$. In resource-theoretic terms, positivity means that all resource objects are freely disposable. We focus on the positive case for the sake of technical simplicity, leaving open the question of how the following considerations generalize to the non-positive case.

\begin{lem}
\label{posrateinf}
If $A$ is positive and has a generating pair, then $\Rreg_{\max}(x\to y) = \infty$ if and only if $y=0$ in $\aovs(A)$.
\end{lem}

\begin{proof}
If $\Rreg_{\max}=\infty$, then every $\lambda\in\Qplus$ is a regularized rate due to positivity. By Corollary~\ref{ocmrrate} and Theorem~\ref{ocmhb}, this means that $x\geq\lambda y$ for all $\lambda\in\Q$ in $\aovs(A)$. Equivalently $- y + \eps x \geq 0$ for every $\eps>0$. Hence $-y\geq 0$ in $\aovs(A)$ by Lemma~\ref{approach}, and therefore $y=0$ in $\aovs(A)$ by positivity. 

The converse $\Rreg_{\max}(x\to 0) = \infty$ follows from $nx\geq 0$ for all $n\in\N$.
\end{proof}

We assume the existence of a resource object $\$$ such that $(\$,0)$ is a generating pair, meaning that every other resource object can be produced from a certain number of $\$$'s, and can also be made to disappear for a certain number of $\$$'s. However, if $n\in\N$ is the smallest integer for which $n\$\geq x$, i.e.~so that $x$ can be bought in exchange for $n\$$, it does not follow that $x\geq n\$$, i.e.~that $x$ can also be sold for $n\$$; in economical terms, there may be a gap between the \emph{buying price} and the \emph{selling price}. However, if our business operates on a large scale and we have access to all sorts of tools (Example~\ref{tools}) and work and the level of economy of scale (Section~\ref{sectovs}) and allow even for the use of seeds (Section~\ref{sectaovs}), then the situation can be different, and the buying price and selling price that applies to us per copy of $x$ might be the same.

\begin{thm}
\label{aovsonedim}
For a positive ordered commutative monoid $A$ with generator $\$$, the following are equivalent:
\begin{enumerate}
\item\label{aovspnum} $\aovs(A)$ is numerical.
\item\label{funpnum} There is $f:A\to\R$ such that every other functional is a scalar multiple of $f$.
\item\label{uniquefunpnum} If $f:A\to\R$ is nonzero, then every other functional is a scalar multiple of $f$.
\item\label{currencypnum} If $\Rreg_{\max}(\$\to x) < \infty$, then
\beq
\Rreg_{\max}(\$\to x) \cdot \Rreg_{\max}(x\to\$) = 1.
\eeq
\item\label{ratepnum} If $\Rreg_{\max}(x\to y)<\infty$ and $\Rreg_{\max}(y\to x)<\infty$, then
\beqn
\label{reciprocalrates}
\Rreg_{\max}(x\to y) \cdot \Rreg_{\max}(y\to x) = 1.
\eeqn
\item\label{pnumf} There exists a functional $f:A\to\Rplus$ such that
\beq
x\geq y \quad\Longleftrightarrow\quad f(x)\geq f(y) \text{\normalfont{} in } \aovs(A).
\eeq
\end{enumerate}
\end{thm}

\begin{proof}
Since the functionals $A\to\R$ are in bijective correspondence with functionals $\aovs(A)\to\R$, we already know the equivalence of~\ref{aovspnum}--\ref{uniquefunpnum} from Proposition~\ref{numaovs}. And condition~\ref{pnumf} is just a restatement of~\ref{aovspnum}, so we ignore~\ref{pnumf}. 
\begin{enumerate}
\item[\implproof{uniquefunpnum}{currencypnum}] If $\aovs(A)\cong \{0\}$, then the condition $R_{\max}(\$\to x)<\infty$ is never satisfied due to Lemma~\ref{posrateinf}, and hence there is nothing to be shown. Therefore we assume that there is at least some $y\in A$ with $y>0$ in $\aovs(A)$, and then Theorem~\ref{aovshb} yields a functional $f$ with $f(y)>0$, to which we can apply the assumption.

For $\Rreg_{\max}(\$\to x)<\infty$, we know that $x>0$ in $\aovs(A)$. By assumption, every functional is a scalar multiple of $f$; hence we can compute regularized rates as in Theorem~\ref{rateformula} without the infimum, i.e.~as a simple quotient of values of $f$. This results in
\beq
\Rreg_{\max}(\$\to x) \cdot \Rreg_{\max}(x\to \$) = \frac{f(\$)}{f(x)} \cdot \frac{f(x)}{f(\$)} = 1,
\eeq
as was to be shown.
\item[\implproof{currencypnum}{ratepnum}] Together with positivity, the two finiteness assumptions in~\ref{ratepnum} guarantee that $x>0$ and $y>0$ in $\aovs(A)$ by Lemma~\ref{posrateinf}. But then also $\Rreg_{\max}(\$\to x)<\infty$ and $\Rreg_{\max}(\$\to y)<\infty$, so that~\ref{currencypnum} is applicable. We also obtain $\Rreg_{\max}(x\to x) = 1$ and $\Rreg_{\max}(y\to y) = 1$ from the analogue of Lemma~\ref{ratedicho} for regularized rates. Thanks to the corresponding analogue of~\eqref{carnot}, this results in
\beq
\Rreg_{\max}(x\to y) \cdot \Rreg_{\max}(y\to x) \leq 1.
\eeq
It remains to prove the other direction of this inequality. To this end, the regularized rates analogue of~\eqref{carnot} is useful again,
\begin{align*}
& \Rreg_{\max}(x\to y) \geq \Rreg_{\max}(x\to\$) \cdot \Rreg_{\max}(\$\to y),\\
& \Rreg_{\max}(y\to x) \geq \Rreg_{\max}(y\to\$) \cdot \Rreg_{\max}(\$\to x).
\end{align*}
These two inequalities can be multiplied to
\begin{align*}
\Rreg_{\max}(x\to y)  \cdot \Rreg_{\max}(y\to x) & \geq \Rreg_{\max}(x\to\$) \cdot \underbrace{\Rreg_{\max}(\$\to y)\cdot \Rreg_{\max}(y\to\$)}_{=1} \cdot \Rreg_{\max}(\$\to x) \\
& = \Rreg_{\max}(x\to\$) \cdot \Rreg_{\max}(\$\to x) \\
& = 1.
\end{align*}
\item[\implproof{ratepnum}{aovspnum}] By Proposition~\ref{numaovs}, we need to prove that $\aovs(A)$ is totally ordered. To this end, it is enough to show that $A$ itself is totally ordered in the ordering induced from $\aovs(A)$. For given $x,y\in A$, we distinguish two cases:
\begin{itemize}
\item If $\Rreg_{\max}(x\to y) \geq 1$, then Corollary~\ref{ocmrrate} and $\Rreg_{\min} = 0$ guarantee that $f(x) \geq f(y)$ for all functionals $f$. Together with Theorem~\ref{ocmhb}, this shows $x\geq y$ in $\aovs(A)$.
\item If $\Rreg_{\max}(x\to y) < 1$, then we have $y>0$ by Lemma~\ref{posrateinf}. But then also $\Rreg_{\max}(y\to y)<\infty$, which necessitates $\Rreg_{\max}(y\to y) = 1$ by the regularized rate analogue of Lemma~\ref{ratedicho}. Hence the analogues of inequalities~\eqref{carnot} tell us
\beq
\Rreg_{\max}(x\to y)\cdot \Rreg_{\max}(y\to x) \leq 1,
\eeq
which is relevant because it lets us conclude $\Rreg_{\max}(y\to x) < \infty$. Then the assumption applies, and the inequality is saturated. Therefore $\Rreg_{\max}(y\to x) = \Rreg_{\max}(x\to y)^{-1} > 1$, which takes us back to the first case with $x$ and $y$ swapped, and we conclude $y\geq x$.
\end{itemize}
So in each case, we have $x\geq y$ or $y\geq x$, as was to be shown.
\qedhere
\end{enumerate}
\end{proof}

The resource-theoretic interpretation of this is that if $A$ satisfies the conditions of the theorem, then it enjoys a kind of \emph{perfect asymptotic interconvertibility}: all resource objects are interchangeable at the many-copy level and with the seed regularization. Exchanging one for another may incur a small overhead which grows sublinearly in the number of copies that are being converted. But up to this, all conversions are reversible and can be undone. 

Ordered commutative monoids which satisfy the properties of Theorem~\ref{aovsonedim} are quite well-behaved. For example, as we have used in the proof, maximal regularized rates can be computed as a simple quotient,
\beqn
\label{onedimrate}
\Rreg_{\max}(x\to y) = \frac{f(x)}{f(y)},
\eeqn
where $f$ is an arbitrary nonzero functional. This closely matches some rate formulas which have previously been derived in quantum information theory~\cite{quantumrts,generalquantum}, although these live within an epsilonified setting (Remark~\ref{epsilonification}). While these kinds of rate formulas are certainly useful, we expect that resource theories satisfying perfect asymptotic interconvertibility are still quite rare, although less so than resource theories whose ordered commutative monoids are numerical ``on the nose''. In general, one will need to use the regularized rate formula of our Theorem~\ref{rateformula} or variants of it, since~\eqref{onedimrate} only applies under the conditions of Theorem~\ref{aovsonedim}.

\begin{ex}
$\Grph$ has at least two linearly independent functionals (Examples~\ref{grphfun}). Therefore $\aovs(\Grph)$ is not numerical either.
\end{ex}

\newpage
\section{Comparison to other mathematical theories of resources}
\label{sectcompare}

Now that our framework has been presented, we discuss some of the other approaches to a general mathematical theory of resources that have been developed or are under development.

\subsection{Linear logic.}
\label{linlog}

Mathematical logic is the study of rules of inference, natural deduction, mathematical proof and provability. While the formal logical system describing conventional ``classical'' reasoning is classical sequent calculus, many other kinds of logics and formal systems have been proposed. Of most interest in our context is \emph{linear logic}~\cite{ll}, which is often described as ``resource-conscious'', or as ``if traditional logic is about truth, then linear logic is about food''. We will now try to restate our approach in logical terms and thereby relate it to linear logic.

If we rewrite the convertibility relation as entailment ``$\vdash$'' instead of ``$\geq$'', then we can express Definition~\ref{ocm} as a set of inference rules for a logic of resource convertibility. The following list of rules loosely follows the notation of~\cite{linearlogic} and blends it with ours. The structural rules of our logic are
\begin{center}
\bgroup
\def\arraystretch{3}%
\setlength\tabcolsep{1cm}
\begin{tabular}{cc}
$\infer[\textrm{(init)}]{x\vdash x}{}$
& $\infer[\textrm{(cut)}]{\Delta,\Delta'\vdash\Gamma,\Gamma'}{\Delta\vdash x,\Gamma && \Delta',x\vdash\Gamma'}$ \\
$\infer[\textrm{(left exchange)}]{\Delta,x,y,\Delta'\vdash\Gamma}{\Delta,y,x,\Delta'\vdash\Gamma}$
& $\infer[\textrm{(right exchange)}]{\Delta\vdash\Gamma,x,y,\Gamma'}{\Delta\vdash\Gamma,y,x,\Gamma'}$
\end{tabular}
\egroup
\end{center}
Here, $x$ and $y$ denote individual resource objects, while the symbols $\Delta,\Gamma,\ldots$ denote collections (lists) of resource objects of which one can consider the union by simply writing them next to each other with a comma in between. Each inference rule consists of a set of assumed convertibility relations---above the horizontal line---and one concluded convertibility relation written below the line. The first rule means that any resource object $x$ is convertible into itself. The second rule states that if $\Delta$ is convertible into $x$ plus $\Gamma$, and $\Delta'$ plus $x$ is convertible into $\Gamma'$, then also $\Delta$ plus $\Delta'$ is convertible into $\Gamma$ plus $\Gamma'$. The exchange rules logically implement the commutativity of ordered commutative monoids; these rules are not always stated explicitly in texts on linear logic.

The non-structural rules are those which concern the combination operation ``$+$'', which now becomes a logical connective, as well as the neutral element ``$0$''. These correspond to the multiplicative rules of~\cite{linearlogic}, where both tensor ``$\otimes$'' and par ``$\parr$'' are identified\footnote{This observation is due to an anonymous referee.} with our ``$+$'', and both ``$\mathbf{1}$'' and ``$\bot$'' are identified with our ``$0$''.
\begin{center}
\bgroup
\def\arraystretch{3}%
\setlength\tabcolsep{1cm}
\begin{tabular}{cc}
$\infer[(1L)]{\Delta,0\vdash\Gamma}{\Delta\vdash\Gamma}$
& $\infer[(1R)]{\vdash 0}{}$ \\
$\infer[(\otimes L)]{\Delta,x + y\vdash\Gamma}{\Delta,x,y\vdash \Gamma}$
& $\infer[(\otimes R)]{\Delta,\Delta'\vdash x+y,\Gamma,\Gamma'}{\Delta\vdash x,\Gamma & \Delta'\vdash y,\Gamma'}$ \\
$\infer[(\bot L)]{0\vdash }{}$
& $\infer[(\bot R)]{\Delta\vdash 0,\Gamma}{\Delta\vdash\Gamma}$ \\
$\infer[(\parr L)]{\Delta,\Delta',x+y\vdash \Gamma,\Gamma'}{\Delta,x\vdash \Gamma & \Delta',y\vdash\Gamma'}$
& $\infer[(\parr R)]{\Delta\vdash x + y,\Gamma}{\Delta\vdash x,y,\Gamma}$
\end{tabular}
\egroup
\end{center}
In this way, it is possible to reformulate our calculus of resources as a kind of logic. As a small fragment of linear logic with two connectives identified, this logic seems quite impoverished. Linear logic is much more expressive than this and capable of capturing other aspects of resources that ordered commutative monoids may not, such as the game-theoretic aspect of ``choices'' being made either by the resource-handling agent himself or by an adversary, which behave quite differently. Nevertheless, we hope to have shown in the main part of this paper how one can use tools from algebra and functional analysis in order to transcend the purely logical aspect of resources and derive nontrivial and relevant results about resource theories.

A slightly more expressive logic is the sequent calculus for compact closed categories~\cite{compactlinear}, which also includes negation. Due to the relation between compact closed categories and ordered abelian groups (Appendix~\ref{enrichment}), this is the logic which governs ordered abelian groups. It might be interesting to investigate what kinds of logics describe ordered $\Q$-vector spaces and Archimedean ordered $\Q$-vector spaces.

\subsection{Constructor theory.} 

As recently proposed by Deutsch~\cite{constructortheory}, constructor theory refers to the idea of developing fundamental laws of physics in terms of which transformations can be caused to happen, and which ones are impossible, together with explanations for why this is so. While Deutsch's discussion is relatively informal, a semi-rigorous version of constructor theory can be found in the work of Deutsch and Marletto on the constructor theory of information~\cite{constructorinformation}, which Marletto has applied to evolutionary theory~\cite{life}. Roughly, the idea is that a \emphalt{constructor} is an entity which can cause the transformation of one object $x$ into another object $y$,
\beqn
\label{xyconstructor}
x \xrightarrow{\textrm{constructor}} y.
\eeqn
In the language of~\cite{constructorinformation}, the term ``object'' is to be interpreted as ``substrate having certain attributes'', where a ``substrate'' refers to any object which may have attributes. So the transformation really looks like this:
\beqn
\label{substrateconstructor}
\textrm{input attributes of the substrate(s)} \xrightarrow{\textrm{constructor}} \textrm{output attributes of the substrate(s)} .
\eeqn
Here, it is assumed that the constructor takes part in the transformation, but remains unchanged in its ability to cause the transformation again. It can therefore perform $x\to y$ over and over again on new instances of $x$. All of this is closely related to our considerations, since we are also precisely concerned with questions of which transformations are possible and which ones are not. If one uses our framework as a basis for constructor theory, the constructor in~\eqref{xyconstructor} should itself be considered a resource object just like $x$ and $y$, and it enables the conversion of $x$ into $y$ as a catalyst or tool (Example~\ref{tools}). This achieves a higher degree of unification at the mathematical level. This point of view is corroborated by the statement that ``constructor theory is the ultimate generalisation of the idea of catalysis''~\cite{constructortheory}. Also the notion of \emphalt{task} of~\cite{constructortheory}, as the most important basic concept in constructor theory, can be formulated straightforwardly in the language of ordered commutative monoids: a task is a set of transformations $x\to y$ that are to be achieved; so for us, they correspond to basic orderings $x\geq y$ which generate an ordering relation on a commutative monoid.

There are more close parallels between constructor theory and our framework. For example, our Remark~\ref{epsilonification} essentially coincides with~\cite[Section~3.14]{constructortheory}. Furthermore, there is a notion of \emphalt{universal constructor}, which is supposed to be a programmable constructor capable of emulating all other programmable constructors~\cite[Section~3.8]{constructortheory}. This seems similar to our notion of generating pair (Definition~\ref{generatingpair}), which consists of two resource objects $g_+$ and $g_-$ from which any desired resource object can be extracted and into which any desired resource object can be absorbed, given a sufficient number of copies of $g_+$ and the possibility to dispose of the same number of copies of $g_-$. % The difference is that our notion of generating pair only applies to one particular resource theory at a time, while the idea of a universal constructor is presumably supposed to comprise all conceivable contexts at once. We suspect that the latter is problematic due to the abundance of paradoxes associated with self-referentiality, such as Russell's paradox or the fact that no oracle Turing machine can solve the halting problem associated to its own Turing degree. This is why we limit ourselves to considering one particular resource theory at a time.

All this suggests that our framework could be an elegant mathematical formalization of constructor theory. However, the latter also seems to be intended to have aspects which are not covered by our definitions. Concretely, the notion of \emphalt{substrate} of~\eqref{substrateconstructor} is intentionally not covered by our framework, since this turns out to be unnecessary for our purposes (those of the introduction). We have arrived at the notion of ordered commutative monoid by following the principle that a mathematical formalization should accurately capture all the relevant structures and ignore all the irrelevant ones\footnote{Compare the ``principle of least power'' in software design~\cite{leastpower}.}, and substrates belong to the latter class.

Nevertheless, it is not hard to imagine that for \emphalt{other} purposes---possibly including those of constructor theory---substrates could play an essential role. We suspect that if one wants to incorporate substrates into our formalism, one can introduce them elegantly by working with a ``relative'' analogue of ordered commutative monoids, roughly as follows. Instead of a plain ordered commutative monoid $A$, one can consider either of two things:
\begin{itemize}
\item a homomorphism $A\to B$ to a fixed base ordered commutative monoid $B$ modelling the substrates, or
\item an ordered commutative monoid $A$ which is \emphalt{indexed} by a symmetric monoidal category $B$ in the sense of a Grothendieck fibration.
\end{itemize}
The difference between these two approaches would be that the first does not distinguish between different ways of how one substrate can evolve into another, while the second retains a genuine category of substrates and maps between substrates.

Finally, the close relation between our work and constructor theory does not mean that we subscribe to Deutsch's idea that the fundamental laws of physics should be formulated in terms of ``what can and what cannot be caused to happen'', thereby renouncing the mainstream view that physics is concerned with ``what happens''. As in the introduction, we regard the framework of resource theories and ordered commutative monoids as a piece of mathematics for engineering, but we do not exclude applications to fundamental physics.

For an earlier attempt towards a framework for theories of physics based on what is possible and what is impossible, see~\cite{possibilistic}.

\subsection{General approaches to resource theories in quantum information theory.}

As explained in the introduction, the idea of resource theories originates in quantum information theory, and in particular in entanglement theory. Hence many of the concepts of this paper have previously been considered there in a less general form. While we have tried to provide references on this throughout the paper, it may help to give a separate rough outline of what has been done in quantum information theory so far. Since the focus of this paper is on the general theory, we only discuss other works which also have looked at more than just one individual resource theory or helped clarify the mathematical structure of resource theories in general.

The selection of the following references is somewhat arbitrary, and we do not intend to assess scientific priority or give an accurate outline of the history. Instead, our goal is merely to illustrate which resource-theoretic ideas have been developed in quantum information theory and provide some references for further study.

\begin{itemize}
\item A close conceptual predecessor to our approach is the idea of ``resource inequalities'' developed in Devetak, Harrow and Winter~\cite{quantumshannon}. These resource inequalities are precisely inequalities in an ordered commutative monoid, and one can consider the present work as being concerned with extending some of the ideas of~\cite{quantumshannon} from the quantum information theory context into a general framework. For example, the quantum teleportation protocol leads to the resource inequality
\beq
[qq] + 2\cdot [c\to c] \geq [q\to q],
\eeq
where $[qq]$ denotes a maximally entangled quantum state of two qubits, $[c\to c]$ stands for one bit of classical communication, and $[q\to q]$ is one qubit of quantum communication. So the teleportation protocol lets us convert a maximally entangled state plus two bits of classical communication into one qubit of quantum communication. More generally, Devetak, Harrow and Winter study two-party quantum information processing from this resource-theoretic perspective. In doing so, they develop definitions and theorems some of which are similar to ours, while others are slightly different. For example,~\cite[Section~3.3]{quantumshannon} develops a notion of \emphalt{asymptotic resource} which achieves a result similar to our many-copy and seed regularizations together. It should be investigated whether a definition of asymptotic resource along these lines makes sense in our general context, e.g.~for any ordered commutative monoid with a generating pair, and whether it could be of use for obtaining an elegant solution for the problem discussed in Remark~\ref{QvsR}. Furthermore, the resulting~\cite[Theorem 3.29]{quantumshannon} is very close to our definition of ordered commutative monoid, and~\cite[Section~4]{quantumshannon} goes along the lines of the equivalence of~\ref{ocmaovs} and~\ref{ocmcond} in our Theorem~\ref{ocmhb}. Finally, we would like to understand better how Devetak, Harrow and Winter have dealt with the problem of epsilonification (Remark~\ref{epsilonification}) in their~\cite[Section~3.3]{quantumshannon}, and whether their solution would be amenable to generalization.

The formalism of resource inequalities has been used for example in~\cite{motherprotocol,entass2,trading,apex} and further developed in~\cite{cleanresource}.
\item In quantum information theory, a resource theory is typically specified in terms of a class of \emphalt{free operations}, which are those quantum operations that are so simple and ``cheap'' to implement that their cost is considered negligible compared to all other operations. The free operations can be used to convert between quantum states or even between other (non-free) quantum operations, which defines the resource ordering. An early reference for this idea is~\cite{uniqueinfo}, although the main goal of that work is to introduce a particular resource theory: Horodecki, Horodecki and Oppenheim investigate the resource theory of quantum states with conversions generated by arbitrary unitaries, introducing new systems in totally mixed states, and discarding subsystems. Subsequently, this resource theory has become known as the theory of \emphalt{informational nonequilibrium} or \emphalt{nonuniformity theory}~\cite{nonuniformity}.

\item In~\cite{frames}, Gour and Spekkens describe the free operations paradigm in more detail and advocate its use. Specifically, they define and investigate a family of resource theories in which the free operations are given by those quantum operations that respect the given action of a symmetry group. This is relevant e.g.~for the problem of aligning reference frames between distant observers. This family of resource theories has subsequently been studied extensively in works such as~\cite{asymmetry,modes}.

\item In~\cite{quantumrts}, Horodecki and Oppenheim review resource theories defined in terms of a class of free operations. They propose a canonical additive monotone based on relative entropy for resource theories of this kind and study its properties, and they derive~\eqref{onedimrate} in their setting.
\item The free operations paradigm has been made rigorous and generalized beyond the quantum information context in~\cite{resourcesI} using category theory as a general framework for theories of processes. The main focus in that paper was the \emphalt{construction} of the ordered commutative monoids that formalize the resource theories under consideration. In the present work, we have started to develop a toolbox for working with ordered commutative monoids and mostly ignored the question of how to construct them in situations of interest. Hence~\cite{resourcesI} complements the present work.
\item In~\cite{generalthermoI,generalthermoII}, Yunger Halpern and Renes have defined and investigated resource theories in thermodynamics quite generally, where one can consider different kinds of ``baths'', like \emphalt{heat baths}, \emphalt{particle baths}, etc., and this leads to an entire family of resource theories for thermodynamics.
\item Finally, very recent work by Brand{\~a}o and Gour~\cite{generalquantum} has considered the paradigm in which resource objects are quantum states which can be converted into each other through the use of free operations as explained above. Under a certain assumption on the free operations, they derive that all states are perfectly asymptotically interconvertible. Roughly, this is a result along the lines of Theorem~\ref{aovsonedim}; the difference is that Brand{\~a}o and Gour work in an epsilonified setting (Remark~\ref{epsilonification}). However, the assumption that any operation which takes free states to free states should itself be a free operation seems so restrictive that the result is unlikely to be of use for most operationally meaningful resource theories.
\end{itemize}

\appendix

\newpage
\section{The Hahn--Banach extension theorem}
\label{appendixhb}

Our main results rely crucially on the Hahn--Banach extension theorem for sublinear maps. This is one of the most basic results in functional analysis; but since we need it for vector spaces over $\Q$ whereas the standard formulation is concerned with vector spaces over $\R$ (or $\C$), we record the statement and proof here for completeness.

\begin{thm}
\label{hbext}
Let $V$ be a $\Q$-vector space and $p:V\to\R$ a map which is sublinear in the sense that
\beq
p(x+y) \leq p(x) + p(y),\qquad p(\lambda x) = \lambda p(x)
\eeq
for all $x,y\in V$ and $\lambda\in\Qplus$. Let $U\subseteq V$ be a subspace and $f:U\to\R$ a linear map which is $p$-dominated in the sense that
\beq
f(x)\leq p(x)
\eeq
for all $x\in U$. Then there exists a linear map $\hat{f}:V\to \R$ which extends $f$ and is $p$-dominated on all of $V$.
\end{thm}

Here and in the following proof, the term ``linear'' always means ``$\Q$-linear'', and similarly for ``subspace''.

\begin{proof}
We consider pairs $(\hat{U},\hat{f})$ consisting of a subspace $\hat{U}\subseteq V$ which contains $U$ and a linear map $\hat{f}:\hat{U}\to\R$ which extends $f$ and is $p$-dominated. The set of all these pairs is naturally partially ordered via $(\hat{U},\hat{f})\leq (\hat{U}',\hat{f}')$ if and only if $\hat{U}\subseteq\hat{U}'$ and $\hat{f}'$ extends $\hat{f}$. This partially ordered set has the property that every totally ordered subset has an upper bound. Therefore by Zorn's lemma, it also has a maximal element $(\hat{U},\hat{f})$.

We claim that $\hat{U}=V$, since otherwise $(\hat{U},\hat{f})$ could not be maximal. We will see this by showing that if $\hat{U}\subsetneq V$, then we can choose $y\in V\setminus\hat{U}$ and extend $\hat{f}$ to the bigger subspace $\hat{U} + \Q y$. Such an extension is uniquely determined by linearity and its value on $y$; if we denote this to-be-determined value by $v$, then the extension takes the form
\beq
\hat{U} + \Q y \lra \R,\qquad x + \lambda y \longmapsto \hat{f}(x) + \lambda v .
\eeq
We need to make sure that $v$ can be chosen in such a way that the extension is $p$-dominated as well. Hence we need to guarantee the inequality
\beq
\hat{f}(x) + \lambda v \leq p(x + \lambda y)
\eeq
holds for all $x\in\hat{U}$ and $\lambda\in\Q$. By homogeneity under rescaling by positive numbers, it is sufficient to consider the two cases $\lambda = +1$ and $\lambda = -1$, in which case we get the two conditions
\beqn
\label{veqs}
v \leq p(x + y) - \hat{f}(x), \qquad \hat{f}(x) - p(x - y) \leq v .
\eeqn
We can find a $v$ satisfying all these inequalities if and only if all these lower bounds are less than or equal to all these upper bounds, which means
\beqn
\label{dedekindcut}
\hat{f}(x) - p(x - y) \leq p(x' + y) - \hat{f}(x').
\eeqn
This in turn follows from the assumption that $\hat{f}$ is $p$-dominated on $\hat{U}$ together with subadditivity of $p$,
\beq
\hat{f}(x + x') \leq p(x + x') \leq p(x - y) + p(x' + y) . \qedhere
\eeq
\end{proof}

\bigskip

Since this proof is precisely the conventional one, we have found that it is of no relevance that the field of coefficients is $\Q$ rather than $\R$. However, it is of paramount importance that the \emphalt{codomain} of $\hat{f}$ is allowed to be $\R$, since otherwise checking~\eqref{dedekindcut} would not be enough to guarantee the existence of a solution to~\eqref{veqs}: it may happen that there are infinitely many lower bounds and infinitely many upper bounds with only one number in between, and this number may be irrational, even if all the bounds are rational.

\newpage
\section{Ordered commutative monoids as symmetric monoidal categories}
\label{enrichment}

In the main text, we have developed a toolbox for investigating questions of resource convertibility. All of these questions were variants of one basic question: \emphalt{is it possible} to convert a resource object $x$ into a resource object $y$? While this type of question is undoubtedly important, answering it in the positive does not tell us \emphalt{how} to convert $x$ into $y$. Trying to answer it means that we should consider the collection of resource objects not as an ordered set, where $x\geq y$ represents the convertibility, but as an entire \emph{category} in which the morphisms $x\to y$ are the conversions of $x$ into $y$~\cite[Section~3]{resourcesI}. In closely related interpretations, the morphisms $x\to y$ may be the different ways in which $y$ can simulate $x$, or the different protocols available for turning $x$ into $y$.

Any category becomes an ordered set by putting $x\geq y$ for objects $x$ and $y$ if and only if there exists a morphism $x\to y$. This is a standard construction called the \emph{preorder reflection}\footnote{The convention for drawing the direction of the arrows is usually opposite, as in Example~\ref{graphs}.}. In our interpretation, the preorder reflection remembers which objects can be converted into which other ones, but it forgets how these conversions can be realized~\cite[Section~4]{resourcesI}.

By regarding an ordered set as a \emph{thin category}, which is a category in which any two parallel morphisms are equal, the \emphalt{if}-type questions become special cases of the \emphalt{how}-type ones. It turns out that certain well-known pieces of category theory can be interpreted as providing answers to the \emphalt{how}-type questions in a way which parallels and generalizes part of the toolbox developed in the main text:

\medskip
\begin{itemize}
\item Symmetric monoidal categories generalize ordered commutative monoids~\cite[Section~3]{resourcesI}.
\item Strong symmetric monoidal functors generalize homomorphisms of ordered commutative monoids.
\item Full faithfulness of functors generalizes Property~\eqref{ff} of homomorphisms to reflect the order.
\item Symmetric monoidal equivalences generalize isomorphisms of ordered commutative monoids.
\item A strong symmetric monoidal functor is an equivalence if and only if it is fully faithful and essentially surjective, and this generalizes Proposition~\ref{isocrit}.
\item Traced symmetric monoidal categories generalize cancellative ordered commutative monoids\footnote{This observation is due to David Spivak (personal communication).}.
\item Compact closed categories generalize ordered abelian groups~\cite{compactlinear}.
\item The Int construction~\cite{tmc}, which freely embeds a traced symmetric monoidal category in a compact closed category, generalizes the embedding of a cancellative ordered commutative monoid $A$ into $\oag(A)$.
\end{itemize}
\medskip

Moreover, the preorder reflection takes us from symmetric monoidal categories down to ordered commutative monoids, and similarly for traced and compact closed categories. So far, we do not know how to generalize ordered $\Q$-vector spaces or even Archimedean ordered $\Q$-vector spaces in a similar manner.

We can consider commutative monoids as ordered commutative monoids in which the ordering is symmetric. This results in analogous correspondences between unordered algebraic structures and symmetric monoidal \emphalt{dagger} categories.

\begin{ex}
Dagger compact categories generalize ordered abelian groups, and the preorder reflection assigns to every dagger compact category an abelian group. For example in topology, manifolds with boundary are morphisms in the \emphalt{cobordism category} interpolating between one part of the boundary and another~\cite[Section~1.2]{cob}. Since this is a dagger compact category, its preorder reflection yields an abelian group known as the \emphalt{cobordism group}~\cite[Section~1.2.19]{cob}.
\end{ex}

Enriched category theory~\cite{enriched} provides yet another perspective on thin categories. With $\B=\{0,1\}$ denoting the Booleans, considered as a symmetric monoidal category with one non-identity morphism $0\to 1$ and multiplication as the monoidal product, thin categories are ``the same thing as'' $\B$-categories: every thin category is automatically $\B$-enriched, and conversely every $\B$-category has an underlying category~\cite[Section~1.3]{enriched} which is thin, and these two constructions are inverses of each other. The preorder reflection can be understood as the canonical change of base from the category of sets to the Booleans. 

It is of interest to consider other base categories for the enrichment as well. A particularly relevant example should be $\Rplus$, considered as a thin category. $\Rplus$-categories are a variant of metric spaces~\cite{lawveremetrics}, and therefore developing an $\Rplus$-enriched theory analogous to the main text is one way to address Remark~\ref{epsilonification}.

\newpage
\bibliographystyle{unsrt}
\bibliography{resources_general}

\end{document}